\documentclass[11pt]{amsart}

\usepackage{epigamath}

\usepackage[english]{babel}

\numberwithin{equation}{section}

\usepackage[shortlabels]{enumitem}

\usepackage{amssymb}
\usepackage{amsmath, amscd}
\usepackage{amsthm}

\usepackage{mathtools}
\mathtoolsset{showonlyrefs=true}

\usepackage{mathrsfs}
\usepackage{booktabs}
\usepackage{perpage}
\usepackage[all,cmtip]{xy}
\usepackage{setspace}
\usepackage{eucal}
\usepackage{amsbsy}
\usepackage[T1]{fontenc}
\usepackage{verbatim}
\usepackage{mathdots}
\usepackage{blkarray}
\usepackage{multirow}
\usepackage{tikz}

\usepackage{colonequals} 

\usepackage{braket}
\usepackage{lipsum}

\usepackage{graphicx}

\newtheorem{theorem}{Theorem}[section]
\newtheorem{lemma}[theorem]{Lemma}
\newtheorem{conjecture}[theorem]{Conjecture}
\newtheorem{corollary}[theorem]{Corollary}
\newtheorem{proposition}[theorem]{Proposition}

\newtheorem{thmIntro}{Theorem}
\newtheorem{conjIntro}{Conjecture}

\theoremstyle{definition}
\newtheorem{definition}[theorem]{Definition}
\newtheorem{condition}[theorem]{Condition}
\newtheorem{assumption}[theorem]{Assumption}

\theoremstyle{remark}
\newtheorem{rmk}[theorem]{Remark}

\makeatletter
\newcommand{\sigmaop}[1]{\mathop{\mathpalette\@sigmaop{#1}}\slimits@}
\newcommand{\@sigmaop}[2]{%
  \vphantom{\sum}%
  \sbox\z@{$\m@th#1\sum$}%
  \dimen@=\ht\z@ \advance\dimen@\dp\z@
  \dimen\tw@=\wd\z@
  \ifx#1\displaystyle\dimen@=.9\dimen@\fi
  \ooalign{%
    \hidewidth
    $\vcenter{\hbox{$\m@th#1#2$\kern.3\dimen\tw@}%
     \ifx#1\scriptstyle\kern-.25ex\fi}$\hidewidth\cr
    $\vcenter{\hbox{%
      \resizebox{!}{\dimen@}{$\m@th\boxtimes$}%
    }\ifx#1\scriptstyle\kern-.25ex\fi}$\cr
  }%
}
\makeatother

\newlist{equivlist}{enumerate}{2}
\setlist[equivlist,1]{label={\rm(\roman*)}, ref={\rm\roman*}}

\newlist{assertionlist}{enumerate}{2}
\setlist[assertionlist,1]{label={\rm(\arabic*)}, ref={\rm\arabic*}}

\newlist{Alist}{enumerate}{2}
\setlist[Alist,1]{label={\rm(\Alph*)}, ref={\rm\Alph*}}

\newlist{alist}{enumerate}{2}
\setlist[alist,1]{label={\rm(\alph*)}, ref={\rm\alph*}}

\newcommand{\GZip}{\mathop{\text{$G$-{\tt Zip}}}\nolimits}

\newcommand{\pha}{\mathsf{pHa}}
\newcommand{\zipsf}{\mathsf{zip}}
\newcommand{\GSsf}{\mathsf{GS}}
\newcommand{\lwsf}{\mathsf{lw}}
\newcommand{\hwsf}{\mathsf{hw}}
\newcommand{\low}{\mathsf{low}}
\newcommand{\high}{\mathsf{high}}

\newcommand{\flag}{\mathsf{flag}}

\newcommand{\GFqrZip}{\mathop{\text{$G_{\FF_{q^r}}$-{\tt Zip}}}\nolimits}
\newcommand{\GjZip}{\mathop{\text{$G_j$-{\tt Zip}}}\nolimits}
\newcommand{\GtilZip}{\mathop{\text{$\widetilde{G}$-{\tt Zip}}}\nolimits}

\newcommand{\GF}{\mathop{\text{$G$-{\tt ZipFlag}}}\nolimits}

\newcommand{\VB}{\mathfrak{VB}}

\newcommand{\Acal}{{\mathcal A}}
\newcommand{\Bcal}{{\mathcal B}}
\newcommand{\Ccal}{{\mathcal C}}
\newcommand{\Fcal}{{\mathcal F}}
\newcommand{\Lcal}{{\mathcal L}}
\newcommand{\Ocal}{{\mathcal O}}
\newcommand{\Ucal}{{\mathcal U}}
\newcommand{\Vcal}{{\mathcal V}}
\newcommand{\Xcal}{{\mathcal X}}
\newcommand{\Zcal}{{\mathcal Z}}
\newcommand{\pfr}{{\mathfrak p}}
\newcommand{\Afr}{{\mathfrak A}}
\newcommand{\Dfr}{{\mathfrak D}}
\newcommand{\Ifr}{{\mathfrak I}}

\renewcommand{\AA}{\mathbb{A}}
\newcommand{\CC}{\mathbb{C}}
\newcommand{\FF}{\mathbb{F}}
\newcommand{\GG}{\mathbb{G}}
\newcommand{\NN}{\mathbb{N}}
\newcommand{\QQ}{\mathbb{Q}}
\newcommand{\RR}{\mathbb{R}}
\newcommand{\XX}{\mathbb{X}}
\newcommand{\ZZ}{\mathbb{Z}}

\newcommand{\Bscr}{{\mathscr B}}
\newcommand{\Gscr}{{\mathscr G}}
\newcommand{\Lscr}{{\mathscr L}}
\newcommand{\Pscr}{{\mathscr P}}
\newcommand{\Sscr}{{\mathscr S}}
\newcommand{\Tscr}{{\mathscr T}}

\newcommand{\cent}{{\rm Cent}}
\newcommand{\gm}{\mathbf G_{\textnormal{m}}}
\newcommand{\gmq}{\mathbf G_{\textnormal{m},\QQ}}
\newcommand{\fp}{\mathbf F_p}

\DeclareMathOperator{\codim}{codim}

\renewcommand{\div}{\operatorname{div}}

\DeclareMathOperator{\ad}{ad}
\DeclareMathOperator{\card}{Card}

\newcommand{\dR}{{\rm dR}}

\DeclareMathOperator{\Gal}{Gal}
\DeclareMathOperator{\gal}{Gal}
\DeclareMathOperator{\Hom}{Hom}
\DeclareMathOperator{\rank}{rank}
\DeclareMathOperator{\Span}{Span}
\DeclareMathOperator{\Stab}{Stab}
\DeclareMathOperator{\pr}{pr}
\DeclareMathOperator{\Ker}{Ker}
\DeclareMathOperator{\Rep}{Rep}
\DeclareMathOperator{\res}{Res}
\DeclareMathOperator{\Sbt}{Sbt}
\DeclareMathOperator{\Sh}{Sh}
\DeclareMathOperator{\spec}{Spec}

\DeclareMathOperator{\zip}{zip} 

\DeclareMathOperator{\GS}{GS}
\DeclareMathOperator{\SL}{SL}
\DeclareMathOperator{\GL}{GL}
\DeclareMathOperator{\GSp}{GSp}
\DeclareMathOperator{\Sp}{Sp}
\DeclareMathOperator{\U}{U}
\DeclareMathOperator{\GU}{GU}

\newcommand{\shgx}{\Sh(\mathbf G, \mathbf X)}
\newcommand{\egx}{E(\mathbf G, \mathbf X)}
\newcommand{\gx}{(\mathbf G, \mathbf X)}

\newcommand{\gofaf}{\mathbf G(\mathbf A_f)}

\newcommand{\id}{{\rm Id}}

\newcommand{\loccit}{{\em loc.\ cit. }} 
\newcommand{\loccitn}{{\em loc.\ cit.}}
\newcommand{\opcitn}{{\em op.\ cit.}}
\newcommand{\cf}{{\em cf. }} 

\DeclareMathOperator{\diag}{diag}
\DeclareMathOperator{\fil}{Fil}

\newcommand{\reg}{{\rm reg}}

\DeclareMathOperator{\GSpin}{GSpin}
\DeclareMathOperator{\Spin}{Spin}
\DeclareMathOperator{\SO}{SO}
\DeclareMathOperator{\Fil}{Fil}

\DeclareMathOperator{\Ha}{Ha}
\DeclareMathOperator{\type}{type}
\DeclareMathOperator{\Norm}{Norm}
\DeclareMathOperator{\Ind}{Ind}
\DeclareMathOperator{\cara}{char}

\DeclareMathOperator{\Res}{Res}
\DeclareMathOperator{\Flag}{Flag}
\DeclareMathOperator{\Aut}{Aut}
\DeclareMathOperator{\ev}{ev}
\DeclareMathOperator{\Weil}{Weil}

\newcommand{\relmiddle}[1]{\mathrel{}\middle#1\mathrel{}}

\newcommand{\lra}{\longrightarrow}
\def\lowsim{\vbox to 0pt{\vss\hbox{$\scriptstyle\sim$}\vskip-1.6pt}}

\newcommand{\supth}[1]{\ensuremath{#1^{\mathrm{th}}}}

\EpigaVolumeYear{10}{2026} \EpigaArticleNr{8} \ReceivedOn{June 13, 2025}
\InFinalFormOn{November 18, 2025}
\AcceptedOn{December 15, 2025}

\title{Weights of mod $\boldsymbol{p}$ automorphic forms\\ and partial Hasse invariants}
\titlemark{Weights of mod $\boldsymbol{p}$ automorphic forms and partial Hasse invariants}

\author{Naoki Imai}
\address{Graduate School of Mathematical Sciences, The University of Tokyo, 
  3-8-1 Komaba, Meguro-ku, Tokyo, 153-8914, Japan}
  \email{naoki@ms.u-tokyo.ac.jp}

\author{Jean-Stefan Koskivirta}
\address{Laboratoire de Math\'ematiques Nicolas Oresme, 
Universit\'e de Caen Normandie, 
Boulevard mar\'echal Juin, 14032 Caen, France}
\email{jean-stefan.koskivirta@unicaen.fr}

\authormark{N.~Imai and J.-S.~Koskivirta}

\AbstractInEnglish{For a connected, reductive group $G$ over a finite field endowed with a cocharacter $\mu$, we define the zip cone of $(G,\mu)$ as the cone of all possible weights of mod $p$ automorphic forms on the stack of $G$-zips. This cone is conjectured to coincide with the cone of weights of characteristic $p$ automorphic forms for Hodge-type Shimura varieties of good reduction. We prove in full generality that the cone of weights of characteristic $0$ automorphic forms is contained in the zip cone, which gives further evidence to this conjecture. Furthermore, we determine exactly when the zip cone is generated by the weights of partial Hasse invariants, which is a group-theoretic generalization of a result of Diamond--Kassaei and Goldring--Koskivirta.}



\acknowledgement{This work was supported by JSPS KAKENHI Grant Numbers 21K13765 and 22H00093. 
W.\,G.\ thanks the Knut \& Alice Wallenberg Foundation for its support under Wallenberg Academy Fellow grants KAW 2019.0256, KAW 2024.0233 and grants KAW 2018.0356, KAW 2022.0308. W.\,G.\ also thanks the Swedish Research Council for its support under grants \"AR-NT-2020-04924, \"AR-NT-2024-05526.
}

\contribution{With an appendix by Wushi Goldring}

\begin{document}



\maketitle

\begin{prelims}

\DisplayAbstractInEnglish

\bigskip

\DisplayKeyWords

\medskip

\DisplayMSCclass

\end{prelims}


\newpage

\setcounter{tocdepth}{1}

\tableofcontents


\section{Introduction}

This paper is aimed at understanding automorphic forms in characteristic $p$. They are sections of certain automorphic vector bundles over Shimura varieties. The second-named author and W.~Goldring have illustrated in several papers (\textit{e.g.}, \cite{Goldring-Koskivirta-Strata-Hasse,Goldring-Koskivirta-global-sections-compositio}) that Shimura varieties share many geometric properties with the stack of $G$-zips of Moonen--Wedhorn and Pink--Wedhorn--Ziegler; see \cite{Moonen-Wedhorn-Discrete-Invariants, Pink-Wedhorn-Ziegler-zip-data}. In this paper, we study various cones generated by weights of some classes of automorphic forms coming from this stack.

Let $(\mathbf{G},\mathbf{X})$ be a Shimura datum and $\Sh_K(\mathbf{G},\mathbf{X})$ the corresponding Shimura variety with level $K$ over a number field $\mathbf{E}$ (the reflex field). Let $\mu\colon \GG_{\mathrm{m},\CC}\to \mathbf{G}_{\CC}$ be a cocharacter attached to $\mathbf{X}$, and $\mathbf{L}\subset \mathbf{G}_\CC$ the Levi subgroup centralizing $\mu$. Choose a Borel pair $(\mathbf{B},\mathbf{T})$ such that $\mathbf{B}$ is contained in the parabolic $\mathbf{P}$ with Levi~$\mathbf{L}$ defined by $\mu$. Write $\Phi$ for the set of $\mathbf{T}$-roots and $\Phi^+$ for the positive roots (with respect to the opposite Borel $\mathbf{B}^+$). Denote by $\Delta$ the set of simple roots, and let $I\colonequals \Delta_{\mathbf{L}}$ be the simple roots of $\mathbf{L}$. To any $\mathbf{L}$-dominant character $\lambda\in X^*(\mathbf{T})$, we can attach a vector bundle $\Vcal_I(\lambda)$ (called automorphic vector bundle) on $\Sh_K(\mathbf{G},\mathbf{X})$, modeled on the $\mathbf{L}$-representation $\mathbf{V}_I(\lambda)\colonequals \Ind_{\mathbf{B}}^{\mathbf{P}}(\lambda)$ induced from $\lambda$.
When $(\mathbf{G},\mathbf{X})$ is of Hodge type and $p$ is a prime of good reduction, we have an integral model $\Sscr_K$ over $\Ocal_{\mathbf{E}_\pfr}$ (where $\pfr\mid p$) by works of Kisin and Vasiu. Furthermore, $\Vcal_I(\lambda)$ extends to a vector bundle over $\Sscr_K$ (\cf Section~\ref{subsec-Shimura} for the case of abelian type). In this paper, we are interested in the question: For which $\lambda\in X^*(T)$ does $\Vcal_I(\lambda)$ admit nonzero global sections?

Set $S_K\colonequals \Sscr_K\otimes_{\Ocal_{\mathbf{E}_\pfr}}\overline{\FF}_p$. When $F=\CC$ (resp.~$F=\overline{\FF}_p$), denote by $C_K(F)$ the cone of $\lambda\in X^*(\mathbf{T})$ such that $\Vcal_I(\lambda)$ admits nonzero sections on $\Sh_K(\mathbf{G},\mathbf{X}) \otimes_{\mathbf{E}} \CC$ (resp.~$S_K$). 
For a cone $C\subset X^*(\mathbf{T})$, define the saturation (or saturated cone) of $C$ as the set of $\lambda\in X^*(\mathbf{T})$ such that some positive multiple of $\lambda$ lies in $C$. We always denote the saturation with a calligraphic letter $\Ccal$. For example, write $\Ccal_K(F)$ for the saturation of $C_K(F)$. The set $C_K(F)$ depends on the level $K$, but one can show that the saturated cone $\Ccal_K(F)$ does not (see \cite[Corollary 1.5.3]{Koskivirta-automforms-GZip}). Therefore, we may denote it simply by $\Ccal(F)$.

We first consider the case $F=\CC$. Griffiths--Schmid introduced in \cite{Griffiths-Schmid-homogeneous-complex-manifolds} the set 
\begin{equation}
\Ccal_{\GSsf}=\left\{ \lambda\in X^{*}(\mathbf{T}) \ \relmiddle| \ 
\parbox{6cm}{
$\langle \lambda, \alpha^\vee \rangle \geq 0 \ \textrm{ for }\alpha\in I, \\
\langle \lambda, \alpha^\vee \rangle \leq 0 \ \textrm{ for }\alpha\in \Phi^+ \setminus \Phi_{\mathbf{L}}^{+}$}
\right\}.
\end{equation}
The following conjecture is expected, but we could not find a reference for it.

\begin{conjIntro}\label{conj1}
    One has $\Ccal(\CC)=\Ccal_{\GSsf}$.
\end{conjIntro}

The inclusion $\Ccal(\CC) \subset \Ccal_{\GSsf}$ is proved for general Hodge-type Shimura varieties in \cite[Theorem 2.6.4]{Goldring-Koskivirta-GS-cone}. The opposite inclusion should follow by studying the Lie algebra cohomology appearing in the cohomology of Shimura varieties.

Regarding $\Ccal(\overline{\FF}_p)$, very little is known. Diamond--Kassaei \cite{Diamond-Kassaei-comp-minimal,Diamond-Kassaei-cone-minimal} and Goldring--Koskivirta \cite{Goldring-Koskivirta-global-sections-compositio} have shown in the case of Hilbert--Blumenthal Shimura varieties that $\Ccal(\overline{\FF}_p)$ equals $\Ccal_{\pha}$, the cone generated by the weights of partial Hasse invariants on $S_K$. One goal of this paper is to discuss possible generalizations of this result to other cases. For general groups, we seek a description or an approximation of the cone $\Ccal(\overline{\FF}_p)$. Our approach uses the stack of $G$-zips of Moonen--Wedhorn and Pink--Wedhorn--Ziegler. Let $G$ be a reductive group over a finite field $\FF_q$ and $\mu\colon \GG_{\mathrm{m},k}\to G_k$ a cocharacter over $k=\overline{\FF}_q$ (in the context of Shimura varieties, we always take $q=p$). The stack of $G$-zips of type $\mu$ is denoted by $\GZip^\mu$. After possibly conjugating $\mu$, we may choose a Borel pair $(B,T)$ over $\FF_q$ such that $B$ is contained in the parabolic subgroup~$P$ defined by $\mu$ (see Section~\ref{sec-Gzips}). 
Write $L\subset G_k$ for the centralizer of $\mu$, and define $I\colonequals \Delta_L$. The vector bundles $\Vcal_I(\lambda)$ for $\lambda\in X^*(T)$ can also be defined on $\GZip^\mu$. We attach to $(G,\mu)$ a cone $C_{\zipsf}\subset X^*(T)$, defined as the set of $\lambda$ such that $\Vcal_I(\lambda)$ admits nonzero sections on $\GZip^\mu$. It is a group-theoretic version of $C_K(\overline{\FF}_p)$ and can be interpreted in terms of representation theory of reductive groups (see Section~\ref{subsec-global-sections}). When $(G,\mu)$ arises by reduction from an abelian-type Shimura datum, there is a natural smooth map $\zeta\colon S_K\to \GZip^\mu$ by \cite{Zhang-EO-Hodge} and \cite{imai-kato-youcis-prim-real}, which is known to be surjective. The map $\zeta$ induces, by pullback of sections, inclusions $C_{\zipsf}\subset C_K(\overline{\FF}_p)$ and $\Ccal_{\zipsf}\subset \Ccal(\overline{\FF}_p)$. Goldring and the second-named author have conjectured the following.  

\begin{conjIntro}[\textit{cf.} {\cite[Conjecture 2.1.6]{Goldring-Koskivirta-global-sections-compositio}}]\label{conj2}
One has $\Ccal(\overline{\FF}_p) = \Ccal_{\zipsf}$.
\end{conjIntro}

In the case of Hilbert--Blumenthal Shimura varieties, one has $\Ccal_{\zipsf} = \Ccal_{\pha}$; hence Conjecture~\ref{conj1} is compatible with the result of Diamond--Kassaei mentioned above. Aside from this case, Goldring and the second-named author showed this conjecture for Picard modular surfaces at a split prime and Siegel threefolds (see \cite[Theorem D]{Goldring-Koskivirta-global-sections-compositio}). They also treat the case of Siegel modular varieties attached to $\GSp(6)$ and unitary Shimura varieties of signature $(r,s)$ with $r+s\leq 4$ at split or inert primes (with the exception of $r=s=2$ and $p$ inert) in the paper \cite{Goldring-Koskivirta-divisibility}. 

We now describe our results more precisely. In \cite{Goldring-Koskivirta-Strata-Hasse}, the stack of $G$-zip flags, denoted by $\GF^\mu$, is defined; it is a group-theoretic analogue of the flag space of Ekedahl--van der Geer \cite{Ekedahl-Geer-EO}. There is a natural projection 
$\pi\colon \GF^\mu\to \GZip^\mu$ 
whose fibers are flag varieties isomorphic to $P/B$. The stack $\GF^\mu$ carries a family of line bundles $\Vcal_{\flag}(\lambda)$ for $\lambda\in X^*(T)$ such that $\pi_*(\Vcal_{\flag}(\lambda))=\Vcal_I(\lambda)$. In particular, we can identify $H^0(\GZip^\mu,\Vcal_I(\lambda))$ and $H^0(\GF^\mu,\Vcal_{\flag}(\lambda))$. Moreover, $\GF^\mu$ admits a stratification $(\Fcal_w)_{w\in W}$ analogous to the Bruhat decomposition, where $W=W(G,T)$ is the Weyl group of $G$. By \cite{Imai-Koskivirta-partial-Hasse}, there exists a family of partial Hasse invariants $\{h_\alpha\}_{\alpha \in \Delta}$ (where $\Delta$ is the set of simple roots). Specifically, $h_\alpha$ is a section of $\Vcal_{\flag}(\lambda_\alpha)$ (for some $\lambda_\alpha\in X^*(T)$) whose vanishing locus is the closure of a single codimension~$1$ stratum in $\GF^\mu$ (and each such stratum is cut out by exactly one of the~$h_{\alpha}$). The cone generated by the $(\lambda_\alpha)_{\alpha\in \Delta}$ is called the partial Hasse invariant cone $C_{\pha}$ (see Definition~\ref{definition-CHasse}). One has by construction $C_{\pha}\subset C_{\zipsf}$. As an analogue of \cite[Corollary 8.3]{Diamond-Kassaei-cone-minimal}, we ask whether $\Ccal_{\pha} = \Ccal_{\zipsf}$ holds in general. 
Let $w_{0,L}$ be the longest element in the Weyl group $W_L=W(L,T)$. Let $\sigma$ denote the action of Frobenius on the based root datum of $(G,B,T)$. By our assumption, the condition that $L$ (or $P$) is defined over $\FF_q$ is equivalent to $\sigma(I)=I$. We show the following. 

\begin{thmIntro}[Theorem~\ref{main-thm-Hasse-type}]\label{main-thm-Hasse-type-intro}
The following are equivalent:
\begin{equivlist}
\item\label{mtHti-1} One has $\Ccal_{\pha} = \Ccal_{\zipsf}$.
\item\label{mtHti-2} One has $\Ccal_{\GSsf}\subset \Ccal_{\pha}$.
\item
\label{item-root-data-main-thm}
The group $L$ is defined over $\FF_q$, and $\sigma$ acts on $\Delta_L$ by $-w_{0,L}$.
\end{equivlist}
\end{thmIntro}

We point out to the reader that the above result holds for an arbitrary pair $(G,\mu)$ (not merely those attached to Shimura varieties). Pairs $(G,\mu)$ with satisfying condition~\eqref{item-root-data-main-thm} are called of Hasse type. For a Shimura variety $S_K$ as above, we always have $\Ccal_{\pha}\subset \Ccal_{\zipsf}\subset \Ccal(\overline{\FF}_p)$. We deduce that a necessary condition for $\Ccal(\overline{\FF}_p)$ to be generated by partial Hasse invariants is that $(G,\mu)$ is of Hasse type. A classification of Hasse-type cases is given in the appendix by Wushi Goldring. For example, orthogonal Shimura varieties give rise to pairs $(G,\mu)$ of Hasse type (see Section~\ref{subsec-orthogonal}). Condition~\eqref{mtHti-2} has also an interpretation for Shimura varieties. One can show in general that $C_K(\CC)  \subset C_K(\overline{\FF}_p)$ (see \cite[Proposition 1.8.3]{Koskivirta-automforms-GZip}) and hence $\Ccal(\CC)\subset \Ccal(\overline{\FF}_p)$. Since it is expected that $\Ccal(\CC) = \Ccal_{\GSsf}$, condition~\eqref{mtHti-2} is necessary for $\Ccal_{\pha} = \Ccal(\overline{\FF}_p)$ to hold. From Conjectures~\ref{conj1} and~\ref{conj2}, we expect that the containment $\Ccal_{\GSsf}\subset \Ccal_{\zipsf}$ should hold in general, which is now a purely group-theoretic statement. We confirm this expectation. 

\begin{thmIntro}[Theorem~\ref{thmGSzip}]\label{thm2}
For general $(G,\mu)$, we have $\Ccal_{\GSsf}\subset \Ccal_{\zipsf}$.
\end{thmIntro}

This theorem gives further evidence for Conjecture~\ref{conj2}. In \cite[Corollary 3.5.6]{Koskivirta-automforms-GZip}, Theorem 2 was proved only when $P$ is defined over $\FF_q$. We now explain the proof of Theorem 2. The proof uses a general technique that makes it possible to reduce questions pertaining to $\Ccal_{\zipsf}$ to the case of a split group. In the split case, Theorem 2 is already known by \cite[Corollary 3.5.6]{Koskivirta-automforms-GZip}. We explain how we can reduce to the case of a split group. Denote by $L_0\subset L$ the largest algebraic subgroup defined over $\FF_q$. It is a Levi subgroup of $L$ containing $T$. There is a cocharacter $\mu_0$ with centralizer $L_0$, and we consider the pair $(G_{\FF_{q^r}},\mu_0)$, where $r\geq 1$ is such that $G_{\FF_{q^r}}$ is split. Denote by $C_{\zipsf}(G_{\FF_{q^r}},\mu_0)$ the zip cone of $(G_{\FF_{q^r}},\mu_0)$ and by $\Ccal_{\zipsf}(G_{\FF_{q^r}},\mu_0)$ its saturation. Let $w_{0,L}$ and $w_{0,L_0}$ be the longest elements in the Weyl groups of $L$ and $L_0$, respectively. Write $X_{+,L}^*(T)$ for the set of $L$-dominant characters. We show the following.

\begin{thmIntro}[Theorem~\ref{thm-GFqr-cone}]\label{thm-GFqr-cone-intro}
We have
\begin{equation}
  X_{+,L}^*(T) \cap \left(w_{0,L} w_{0,L_0} \Ccal_{\zipsf}\left(G_{\FF_{q^r}},\mu_0\right) \right) \subset \Ccal_{\zipsf}.
\end{equation}
\end{thmIntro}

This theorem is useful in general to reduce questions on $\Ccal_{\zipsf}$ to the case of a split group, as explained in Remark~\ref{rmk-reduce-split}. In particular, Theorem~\ref{thm-GFqr-cone-intro} reduces Theorem~\ref{thm2} to the case of a split group, for which it is already known. The proof of Theorem~\ref{thm-GFqr-cone-intro} relies on a closer study of the case when $G$ is a Weil restriction (see Section~\ref{sec-Weil}).

Our final result is the construction of natural mod $p$ automorphic forms attached to the highest-weight vectors of the representations $V_I(\lambda)$. Let $\lambda$ be an $L$-dominant character, and let $f_\lambda\in V_I(\lambda)$ denote the highest-weight vector of $V_I(\lambda)$. There is a natural way of defining the norm $\mathbf{f}_\lambda \colonequals \Norm_{L_\varphi}(f_\lambda)$ of $f_{\lambda}$. Here $L_\varphi$ is a certain finite (generally nonsmooth) subgroup of $L$ containing $L_0(\FF_q)$. There is an integer $m\geq 0$, determined by $L_\varphi$, such that the norm $\Norm_{L_\varphi}(f_\lambda)$ is a section of $\Vcal_I(d\lambda)$ (where $d=q^m|L_0(\FF_q)|$) over the $\mu$-ordinary locus $\Ucal_\mu$ of $\GZip^\mu$ (see Section~\ref{subsec-norm} for details). For $\alpha\in \Delta$, let $r_\alpha$ be the smallest integer $r\geq 1$ such that $\sigma^r(\alpha)=\alpha$.

\begin{thmIntro}[Proposition~\ref{prop-Norm}]\label{thm4}
The section $\mathbf{f}_\lambda$ extends to $\GZip^\mu$ if and only if for all $\alpha \in \Delta\setminus \Delta_L$ one has
\begin{equation}\label{formula-norm-intro}
\sum_{w\in W_{L_0}(\FF_q)} \sum_{i=0}^{r_\alpha-1} q^{i+\ell(w)} \ \left\langle w\lambda, \sigma^i\left(\alpha^\vee\right) \right\rangle\leq 0. \tag{1}
\end{equation}
\end{thmIntro}

Let $\Ccal_{\hwsf}$ be the set of $L$-dominant characters $\lambda$ satisfying the above inequality \eqref{formula-norm-intro}. Theorem~\ref{thm4} shows that $\Ccal_{\hwsf}\subset \Ccal_{\zipsf}$, which provides another natural subcone of $\Ccal_{\zipsf}$. We obtain a family of interesting automorphic forms $(\mathbf{f}_\lambda)_{\lambda\in \Ccal_{\hwsf}}$ in characteristic $p$ of weight $d\lambda$ (by pullback via $\zeta$). 
There is also an analogue of Theorem~\ref{thm4} for the lowest-weight vector (see Section~\ref{sec-low}), and we define the lowest-weight cone $\Ccal_{\lwsf}$ similarly. When $P$ is defined over $\FF_q$, one has $\Ccal_{\lwsf}=\Ccal_{\hwsf}$, but in general one only has  $\Ccal_{\hwsf}\subset \Ccal_{\lwsf}$. 

The motivation for introducing the family $(\mathbf{f}_\lambda)_{\lambda}$ is the following. As mentioned above, Diamond--Kassaei showed in \cite{Diamond-Kassaei-comp-minimal} that the weight of any Hilbert modular form in characteristic $p$ is spanned by the weights of partial Hasse invariants. This is also true for the Siegel-type Shimura variety $\Acal_2$, but it fails for $\Acal_n$ when $n\geq 3$. In the case $n=3$, Goldring and the second-named author showed that the weight of any automorphic form for $\Acal_3$ is spanned by the weights of partial Hasse invariants and of the forms $(\mathbf{f}_{\lambda})_{\lambda\in \Ccal_{\hwsf}}$. Therefore, these forms seem to have some significance for more general groups. Moreover, the vanishing locus of $\mathbf{f}_{\lambda}$ is an interesting subvariety stable by Hecke operators, which we plan to investigate in future papers.

We briefly explain the content of each section. In Section~\ref{sec2}, we review the stack of $G$-zips, vector bundles thereon and the connection with Shimura varieties. Section~\ref{sec-zip-cone} is dedicated to the study of the cone $C_{\zip}$, called the zip cone. We explain the motivation for introducing this set. We define several related subcones that arise naturally. We define automorphic forms on $\GZip^\mu$ attached to highest-weight vectors. In Section~\ref{sec4}, we consider pairs $(G,\mu)$ of Hasse type, and we give a complete characterization in terms of $C_{\zip}$. In Section~\ref{sec5}, similarly to the case of highest-weight vectors, we show that the lowest-weight vectors naturally give rise  to certain automorphic forms on $\GZip^\mu$. In Section~\ref{sec-Weil}, we study pairs $(G,\mu)$, where $G$ is the Weil restriction of a reductive group defined over an extension. This machinery makes it possible to reduce several questions to the case of a split group. Using this, we can check in full generality the expectation that $\Ccal_{\GS}\subset \Ccal_{\zip}$. Finally, in the last section, we illustrate the results in the case of a unitary group $U(2,1)$ and for odd orthogonal groups. In the appendix by Wushi Goldring, we give an exhaustive classification of pairs $(G,\mu)$ of Hasse type.

\subsection*{Acknowledgments}
We thank the anonymous referee for useful comments on our manuscript.

\section{Preliminaries and background on the stack of \texorpdfstring{$\boldsymbol{G}$}{G}-zips}\label{sec2}

\subsection{Notation}\label{subsec-notation}

Throughout the paper, $p$ is a prime number, $q$ is a power of $p$, and $\FF_q$ is a finite field with $q$ elements. We write $k=\overline{\FF}_q$ for an algebraic closure of $\FF_q$. The notation $G$ will always denote a connected reductive group over $\FF_q$. For a $k$-scheme $X$, we denote by $X^{(q)}$ its $\supth{q}$ power Frobenius twist and by $\varphi \colon X\to X^{(q)}$ its relative Frobenius morphism. Write $\sigma \in \Gal(k/\FF_q)$ for the $q$-power Frobenius. We will always write $(B,T)$ for a Borel pair of $G$; \textit{i.e.}, $T \subset B \subset G$ are 
a maximal torus and a Borel subgroup in $G$. We do not assume that $T$ is split over $\FF_q$. Let $B^+$ be the Borel subgroup of $G$  opposite to $B$ with respect to $T$ (\textit{i.e.}, the unique Borel subgroup $B^+$ of $G$ such that $B^+\cap B=T$). We will use the following notation: 
\begin{itemize}
\item As usual, $X^*(T)$ (resp.~$X_*(T)$) denotes the group of characters (resp.~cocharacters) of $T$. The group $\Gal(k/\FF_q)$ acts naturally on these groups. Let $W=W(G_k,T)$ be the Weyl group of $G_k$. Similarly, $\Gal(k/\FF_q)$ acts on $W$. Furthermore, the actions of $\Gal(k/\FF_q)$ and $W$ on $X^*(T)$ and $X_*(T)$ are compatible in a natural sense. We write $W(\FF_q)$ for the $\Gal(k/\FF_q)$-fixed subgroup of $W$. 
\item We write $\Phi\subset X^*(T)$ for the set of $T$-roots of $G$.
\item We write $\Phi^+\subset \Phi$ for the system of positive roots with respect to $B^+$ (\textit{i.e.}, $\alpha \in \Phi^+$ when the $\alpha$-root group $U_{\alpha}$ is contained in $B^+$). This convention may differ from that of other authors. We use it to match the conventions of previous publications \cite{Goldring-Koskivirta-Strata-Hasse,Koskivirta-automforms-GZip}.
\item We write $\Delta\subset \Phi^+$ for the set of simple roots. 
\item For $\alpha \in \Phi$, let $s_\alpha \in W$ be the corresponding reflection. The system $(W,\{s_\alpha \mid \alpha \in \Delta\})$ is a Coxeter system. 
We write $\ell  \colon W\to \NN$ for the length function. Hence $\ell(s_\alpha)=1$ for all $\alpha\in \Delta$. Let $w_0$ denote the longest element of $W$.
\item For a subset $K\subset \Delta$, let $W_K$ denote the subgroup of $W$ generated by $\{s_\alpha \mid \alpha \in K\}$. Write $w_{0,K}$ for the longest element in $W_K$.
\item Let ${}^KW$ (resp.~$W^K$) denote the subset of elements $w\in W$ that have minimal length in the coset $W_K w$ (resp.~$wW_K$). Then ${}^K W$ (resp.\ $W^K$) is a set of representatives of $W_K\backslash W$ (resp.\ $W/W_K$). The map $g\mapsto g^{-1}$ induces a bijection ${}^K W\to W^K$. The longest element in the set ${}^K W$ is $w_{0,K} w_0$.
\item Let $X_{+}^*(T)$ denote the set of dominant characters, \textit{i.e.}, characters $\lambda\in X^*(T)$ such that $\langle \lambda,\alpha^\vee \rangle \geq 0$ for all $\alpha \in \Delta$.
\item For a subset $I\subset \Delta$, let $X_{+,I}^*(T)$ denote the set of characters $\lambda\in X^*(T)$ such that $\langle \lambda,\alpha^\vee \rangle \geq 0$ for all $\alpha \in I$. We call them $I$-dominant characters.
\item Let $P\subset G_k$ be a parabolic subgroup containing $B$, and let $L\subset P$ be the unique Levi subgroup of $P$ containing $T$. Then we define a subset $I_P\subset \Delta$ as the unique subset such that $W(L,T)=W_{I_P}$. For an arbitrary parabolic subgroup $P\subset G_k$ containing $T$, we define $I_P\subset \Delta$ as $I_P \colonequals I_{P'}$, where $P'$ is the unique conjugate of $P$ containing $B$.
\item For a parabolic $P\subset G_k$, write $\Delta^P \colonequals \Delta \setminus I_P$.
\item For all $\alpha\in \Phi$, choose an isomorphism $u_\alpha\colon \GG_{\mathrm{a}}\to U_\alpha$ so that 
  $(u_{\alpha})_{\alpha \in \Phi}$ is a realization in the sense of \cite[Section~8.1.4]{Springer-Linear-Algebraic-Groups-book}.
  In particular, we have 
\begin{equation}\label{eq:phiconj}
 t u_{\alpha}(x)t^{-1}=u_{\alpha}(\alpha(t)x), \quad \forall x\in \GG_{\mathrm{a}},\  \forall t\in T.
\end{equation}
\item Let $\phi_{\alpha}\colon \SL_2\to G$ denote the map attached to $\alpha$ as in \cite[Proof of Lemma~9.2.2]{Springer-Linear-Algebraic-Groups-book}.
  It satisfies
\[
 \phi_\alpha 
 \left( \begin{pmatrix}
 1 & x \\ 0 & 1 
 \end{pmatrix}\right) = u_{\alpha}(x), \quad 
 \phi_\alpha 
 \left( \begin{pmatrix}
 1 & 0 \\ x & 1 
 \end{pmatrix}\right) = u_{-\alpha}(x).
\]
\item Fix a $B$-representation $(V,\rho)$. For $j\in \ZZ$ and $\alpha\in \Phi$, we define a map $E_{\alpha}^{(j)} \colon V \to V$ as follows. Let $V=\bigoplus_{\nu \in X^*(T)}V_\nu$ be the weight decomposition of $V$. For $v\in V_\nu$, we can write uniquely
\[
 u_{\alpha}(x)v=\sum_{j \geq 0} x^j E_{\alpha}^{(j)}(v), \quad \forall x\in \GG_{\mathrm{a}},
\]
for elements $E_{\alpha}^{(j)}(v) \in V_{\nu+j\alpha}$ (see \cite[Lemma 3.3.1]{Imai-Koskivirta-vector-bundles}). Extend $E_{\alpha}^{(j)}$ by additivity to a map $V\to V$. For $j<0$, put $E_{\alpha}^{(j)}=0$. 
\end{itemize}

\subsection{The stack of $\boldsymbol{G}$-zips}\label{sec-Gzips}

We recall some facts about the stack of $G$-zips of Pink--Wedhorn--Ziegler in \cite{Pink-Wedhorn-Ziegler-zip-data}.

\subsubsection{Definitions} \label{subsec-zipdatum}
Let $G$ be a connected reductive group over $\FF_q$. In this paper, a zip datum is a tuple $\Zcal \colonequals (G,P,L,Q,M)$ consisting of the following objects:
\begin{equivlist}
    \item $P\subset G_k$ and $Q\subset G_k$ are parabolic subgroups of $G_k$.
    \item $L\subset P$ and $M\subset Q$ are Levi subgroups such that $L^{(q)}=M$. 
\end{equivlist}
For an algebraic group $H$, denote by $R_{\mathrm{u}}(H)$ the unipotent radical of $H$. If $P'\subset G_k$ is a parabolic subgroup with Levi subgroup $L'\subset P'$, any $x\in P'$ can be written uniquely as $x=\overline{x}u$ with $\overline{x}\in L'$ and $u\in R_{\mathrm{u}}(P')$. We denote by $\theta^{P'}_{L'} \colon P'\to L'$ the map $x\mapsto \overline{x}$. Since $M=L^{(q)}$, we have a Frobenius isogeny $\varphi \colon L\to M$. Put
\begin{equation}\label{zipgroup}
E \colonequals \left\{(x,y)\in P\times Q \relmiddle|  \varphi\left(\theta^P_L(x)\right)=\theta^Q_M(y)\right\}.
\end{equation}
Equivalently, $E$ is the subgroup of $P\times Q$ generated by $R_{\mathrm{u}}(P)\times R_{\mathrm{u}}(Q)$ and elements of the form $(a,\varphi(a))$ with $a\in L$. Let $G\times G$ act on $G$ by $(a,b)\cdot g \colonequals agb^{-1}$, and let $E$ act on $G$ by restricting this action to $E$. The stack of $G$-zips of type $\Zcal$ (see \cite{Pink-Wedhorn-Ziegler-zip-data,Pink-Wedhorn-Ziegler-F-Zips-additional-structure}) can be defined as the quotient stack
\[
\GZip^\Zcal = \left[E\backslash G_k \right].
\]

\subsubsection{Cocharacter datum} \label{subsec-cochar}
A \emph{cocharacter datum} is a pair $(G,\mu)$, where $G$ is a reductive connected group over $\FF_q$ and $\mu \colon \GG_{\mathrm{m},k}\to G_k$ is a cocharacter. One can attach to $(G,\mu)$ a zip datum $\Zcal_\mu$, defined as follows. First, denote by $P_+(\mu)$ (resp.\ $P_-(\mu)$) the unique parabolic subgroup of $G_k$ such that $P_+(\mu)(k)$ (resp.\ $P_-(\mu)(k)$) consists of the elements $g\in G(k)$ satisfying that the map 
\[
\GG_{\mathrm{m},k} \lra G_{k},\quad  t\longmapsto\mu(t)g\mu(t)^{-1} \quad \left(\textrm{resp.\ } t\longmapsto\mu(t)^{-1}g\mu(t)\right)
\]
extends to a morphism of varieties $\AA_{k}^1\to G_{k}$. We obtain a pair of parabolics $(P_+(\mu),P_{-}(\mu))$ in $G_k$ whose intersection $P_+(\mu)\cap P_-(\mu)=L(\mu)$ is the centralizer of $\mu$ (it is a common Levi subgroup of $P_+(\mu)$ and $P_-(\mu)$). Set $P \colonequals P_-(\mu)$, $Q \colonequals (P_+(\mu))^{(q)}$, $L \colonequals L(\mu)$ and $M \colonequals  (L(\mu))^{(q)}$. The tuple $\Zcal_\mu \colonequals (G,P,L,Q,M)$ is a zip datum, which we call the zip datum attached to the cocharacter datum $(G,\mu)$. We write simply $\GZip^\mu$ for $\GZip^{\Zcal_\mu}$. We always consider zip data of this form.

\begin{rmk}\label{rmk-opposite-cochar}
A general zip datum $(G,P,L,Q,M)$ is of the form $\Zcal_\mu$ for a cocharacter $\mu \colon \GG_{\mathrm{m},k}\to G_k$ if and only if $\sigma(P)$ and $Q$ are opposite parabolic subgroups with common Levi $M=\sigma(L)$.
\end{rmk}

\begin{rmk}\label{rmkmuFp}
If $\mu$ is defined over $\FF_q$, then so are $P$ and $Q$. In this case, we have $L=M$, and $P$ and $Q$ are opposite parabolic subgroups with common Levi subgroup $L$.
\end{rmk}

\subsubsection{Frames} \label{sec-frames}
Let $\Zcal=(G,P,Q,L,M)$ be a zip datum. In this paper, a frame for $\Zcal$ is a triple $(B,T,z)$, where $(B,T)$ is a Borel pair of $G_k$ defined over $\FF_q$ such that 
\begin{equivlist}
    \item one has the inclusion $B\subset P$; 
    \item $z\in W$ is an element satisfying the conditions
\begin{equation}\label{eqBorel}
{}^z \! B \subset Q \quad \textrm{and} \quad
 B\cap M= {}^z \! B\cap M. 
\end{equation}
\end{equivlist}
We put $B_M\colonequals B\cap M$. Other papers (see \cite{Pink-Wedhorn-Ziegler-zip-data,Pink-Wedhorn-Ziegler-F-Zips-additional-structure, Koskivirta-Wedhorn-Hasse}) use the convention $B\subset Q$ instead of $B\subset P$. A frame (as defined here) may not always exist. However, if $(G,\mu)$ is a cocharacter datum and $\Zcal_\mu$ is the associated zip datum by Section~\ref{subsec-cochar}, then there exists a $G(k)$-conjugate $\mu'=\ad(g)\circ \mu$ (with $g\in G(k)$) such that $\Zcal_{\mu'}$ admits a frame, by Lemma~\ref{equ-IJDeltaP} below. Hence, it is harmless to assume that a frame exists, and we only consider zip data that admit frames. With respect to the Borel pair $(B,T)$, we define subsets $I,J,\Delta^P $ of~$\Delta$ as follows:
\begin{equation}\label{equ-IJDeltaP}
 I \colonequals I_P, \quad J \colonequals I_Q, \quad \Delta^P=\Delta\setminus I.  
\end{equation}

\begin{lemma}[\textit{cf.} {\cite[Lemma 2.3.4]{Goldring-Koskivirta-zip-flags}}]\label{lem-framemu}
Let $\mu \colon \GG_{\mathrm{m},k}\to G_k$ be a cocharacter, and let $\Zcal_\mu$ be the attached zip datum. Assume that $(B,T)$ is a Borel pair defined over $\FF_q$ such that $B\subset P$. Define the element
\begin{equation}\label{z-def}
z \colonequals w_0 w_{0,J}=\sigma(w_{0,I})w_0.
\end{equation}
Then $(B,T,z)$ is a frame for $\Zcal_\mu$.
\end{lemma}

\subsubsection{Parametrization of the \texorpdfstring{$\boldsymbol{E}$}{E}-orbits in \texorpdfstring{$\boldsymbol{G}$}{G}}
\label{subsec-zipstrata}
By \cite[Proposition 7.1]{Pink-Wedhorn-Ziegler-zip-data}, there are finitely many $E$-orbits in $G$. The $E$-orbits are smooth and locally closed in $G$, and the Zariski closure of an $E$-orbit is a union of $E$-orbits. We review the parametrization of $E$-orbits following \cite{Pink-Wedhorn-Ziegler-zip-data}. For $w\in W$, fix a representative $\dot{w}\in N_G(T)$ such that $(w_1w_2)^\cdot = \dot{w}_1\dot{w}_2$ whenever $\ell(w_1 w_2)=\ell(w_1)+\ell(w_2)$ (this is possible by choosing a Chevalley system; see \cite[Section~XXIII.6]{SGA3}). For $w\in W$, define $G_w$ as the $E$-orbit of $\dot{w}\dot{z}^{-1}$. If no confusion occurs, we write $w$ instead of $\dot{w}$. For $w,w'\in {}^I W$, write $w'\preccurlyeq w$ if there exists a $w_1\in W_I$ such that $w'\leq w_1 w \sigma(w_1)^{-1}$. This defines a partial order on ${}^I W$ (see \cite[Corollary 6.3]{Pink-Wedhorn-Ziegler-zip-data}).

\begin{theorem}[\textit{cf.} {\cite[Theorems 7.5,~11.2,~11.3 and~11.5]{Pink-Wedhorn-Ziegler-zip-data}}] \label{thm-E-orb-param}
We have two bijections:
\begin{align} \label{orbparam}
{}^I W &\longrightarrow \{ \textrm{$E$-orbits in $G_k$} \}, \quad w\longmapsto G_w,  \\  
\label{dualorbparam} W^J &\longrightarrow \{ \textrm{$E$-orbits in $G_k$}\}, \quad  w\longmapsto G_w. 
\end{align}
For $w\in {}^I W\cup W^J$, one has $\dim(G_w)= \ell(w)+\dim(P)$, and the Zariski closure of $G_w$ is 
\begin{equation}\label{equ-closure-rel}
\overline{G}_w=\bigsqcup_{w'\in {}^IW,\  w'\preccurlyeq w} G_{w'}
\end{equation}
for $w\in {}^I W$ and 
\begin{equation}\label{equ-closure-relJ}
\overline{G}_w=\bigsqcup_{w'\in W^J,\  w'\preccurlyeq w} G_{w'}
\end{equation}
for $w\in W^J$.
\end{theorem}
In particular, there is a unique open $E$-orbit $U_\Zcal\subset G$ corresponding to the longest elements $w_{0,I}w_0\in {}^I W$ via \eqref{orbparam} and to $w_0w_{0,J}\in W^J$ via \eqref{dualorbparam}. The $E$-orbit $U_\Zcal$ is dense in $G$.  If $\Zcal=\Zcal_\mu$ (see Section~\ref{subsec-cochar}), write $U_\mu=U_{\Zcal_\mu}$. In this case, we can choose $z=w_{0}w_{0,J}=\sigma(w_{0,I})w_0$ (see Lemma~\ref{lem-framemu}); hence \eqref{dualorbparam} shows that $1\in U_\mu$.  We put $\Ucal_\mu \colonequals [E\backslash U_\mu]$, which we call the $\mu$-ordinary locus.

\subsection{Vector bundles on the stack of $\boldsymbol{G}$} \label{sec-vector-bundles-gzipz}

\subsubsection{Representation theory}\label{subsec-remind}

For an algebraic group $G$ over a field $K$, denote by $\Rep(G)$ the category of algebraic representations of $G$ on finite-dimensional $K$-vector spaces. We denote a representation $\rho\colon G\to \GL_K(V)$ by $(V,\rho)$, or sometimes simply $\rho$ or $V$. For an algebraic group $G$ over $\FF_q$, a $G_k$-representation $(V,\rho)$ and an integer $m$, we denote by $(V^{[m]},\rho^{[m]})$ the representation such that $V^{[m]}=V$ and
\begin{equation}\label{equ-rhom}
  \rho^{[m]} \colon  G_k \xrightarrow{\varphi^m} G_k \overset{\rho}\lra \GL(V).
\end{equation}

Let $H$ be a split connected reductive $K$-group, and choose a Borel pair $(B_H,T)$ defined over $K$. If $K$ has characteristic zero, $\Rep(H)$ is semisimple. In characteristic $p$, however, this is no longer true in general. For $\lambda\in X_{+}^*(T)$, let $\Lcal_\lambda$ be the line bundle attached to $\lambda$ on the flag variety $H/B_H$ by the usual associated sheaf construction (see \cite[Section~5.8]{jantzen-representations}). Define an $H$-representation $V_H(\lambda)$ by
\begin{equation}\label{VlambdadefH}
    V_H(\lambda) \colonequals H^0\left(H/B_H,\Lcal_\lambda\right) . 
\end{equation}
In other words, one has $V_H(\lambda)=\Ind_{B_H}^{H} \lambda$. The representation $V_H(\lambda)$ is of highest weight $\lambda$. If $\cara(K)=0$, the representation $V_H(\lambda)$ is irreducible. We view elements of $V_H(\lambda)$ as regular maps $f \colon H\to \AA^1$ satisfying
\begin{equation}\label{Vlambda-function}
f(hb)=\lambda\left(b^{-1}\right)f(h), \quad \forall h\in H, \ \forall b\in B_H.    
\end{equation}
For dominant characters $\lambda,\lambda'$, there is a natural surjective map
\begin{equation}\label{Vlambda-natural-map}
V_H(\lambda)\otimes V_H(\lambda')\lra V_H(\lambda+\lambda').
\end{equation}
In the description given by \eqref{Vlambda-function}, this map is $f\otimes f'\mapsto ff'$ (for $f\in V_H(\lambda)$, $f'\in V_H(\lambda')$). Denote by $W_H \colonequals W(H,T)$ the Weyl group and by $w_{0,H}\in W_H$ the longest element. Then $V_H(\lambda)$ has a unique $B_H$-stable line, which is a weight space for the weight $w_{0,H}\lambda$.

\subsubsection{Vector bundles on quotient stacks} \label{sec-genth-vb}
For an algebraic stack $\Xcal$, write $\VB(\Xcal)$ for the category of vector bundles on $\Xcal$. Let $X$ be a $k$-scheme and $H$ an affine $k$-group scheme acting on $X$. If $\rho \colon H\to \GL(V)$ is an algebraic representation of $H$, it gives rise to a vector bundle $\Vcal_{H,X}(\rho)$ on the stack $[H\backslash X]$. This vector bundle can be defined geometrically as $[H\backslash (X\times_k V)]$, where $H$ acts diagonally on $X\times_k V$. We obtain a functor
\begin{equation}\label{funct-rep}
    \Vcal_{H,X} \colon \Rep(H)\lra \VB([H\backslash X]).
\end{equation}
Similarly to the usual associated sheaf construction, see \cite[Section~5.8, Equation (1)]{jantzen-representations}, the global sections of $\Vcal_{H,X}(\rho))$ are given by
\begin{equation}\label{globquot}
 H^0([H\backslash X],\Vcal_{H,X}(\rho))
 =\{f \colon X\to V \mid  f(h \cdot x)=\rho(h) f(x) , \  \forall h\in H, \ \forall x\in X\},
\end{equation}
where $f\colon X\to V$ is a morphism of $k$-schemes, and $V$ is viewed as an affine space over $k$.

\subsubsection{Vector bundles on \texorpdfstring{$\boldsymbol{\GZip^\mu}$}{GZip\textasciicircum\{mu\}}} \label{sec-VB-Gzip}

Fix a cocharacter datum $(G,\mu)$. Let $\Zcal=(G,P,L,Q,M)$ be the attached zip datum. Fix a frame $(B,T)$ as in Section~\ref{sec-frames}. By \eqref{funct-rep}, we have a functor $\Vcal_{E,G}\colon \Rep(E)\to \VB(\GZip^\mu)$, which we simply denote by $\Vcal$. For $(V,\rho)\in \Rep(E)$, the global sections of $\Vcal(\rho)$ are
\[
 H^0(\GZip^\mu,\Vcal(\rho))=\left\{f \colon G_k \to V \mid f(\epsilon \cdot g)=\rho(\epsilon) f(g) , \ \forall\epsilon\in E, \ \forall g\in G_k \right\}.
\]
Since $G$ admits an open dense $E$-orbit (see the discussion below Theorem~\ref{thm-E-orb-param}), the space $H^0(\GZip^\mu,\Vcal(\rho))$ is finite-dimensional (see \cite[Lemma 1.2.1]{Koskivirta-automforms-GZip}). The first projection $p_1 \colon E\to P$ induces a functor $p_1^* \colon \Rep(P)\to \Rep(E)$. If $(V,\rho)\in \Rep(P)$, we write again $\Vcal(\rho)$ for $\Vcal(p_1^*(\rho))$. In this paper, we only consider $E$-representations coming from $P$ in this way. Let $\theta^P_L \colon P\to L$ be the natural projection modulo $R_{\mathrm{u}}(P)$, as in Section~\ref{subsec-zipdatum}. It induces a fully faithful functor
\begin{equation}
(\theta^P_L)^* \colon \Rep(L)\lra \Rep(P)  
\end{equation}
whose image is the full subcategory of $\Rep(P)$ of $P$-representations trivial on $R_{\mathrm{u}}(P)$. Hence, we view $\Rep(L)$ as a full subcategory of $\Rep(P)$. If $(V,\rho)\in \Rep(L)$, write again $\Vcal(\rho) \colonequals \Vcal((\theta^P_L)^*\rho)$. For $\lambda\in X^*(T)$, write $B_L \colonequals B\cap L$, and define an $L$-representation $(V_I(\lambda),\rho_{I,\lambda})$ as follows: 
\begin{equation}\label{equ-VlambdaL}
V_I(\lambda)=\Ind_{B_L}^L \lambda, \quad \rho_{I,\lambda}\colon L\lra \GL(V_I(\lambda)). 
\end{equation}
This is the representation defined in \eqref{VlambdadefH} for $H=L$ and $B_H=B_L$. Let $\Vcal_I(\lambda)$ be the vector bundle on $\GZip^\mu$ attached to $V_I(\lambda)$, and call it an \emph{automorphic vector bundle} on $\GZip^\mu$ associated to $\lambda$. This terminology stems from Shimura varieties (see Section~\ref{subsec-Shimura} below for further details). For $\lambda \in X^*(L)$, viewing $\lambda$ as an element of $X^*(T)$ by restriction, the vector bundle $\Vcal_I(\lambda)$ is a line bundle. Note that if $\lambda\in X^*(T)$ is not $I$-dominant, then $V_I(\lambda)=0$ and thus $\Vcal_I(\lambda)=0$.

\subsection{Global sections over $\boldsymbol{\GZip^\mu}$} \label{subsec-global-sections}
We review some results of \cite{Imai-Koskivirta-vector-bundles} regarding the global sections of $\Vcal(\rho)$ for a $P$-representation $\rho$. We start with sections over the open substack $\Ucal_\mu\subset \GZip^\mu$. Recall that $\Ucal_\mu=[E\backslash U_\mu]$ and $1\in U_\mu$ (see Section~\ref{subsec-zipstrata}). By \eqref{globquot}, an element of $H^0(\Ucal_\mu,\Vcal(\rho))$ can be viewed as a map $h\colon G\to V$ satisfying $h(agb^{-1})=\rho(a)h(g)$ for all $(a,b)\in E$ and all $g\in G$. Since the $E$-orbit of $1$ is open dense in $G$, the map $h\mapsto h(1)$ is an injection
\begin{equation}\label{injection-ev1}
\ev_1 \colon H^0\left(\Ucal_\mu,\Vcal(\rho)\right)\lra V.
\end{equation}
We give the image of this map. Let $L_\varphi$ be the scheme-theoretic stabilizer subgroup of $1$ in $E$. By definition, one has
\begin{equation}\label{Lphi-equ}
L_{\varphi}=E\cap \{(x,x) \mid x\in G_k \}, 
\end{equation}
which is a $0$-dimensional algebraic group (in general nonsmooth). The first projection $E\to P$ induces a closed immersion $L_{\varphi}\to P$. Identify $L_\varphi$ with its image, and view it as a subgroup of $P$. Denote by $L_0\subset L$ the largest algebraic subgroup defined over $\FF_q$. 
In other words,
\begin{equation}\label{eqL0}
    L_0=\bigcap_{n\geq 0}L^{(q^n)}.
\end{equation}

\begin{lemma}[\textit{cf.} {\cite[Lemma 3.2.1]{Koskivirta-Wedhorn-Hasse}}]\label{lemLphi} \leavevmode
\begin{assertionlist}
\item \label{lemLphi-item1} One has $L_{\varphi}\subset L$.
\item \label{lemLphi-item2}  The group $L_{\varphi}$ can be written as a semidirect product $L_{\varphi}=L_{\varphi}^\circ\rtimes L_0(\FF_q)$, where $L_{\varphi}^\circ$ is the identity component of\, $L_{\varphi}$. Furthermore, $L_{\varphi}^\circ$ is a finite unipotent algebraic group.
\item \label{lemLphi-item3}  Assume that $P$ is defined over $\FF_q$. Then $L_0=L$ and $L_{\varphi}=L(\FF_q)$, viewed as a constant algebraic group.
\end{assertionlist}
\end{lemma}

\begin{lemma}[\textit{cf.} {\cite[Corollary 3.2.3]{Imai-Koskivirta-vector-bundles}}] \label{lem-Umu-sections}
The map \eqref{injection-ev1} induces an identification
\begin{equation}
 H^0\left(\Ucal_\mu,\Vcal(\rho)\right)=V^{L_\varphi}.   
\end{equation}
\end{lemma}

Here, the notation $V^{L_{\varphi}}$ denotes the space of scheme-theoretic invariants, \textit{i.e.}, the set of $v\in V$ such that for any $k$-algebra $R$, one has $\rho(x)v=v$ in $V\otimes_k R$ for all $x\in L_{\varphi}(R)$. We now consider the space of global sections over $\GZip^\mu$. Restriction of sections to $\Ucal_\mu\subset \GZip^\mu$ induces an injective map $H^0(\GZip^{\mu},\Vcal(\rho))\to H^0(\Ucal_\mu,\Vcal(\rho))=V^{L_\varphi}$. For simplicity, we assume here that $P$ is defined over $\FF_q$ (for the general result, see \cite[Theorem 3.4.1]{Imai-Koskivirta-vector-bundles}). We will need the general version in the proof of Proposition~\ref{prop-charL}, but in the simple setting when $\rho$ is a character $L\to \GG_{\mathrm{m}}$. For $\alpha\in \Phi$, choose a realization $(u_{\alpha})_{\alpha \in \Phi}$ (see Section~\ref{subsec-notation}). Fix a $P$-representation $(V,\rho)$, and let $V=\bigoplus_{\nu \in X^*(T)} V_\nu$ be its $T$-weight decomposition. Define the Brylinski--Kostant filtration (\cf \cite[Equation~(3.3.2)]{Xiao-Zhu-on-vector-valued}) indexed by $c\in \RR$ on $V_\nu$ by 
\begin{equation}\label{equ-Filc}
\fil_{c}^{\alpha} V_{\nu} = 
 \bigcap_{j > c} \Ker \left( 
 E_{\alpha}^{(j)} \colon V_{\nu} \to V_{\nu+j\alpha} \right), 
\end{equation}
where the map $E_{\alpha}$ was defined in Section~\ref{subsec-notation}. For $\chi \in X^*(T)_{\RR}$ and $\nu \in X^*(T)$, also set
\begin{equation}\label{equ-FilPchi}
 \fil_{\chi}^P 
 V_{\nu} = 
 \bigcap_{\alpha \in \Delta^P} 
 \fil_{\langle \chi, \alpha^\vee \rangle}^{-\alpha} 
 V_{\nu}.
\end{equation}
The Lang torsor morphism $\wp \colon T \to T$, $g\mapsto g\varphi(g)^{-1}$ induces isomorphisms 
\begin{align}
 &\wp^* \colon X^*(T)_{\RR} \stackrel{\lowsim}{\longrightarrow} X^*(T)_{\RR},\quad \lambda \longmapsto \lambda \circ \wp  = \lambda -q\sigma^{-1}(\lambda), \label{equ-Pupstar} \\
 &\wp_* \colon X_*(T)_{\RR} \stackrel{\lowsim}{\longrightarrow} X_*(T)_{\RR},\quad \delta \longmapsto \wp \circ \delta  = \delta -q\sigma(\delta).  \label{equ-Plowstar}
\end{align}

\begin{theorem}[\textit{cf.} {\cite[Corollary 3.4.2]{Imai-Koskivirta-vector-bundles}}]\label{main-Fq}
Assume that $P$ is defined over $\FF_q$. 
For all $(V,\rho)\in \Rep(P)$, the map $\ev_1$ induces an identification
\begin{equation}
H^0(\GZip^\mu,\Vcal(\rho))=V^{L(\FF_q)}\cap 
 \bigoplus_{\nu \in X^*(T)} 
 \fil_{{\wp^*}^{-1}(\nu)}^P 
 V_{\nu}.
\end{equation}
\end{theorem}

In the general case of an arbitrary parabolic $P$, $V^{L(\FF_q)}$ is replaced by $V^{L_\varphi}$ and $\Fil^\alpha_c V_\nu$ is replaced by a generalized Brylinski--Kostant filtration (see \cite[Theorem 3.4.1]{Imai-Koskivirta-vector-bundles}). In the special case when $\rho$ is trivial on $R_{\mathrm{u}}(P)$, Theorem~\ref{main-Fq} simplifies greatly. Set $\delta_{\alpha}\colonequals \wp_*^{-1}(\alpha^\vee)$, and define a subspace $V_{\geq 0}^{\Delta^P}\subset V$ by

\begin{equation}\label{equ-VDeltaP}
V_{\geq 0}^{\Delta^P} = \bigoplus_{\substack{\langle \nu,\delta_{\alpha} \rangle \geq 0, \ 
 \forall \alpha\in \Delta^P}} V_\nu.
\end{equation}
If $T$ is split over $\FF_q$, then 
$\delta_{\alpha}=-\alpha^\vee /(q-1)$, 
and $V_{\geq 0}^{\Delta^P}$ is the direct sum of the weight spaces $V_\nu$ for those $\nu\in X^*(T)$ satisfying $\langle \nu,\alpha^\vee \rangle \leq 0$ for all $\alpha \in \Delta^P$.

\begin{corollary}\label{cor-Fq-Levi}
Assume that $P$ is defined over $\FF_q$ and furthermore that $(V,\rho)\in \Rep(P)$ is trivial on $R_{\mathrm{u}}(P)$. Then one has
\begin{equation}
H^0(\GZip^\mu,\Vcal(\rho))=V^{L(\FF_q)}\cap V_{\geq 0}^{\Delta^P} .
\end{equation}
\end{corollary}

\subsection{The stack of $\boldsymbol{G}$-zip flags} \label{subsec-zipflag}

\subsubsection{Definition}
Let $(G,\mu)$ be a cocharacter datum with attached zip datum $\Zcal_{\mu}=(G,P,L,Q,M)$ (see Section~\ref{subsec-cochar}). Fix a frame $(B,T,z)$ with $z=\sigma(w_{0,I})w_0=w_0w_{0,J}$ (see Lemma~\ref{lem-framemu}). The stack of zip flags (see \cite[Definition 2.1.1]{Goldring-Koskivirta-Strata-Hasse}) is defined as
\begin{equation}\label{eq-Gzipflag-PmodB}
\GF^\mu=[E\backslash (G_k \times P/B)],     
\end{equation}
where the group $E$ acts on the variety $G_k \times (P/B)$ by the rule $(a,b)\cdot (g,hB) \colonequals (agb^{-1},ahB)$ for all $(a,b)\in E$ and all $(g,hB)\in G_k \times P/B$. The first projection $G_k \times P/B \to G_k$ is $E$-equivariant and yields a natural morphism of stacks
 \begin{equation}\label{projmap-flag}
     \pi  \colon  \GF^\mu \lra \GZip^\mu.
 \end{equation}

Set $E' \colonequals E\cap (B\times G_k)$. Then the injective map $G_k \to G_k \times P/B$, $g\mapsto (g,B)$ yields an isomorphism of stacks $[E' \backslash G_k]\simeq \GF^\mu$ (see \cite[Equation~(2.1.5)]{Goldring-Koskivirta-Strata-Hasse}). We recall the stratification of $\GF^\mu$. First, define the Schubert stack as the quotient stack
\begin{equation}\label{equ-def-Sbt}
\Sbt \colonequals [B\backslash G_k /B].
\end{equation}
This stack is finite and smooth. Its topological space is isomorphic to $W$, endowed with the topology induced by the Bruhat order on $W$. This follows easily from the Bruhat decomposition of $G$. One can show that $E'\subset B\times {}^z B$. In particular, there is a natural projection map $[E'\backslash G_k]\to [B\backslash G_k/{}^z\! B]$. Composing with the isomorphism $[B\backslash G_k/{}^z B]\to [B\backslash G_k/B]$ induced by $G_k\to G_k$, $g\mapsto gz$, we obtain a smooth, surjective map
 \begin{equation}\label{eq-GF-to-Sbt}
     \psi  \colon  \GF^\mu \lra \Sbt.
 \end{equation}
For $w\in W$, put $\Sbt_w\colonequals [B\backslash BwB /B]$; it is a locally closed substack of $\Sbt$. The \emph{flag strata} of $\GF^\mu$ are defined as fibers of the map $\psi$. They are locally closed substacks (endowed with the reduced structure). Concretely, let $w\in W$, and put
\begin{equation}
F_w \colonequals B\left(wz^{-1}\right){}^zB=BwBz^{-1},
\end{equation} 
which is the $B\times {}^zB$-orbit of $wz^{-1}$. The set $F_w$ is locally closed in $G_k$, and one has $\dim(F_w)=\ell(w)+\dim(B)$. Then, via the isomorphism $\GF^\mu\simeq [E'\backslash G_k]$, the flag strata of $\GF^\mu$ are the locally closed substacks
\begin{equation}\label{zipflag-Cw}
\Fcal_w \colonequals [E'\backslash F_w], \quad w\in W.    
\end{equation}
The set $F_{w_0}\subset G_k$ is open in $G_k$, and similarly the stratum $\Fcal_{w_0}$ is open in $\GF^\mu$. The $B\times {}^zB$-orbits of codimension $1$ are $F_{s_\alpha w_0}$ for $\alpha\in \Delta$. The Zariski closure $\overline{F}_w$ is normal (see \cite[Theorem 3]{Ramanan-Ramanathan-projective-normality}) and coincides with $\bigcup_{w'\leq w}F_{w'}$, where $\leq$ is the Bruhat order of $W$.

\subsubsection{Vector bundles on  \texorpdfstring{$\boldsymbol{\GF^\mu}$}{GF\textasciicircum mu}}
Let $\rho \colon B\to \GL(V)$ be an algebraic representation, and view~$\rho$ as a representation of $E'$ via the first projection $E'\to B$. Via the isomorphism $\GF^\mu\simeq [E'\backslash G_k]$, we obtain a vector bundle $\Vcal_{\rm flag}$ on $\GF^\mu$. Let $(V,\rho)\in \Rep(P)$, and let $\Vcal(\rho)$ be the attached vector bundle on $\GZip^\mu$. Then one has
\begin{equation}\label{pullbackV}
\pi^*(\Vcal(\rho))=\Vcal_{\flag}(\rho|_B).
\end{equation}
Note that the rank of $\Vcal_{\flag}(\rho)$ is the dimension of $\rho$. In particular, if $\lambda\in X^*(B)$, then $\Vcal_{\flag}(\lambda)$ is a line bundle. For $(V,\rho)\in \Rep(B)$, consider the $P$-representation $\Ind_B^P(\rho)$ defined by
\begin{equation}\label{eq-Indrho}
\Ind_B^P(\rho)=\left\{f\colon P\to V \relmiddle| f(xb) = \rho\left(b^{-1}\right)f(x), \ \forall b\in B,\ \forall x\in P \right\}.
\end{equation}
For $y\in P$ and $f\in \Ind_B^P(\rho)$, the element $y\cdot f$ is the function $x\mapsto f(y^{-1} x)$.

\begin{proposition}[\textit{cf.} {\cite[Proposition 3.2.1]{Imai-Koskivirta-partial-Hasse}}]\label{prop-push-flag-zip}
For $(V,\rho)\in \Rep(B)$, we have the identification 
$\pi_{*}(\Vcal_{\flag}(\rho))=\Vcal(\Ind_B^P(\rho))$. In particular $\pi_{*}(\Vcal_{\flag}(\rho))$ is a vector bundle on $\GZip^\mu$.
\end{proposition}

In particular, if $\rho$ is a character $\lambda\in X^*(T)$, then $\Vcal_{\flag}(\lambda)$ is a line bundle and one has 
\begin{equation}\label{pushforward-lambda}
    \pi_*\left(\Vcal_{\flag}(\lambda)\right)=\Vcal_I(\lambda), 
\end{equation}
where the vector bundle $\Vcal_I(\lambda)$ was defined in Section~\ref{sec-VB-Gzip}. Hence, we have 
\begin{equation}\label{ident-H0-GF}
    H^0(\GZip^\mu,\Vcal_I(\lambda))=H^0\left(\GF^\mu,\Vcal_{\flag}(\lambda)\right).
\end{equation}
If $f \colon G_k \to k$ is a section of the right-hand side of \eqref{ident-H0-GF}, 
then the corresponding function 
$f_I \colon G_k \to V_I (\lambda)$ 
on the left-hand side of 
\eqref{ident-H0-GF} is given by 
\begin{equation}\label{eq:H0expid}
 (f_I (g))(x)=f \left(\left(x^{-1},\varphi(x)^{-1}\right) \cdot g\right)= 
 f\left(x^{-1} g \varphi (x)\right) 
\end{equation}
for all $g \in G_k$ and $x \in L$, by the construction of the identification. 
Also note  that the line bundles $\Vcal_{\flag}(\lambda)$ satisfy the following identity:
\begin{equation}\label{Vflag-additivity}
\Vcal_{\flag}(\lambda+\lambda')=\Vcal_{\flag}(\lambda)\otimes \Vcal_{\flag}(\lambda'), \quad \forall \lambda, \lambda'\in X^*(T).
\end{equation}
We can also define vector bundles on the stack $\Sbt$ as in \cite[Section~4]{Imai-Koskivirta-partial-Hasse}. For our purpose, it is enough to define line bundles on $\Sbt$. Using \eqref{funct-rep}, we can attach to each $(\chi_1,\chi_2)\in X^*(T)\times X^*(T)$ a line bundle $\Vcal_{\Sbt}(\chi_1,\chi_2)$ on $\Sbt$. One has
\begin{equation}\label{pullbackformula}
\psi^*\Vcal_{\Sbt}(\chi_1,\chi_2)=\Vcal_{\flag}(\chi_1+(z\chi_2)\circ \varphi)=\Vcal_{\flag}\left(\chi_1+q\sigma^{-1}(z\chi_2)\right).
\end{equation}

\subsubsection{Partial Hasse invariants}\label{sec-partialHasse}
We recall some results of \cite{Imai-Koskivirta-partial-Hasse}. By \cite[Theorem 2.2.1(a)]{Goldring-Koskivirta-Strata-Hasse}, the line bundle $\Vcal_{\Sbt}(\chi_1,\chi_2)$ admits a nonzero section over $\Sbt_{w_0}$ if and only if $\chi_1=-w_0\chi_2$. If this condition is satisfied, $H^0(\Sbt_{w_0}, \Vcal_{\Sbt}(\chi_1,\chi_2))$ is $1$-dimensional. For $\chi\in X^*(T)$, let $h_\chi$ be any nonzero element
\begin{equation}\label{hchi}
  h_\chi\in H^0\left(\Sbt_{w_0}, \Vcal_{\Sbt}(-w_0\chi,\chi)\right).
\end{equation}
By \cite[Theorem 2.2.1(c)]{Goldring-Koskivirta-Strata-Hasse}, $h_\chi$ extends to $\Sbt$ if and only if $\chi$ is a dominant character. Using \eqref{pullbackformula} and $z=\sigma(w_{0,I})w_0$, we obtain a section
\begin{equation}
\Ha_{\chi}\colonequals \psi^*(h_\chi)\in H^0\left(\Fcal_{w_0},\Vcal_{\flag}\left(-w_0\chi+qw_{0,I}w_0\left(\sigma^{-1}\chi\right)\right)\right), 
\end{equation}
and for $\chi\in X^*_+(T)$ the section $\Ha_{\chi}$ extends to $\GF^\mu$. In particular, let $\alpha\in \Delta$, and suppose $\chi_{\alpha}$ is a character satisfying 
\begin{equation}\label{chia-cond}
\begin{cases}
\left\langle \chi_\alpha, \alpha^{\vee} \right\rangle>0, & \\
\left\langle \chi_\alpha, \beta^{\vee} \right\rangle=0 & \ \textrm{ for all } \beta\in \Delta\setminus \{\alpha\}.
\end{cases}
\end{equation}
In this case, the section $h_{\chi_\alpha}$ vanishes exactly on the codimension~$1$ stratum $\overline{\Sbt}_{w_0s_\alpha}$. Similarly, the section $\Ha_{\chi_{\alpha}}$ cuts out the Zariski closure of the codimension~$1$ stratum $\Fcal_{w_0s_\alpha}$.

\begin{definition}\label{partial-Hasse-def}
For $\alpha\in \Delta$ and $\chi_\alpha$ satisfying \eqref{chia-cond}, we call the section $\Ha_{\chi_{\alpha}}$ a partial Hasse invariant for the stratum $\Fcal_{w_0 s_\alpha}$.
\end{definition}

\subsection{Shimura varieties and $\boldsymbol{G}$-zips} \label{subsec-Shimura}
We explain the connection between the stack of $G$-zips and Shimura varieties. Let $\gx$ be a Shimura datum; see \cite[Section~2.1.1]{Deligne-Shimura-varieties}. 
Write $\mathbf{E}=\egx$ for the reflex field of $\gx$ and $\Ocal_\mathbf{E}$ for its ring of integers. 
Given an open compact subgroup $K \subset \gofaf$, write $\shgx_{K}$ for Deligne's canonical model at level $K$ over $\mathbf{E}$ (see \cite{Deligne-Shimura-varieties}). For $K\subset \mathbf{G}(\AA_f)$ small enough, $\shgx_K$ is a smooth, quasi-projective scheme over $\mathbf{E}$. Every inclusion $K' \subset K$ induces a finite \'{e}tale projection $\pi_{K'/K} \colon \shgx_{K'} \to \shgx_{K}$.

Fix a prime number $p$, and assume that the level $K$ is of the form $K=K_pK^p$,  where $K_p\subset \mathbf{G}(\QQ_p)$ is a hyperspecial subgroup and $K^p\subset \mathbf{G}(\AA_f^p)$ is an open compact subgroup. Then one has $K_p=\Gscr(\ZZ_p)$, where $\Gscr$ is a reductive group over $\ZZ_p$ such that $\Gscr\otimes_{\ZZ_p}\QQ_p\simeq \mathbf{G}_{\QQ_p}$ and $\Gscr\otimes_{\ZZ_p}\mathbb{F}_p$ is connected. 

We assume that $\gx$ is of abelian type.
For any place $v$ above $p$ in $\mathbf{E}$, Kisin \cite{Kisin-Hodge-Type-Shimura} and Vasiu \cite{Vasiu-Preabelian-integral-canonical-models} constructed a smooth, canonical model $\Sscr_K$  of $\shgx_K$ over $\Ocal_{\mathbf{E}_v}$. Let $\kappa(v)$ denote the residue field of $\Ocal_{\mathbf{E}_v}$, and let $\overline{\FF}_p$ be an algebraic closure of $\kappa(v)$. Set $S_K\colonequals \Sscr_K\otimes_{\Ocal_{\mathbf{E}_v}} \overline{\FF}_p$. We can take a representative $\mu\in \{\mu\}$ defined over $\mathbf{E}_v$, by \cite[Lemma~(1.1.3)(a)]{Kottwitz-Shimura-twisted-orbital}. 
We can also assume that $\mu$ extends to $\mu \colon \GG_{\mathrm{m},\Ocal_{\mathbf{E}_v}}\to \mathscr{G}_{\Ocal_{\mathbf{E}_v}}$ (see \cite[Corollary 3.3.11]{Kim-Rapoport-Zink-uniformization}). It gives rise to a parabolic subgroup $\Pscr\subset \mathscr{G}_{\Ocal_{\mathbf{E}_v}}$, with Levi subgroup $\Lscr$ equal to the centralizer of $\mu$. As explained in \cite[Section~2.5]{Imai-Koskivirta-vector-bundles}, we can assume (after possibly twisting $\mu$) that there is a Borel pair $(\Bscr, \Tscr)$ over $\ZZ_p$ in $\Gscr$ such that $\mathscr{B}_{\Ocal_{\mathbf{E}_v}}\subset \Pscr$. Let $G,P,B,T$ denote the special fibers of $\Gscr, \Pscr, \Bscr, \Tscr$, respectively. By slight abuse of notation, we denote again by $\mu$ its mod $p$ reduction $\mu \colon \GG_{\mathrm{m},k}\to G_k$. 

We define a quotient $\mathbf{G}^c$ of $\mathbf{G}$ by a subtorus of the center of $\mathbf{G}$ as in \cite[Section~2.3]{imai-kato-youcis-prim-real}. We note that $\mathbf{G}^c=\mathbf{G}$ if $\gx$ is of Hodge type by \cite[Remark 2.6]{imai-kato-youcis-prim-real}.  Let $G^c$ be the quotients of $G$ determined by $\mathbf{G}^c$.  Let $\mu^c \colon \GG_{\mathrm{m},k}\to G^c_k$ be the cocharacter induced by $\mu$.  Then $(G^c,\mu^c)$ is a cocharacter datum, and it yields a zip datum as in Section~\ref{subsec-cochar}, where $q=p$.  By \cite[Section~4.1]{Zhang-EO-Hodge} and \cite[Section~3.5]{imai-kato-youcis-prim-real}, there exists a natural smooth morphism
\begin{equation}\label{zeta-Shimura}
\zeta \colon S_K\lra \mathop{\text{$G^c$-{\tt Zip}}}\nolimits^{\mu^c}. 
\end{equation}
This map is also surjective by \cite[Corollary 3.5.3(1)]{Shen-Yu-Zhang-EKOR}. 

Let $T^c$ be the maximal torus of $G^c$ determined by $T$. 
For $\lambda\in X^*_{+,I}(T^c)$, we have a vector bundle $\Vcal_{I}(\lambda)$ on $\mathop{\text{$G^c$-{\tt Zip}}}\nolimits^{\mu^c}$ as in Section~\ref{sec-VB-Gzip}. We denote the pullback of $\Vcal_{I}(\lambda)$ under $\zeta$ by the same symbol. 
The vector bundle $\Vcal_{I}(\lambda)$ on $S_K$ is called \emph{the automorphic vector bundle of weight $\lambda$}.

The flag space of the Siegel modular variety $\Acal_n$ was first introduced by Ekedahl--van der Geer in \cite{Ekedahl-Geer-EO}. It parametrizes pairs $(\underline{A},\Fcal_{\bullet})$, where $\underline{A}=(A,\lambda)\in \Acal_n$ is a principally polarized abelian variety of relative dimension $n$ and $\Fcal_\bullet \subset H^1_{\dR}(A)$ is a full symplectic flag refining the Hodge filtration of $H^1_{\dR}(A)$. In general,  the flag space $\Flag(S_K)$ of $S_K$ was defined in \cite[Section~9.1]{Goldring-Koskivirta-Strata-Hasse} as the fiber product
\begin{equation}
\xymatrix@1@M=7pt{
\Flag(S_K) \ar[d]_{\pi_K} \ar[r]^-{\zeta_{\flag}} & \GF^{\mu} \ar[d]^{\pi} \\
S_K \ar[r]_-{\zeta} & \GZip^\mu\rlap{.}
}
\end{equation}
The stratification $(\Fcal_w)_{w\in W}$ on $\GF^\mu$ induces by pullback via $\zeta_{\flag}$ a stratification $(\Flag(S_K)_w)_{w\in W}$ of $\Flag(S_K)$ by locally closed, smooth subschemes. Moreover, we obtain a line bundle $\Vcal_{\flag}(\lambda)$ on $\Flag(S_K)$. Since $\zeta$ is smooth, pullback and pushforward commute, so we have $\pi_{K,*}(\Vcal_{\flag}(\lambda))=\Vcal_I(\lambda)$. In particular, the space of automorphic forms $H^0(S_K,\Vcal_I(\lambda))$ can be identified with $H^0(\Flag(S_K),\Vcal_{\flag}(\lambda))$.

\section{The zip cone}\label{sec-zip-cone}

In this section, we will consider several subsets of $X^*(T)$. A \emph{cone} in $X^*(T)$ will be an additive submonoid containing $0$. If $C\subset X^*(T)$ is a cone, we define its saturation (or saturated cone) as follows: 
\begin{equation}
\Ccal = \{\lambda \in X^*(T) \mid \exists N\geq 1, \ N\lambda\in C \}.
\end{equation}
We say that $C$ is saturated if $C=\Ccal$. Also define  $C_{\QQ_{\geq 0}}$ as follows: 
\begin{equation}
C_{\QQ_{\geq 0}}=\left\{\sum_{i=1}^N a_i\lambda_i \relmiddle| N\geq 1, \ a_i\in \QQ_{\geq 0}, \ \lambda_i\in C \right\}.
\end{equation}
There is a bijection between saturated cones of $X^*(T)$ and additive submonoids of $X^*(T)\otimes_{\ZZ} \QQ$ stable by $\QQ_{\geq 0}$. The bijection is given by the maps $C\mapsto C_{\QQ_{\geq 0}}$ and $C'\mapsto C'\cap X^*(T)$.

\subsection{Example: Hilbert--Blumenthal Shimura varieties}
We recall some results of Diamond--Kassaei in \cite{Diamond-Kassaei-comp-minimal} and extended in \cite{Diamond-Kassaei-cone-minimal} that motivate this paper. We give a short explanation of \cite[Corollary 5.4]{Diamond-Kassaei-comp-minimal}. The authors study Hilbert automorphic forms in characteristic $p$. Specifically, let $\mathbf{F}/\QQ$ be a totally real extension of degree $d=[\mathbf{F}:\QQ]$, and let $\mathbf{G}$ be the subgroup of $\Res_{\mathbf{F}/\QQ}(\GL_{2,\mathbf{F}})$ defined by
\begin{equation}
    \mathbf{G}(R)=\left\{g\in \GL_2(R\otimes_\QQ \mathbf{F}) \relmiddle| \det(g)\in R^\times\right\}.
\end{equation}
Let $p$ be a prime number unramified in $\mathbf{F}$ (in \cite{Diamond-Kassaei-cone-minimal}, $p$ is allowed to be ramified in $\mathbf{F}$). The lattice $\ZZ_p\otimes_{\ZZ}\Ocal_{\mathbf{F}}\subset \QQ_p\otimes_{\QQ} \mathbf{F}$ gives rise to a reductive model $\Gscr$ over $\ZZ_p$. Fix a small enough level $K^p\subset \mathbf{G}(\AA_f^p)$ outside $p$, and set $K_p\colonequals\Gscr(\ZZ_p)$ and $K=K_pK^p$. Let $S_K$ be the (geometric) special fiber of the corresponding Hilbert--Blumenthal Shimura variety of level $K$. 
The scheme $S_K$ is smooth of dimension $d$ over $\overline{\FF}_p$. It parametrizes tuples $(A,\lambda,\iota,\overline{\eta})$ of abelian schemes over $\overline{\FF}_p$ of dimension $d$ endowed with a principal polarization $\lambda$, an action $\iota$ of $\Ocal_{\mathbf{F}}$ on $A$ and a $K^p$-level structure $\overline{\eta}$. 

Let $\Sigma\colonequals \Hom(\mathbf{F},\overline{\QQ}_p)$ be the set of field embeddings $\mathbf{F}\to \overline{\QQ}_p$. Write $(\mathbf{e}_\tau)_\tau$ for the canonical basis of $\ZZ^\Sigma$. Let $\sigma$ denote the action of Frobenius on $\Sigma$. For each $\tau\in \Sigma$, there is an associated line bundle $\omega_\tau$ on $S_K$. For $\mathbf{k}=\sum_\tau k_\tau \mathbf{e}_\tau\in \ZZ^\Sigma$, define
\begin{equation}
\omega^\mathbf{k}\colonequals \bigotimes_{\tau\in \Sigma}\omega_{\tau}^{k_\tau}.    
\end{equation}
Elements of $H^0(X_{\overline{\FF}_p},\omega^\mathbf{k})$ are called mod $p$ Hilbert modular forms of weight $\mathbf{k}$. There is an Ekedahl--Oort stratification on $S_K$ given by the isomorphism class of the $p$-torsion $A[p]$ (with its additional structure given by $\lambda$ and $\iota$). There is a unique open stratum (on which $A$ is an ordinary abelian variety). The codimension~$1$ strata can be labeled as $(S_{K,\tau})_{\tau\in \Sigma}$. Andreatta--Goren \cite{Andreatta-Goren-book} constructed partial Hasse invariants $\Ha_{\tau}$ for each $\tau\in \Sigma$. The weight of $\Ha_{\tau}$ is given by
\begin{equation}
    \mathbf{h}_\tau \colonequals e_{\tau}-pe_{\sigma^{-1} \tau}.
\end{equation}
Note that the sign of $\mathbf{h}_\tau$ is different in \cite{Andreatta-Goren-book} and \cite{Diamond-Kassaei-comp-minimal}, due to a different positivity convention. The main property of $\Ha_{\tau}$ is that it vanishes exactly on the Zariski closure of the codimension~$1$ stratum $S_{K,\tau}$. It is a special case of the sections $\Ha_{\alpha}$ defined in Definition~\ref{partial-Hasse-def}. Define the partial Hasse invariant cone $\Ccal_{\pha}\subset \ZZ^\Sigma$ as the cone of $\mathbf{k}\in \ZZ^{\Sigma}$ that are spanned (over $\QQ_{\geq 0}$) by the weights $(\mathbf{h}_\tau)_{\tau\in \Sigma}$ defined above.

\begin{theorem}[Diamond--Kassaei, \textit{cf.} {\cite[Theorem 5.1 and Corollary 5.4]{Diamond-Kassaei-comp-minimal}}] \label{thmDK} \ 
\begin{assertionlist}
    \item\label{thm-DK-item1} Let $f\in H^0(S_K,\omega^\mathbf{k})$ and $\tau \in \Sigma$. Assume that $pk_\tau > k_{\sigma^{-1}\tau}$. Then $f$ is divisible by $\Ha_{\tau}$.
    \item\label{thm-DK-item2} If\, $H^0(S_K,\omega^\mathbf{k})\neq 0$, then $\mathbf{k}\in \Ccal_{\pha}$.
\end{assertionlist}
\end{theorem}
The authors define a minimal cone $\Ccal_{\rm min}\subset \Ccal_{\pha}$ as follows:
\begin{equation}
    \Ccal_{\rm min}=\left\{\mathbf{k}\in \ZZ^\Sigma \relmiddle| pk_\tau \leq k_{\sigma^{-1}\tau} \ \textrm{ for all }\tau\in \Sigma\right\}.
\end{equation}
Theorem~\ref{thmDK}\eqref{thm-DK-item1} shows that any Hilbert modular form $f$ of weight $\mathbf{k}$ can be written as a product $f=f_{\min} h$, where $f_{\min}$ has weight $\mathbf{k}_{\min}\in \Ccal_{\min}$ and $h$ is a product of partial Hasse invariants. In particular, \eqref{thm-DK-item2} is a direct consequence of \eqref{thm-DK-item1}. One motivation of this paper is to understand the natural setting in which one might expect a generalization of Theorem~\ref{thmDK}\eqref{thm-DK-item2} to other Shimura varieties. In \cite{Goldring-Koskivirta-divisibility}, Goldring and the second-named author show that \eqref{thm-DK-item1} also admits a similar generalization for several Hodge-type Shimura varieties.

\subsection{General setting}\label{sec-gen-set}
We attempt to give an abstract setting in which Theorem~\ref{thmDK} may generalize. First, by observing the example of Hilbert--Blumenthal varieties, we extract the essential properties of the objects we want to study. Specifically, we consider a stack $Y$ over $k=\overline{\FF}_p$ endowed with the following structure:
\begin{alist}
\item\label{sgs-a} There is a locally closed stratification $Y=\bigsqcup_{i=1}^N Y_i$ such that the Zariski closure of a stratum is a union of strata.
\item\label{sgs-b} There exist a free, finite-type $\ZZ$-module $\Lambda$ and a family of line bundles $(\omega(\lambda))_{\lambda\in \Lambda}$ on $Y$ such that $\omega(\lambda+\lambda')=\omega(\lambda)\otimes \omega(\lambda')$ for all $\lambda,\lambda'\in \Lambda$.
\item\label{sgs-c} For each codimension~$1$ stratum $Y_i\subset Y$, there are fixed $\lambda_i\in \Lambda$ and $\Ha_i\in H^0(X,\omega(\lambda_i))$ such that the support of $\div(\Ha_i)$ is $\overline{Y}_i$. By analogy, we call $\Ha_i$ a partial Hasse invariant for $Y_i$.
\end{alist}
Denote by $I_1\subset I$ the indices such that $Y_i$ has codimension~$1$. Let $C_{\pha}\subset \Lambda$ denote the cone generated by the elements $\{\lambda_i \mid i\in I_1\}$, and call it the partial Hasse invariant cone. Put
\begin{equation}\label{CY-cone}
    C_Y\colonequals \{\lambda\in \Lambda \mid H^0(Y,\omega(\lambda))\neq 0\}.
\end{equation}
By definition, one has $C_{\pha}\subset C_Y$. If $Y$ is integral, then $C_Y$ is a cone (\textit{i.e.}, an additive submonoid) of $\Lambda$. Indeed, if $\lambda,\lambda'\in \Lambda$ and $f,f'$ are nonzero sections of $\omega(\lambda)$ and $\omega(\lambda')$, respectively, then $ff'$ is a section of $\omega(\lambda+\lambda')$. Since $Y$ is integral, $ff'$ is nonzero. Write $\Ccal_{\pha}$ and $\Ccal_Y$ for the saturations of $C_{\pha}$ and $C_Y$, respectively, inside $\Lambda$.

\begin{definition}\label{def-Hasse-property}
Let $Y$ be a stack satisfying~\eqref{sgs-a},~\eqref{sgs-b} and~\eqref{sgs-c}. We say that $Y$ has the Hasse property if $\Ccal_{\pha} = \Ccal_Y $.
\end{definition}

For example, Theorem~\ref{thmDK}\eqref{thm-DK-item2} shows that the geometric special fiber of the Hilbert--Blumenthal Shimura variety at a place of good reduction satisfies the Hasse property. Let $(G,\mu)$ be a cocharacter datum, and let $\GZip^\mu$ be the attached stack of $G$-zips. Fix a frame $(B,T,z)$ as in Section~\ref{sec-frames}. 
Then, the stack of zip flags $\GF^\mu$ (see Section~\ref{subsec-zipflag}) satisfies all requirements~\eqref{sgs-a},~\eqref{sgs-b} and~\eqref{sgs-c} above. First, we have the flag stratification $\GF^{\mu}=\bigsqcup_{w\in W}\Fcal_w$ as in Section~\eqref{zipflag-Cw}. Setting $\Lambda \colonequals X^*(T)$, we have the family of line bundles $(\Vcal_{\flag}(\lambda))_{\lambda\in X^*(T)}$ satisfying (b) by \eqref{Vflag-additivity}. Finally, we have partial Hasse invariants (Section~\ref{sec-partialHasse}). 
To be precise, there is an ambiguity in the definition of $C_{\pha}$, because if $f$ is a partial Hasse invariant for $\Fcal_{w_0 s_\alpha}$ (see Definition~\ref{partial-Hasse-def}), then $\chi f^n$ for $\chi \in X^*(G)$ and $n\geq 1$ is also a partial Hasse invariant for $\Fcal_{w_0 s_\alpha}$. We will give an unambiguous definition of $C_{\pha}$ in Definition~\ref{definition-CHasse}. Moreover, the saturation $\Ccal_{\pha}$ is independent of all choices. In this paper, we give a full answer as to whether $\GF^{\mu}$ satisfies the Hasse property.

Similarly, let $Y$ be a scheme endowed with a smooth, surjective morphism $\zeta_{Y} \colon Y\to \GF^{\mu}$. Then $Y$ naturally inherits  by pullback all the structure from $\GF^{\mu}$, and hence satisfies all required properties~\eqref{sgs-a},~\eqref{sgs-b} and~\eqref{sgs-c}  above. In particular, if we start with a scheme $X$ and a smooth, surjective morphism $\zeta \colon X\to \GZip^\mu$, then we can consider the flag space $Y\colonequals\Flag(X)$ (similarly to the flag space of $S_K$ defined at the end of Section~\ref{subsec-Shimura}). It is defined as the fiber product
\begin{equation}\label{equ-Y}
    Y \colonequals X\times_{\GZip^\mu} \GF^{\mu}.
\end{equation}
The induced map $\zeta_{\flag} \colon Y\to \GF^{\mu}$ is again smooth and surjective. Hence, $Y$ inherits the structure as above and satisfies~\eqref{sgs-a},~\eqref{sgs-b} and~\eqref{sgs-c}. Denote by $\pi \colon Y\to X$ and $\pi \colon \GF^{\mu}\to \GZip^\mu$ the natural projections. In both cases, we have $\pi_{*}(\Vcal_{\flag}(\lambda))=\Vcal_I(\lambda)$ because $\zeta$ is smooth and $\pi$ is proper. Therefore, the cones $C_Y$ and $C_{\GF^{\mu}}$ can also be written as follows:
\begin{align}
    C_Y&=\left\{\lambda\in X^*(T) \relmiddle| H^0(X,\Vcal_I(\lambda))\neq 0\right\}, \label{equ-CY-Vlambda}  \\
    C_{\zipsf} \colonequals C_{\GF^{\mu}} &= \left\{\lambda\in X^*(T)\relmiddle| H^0(\GZip^\mu,\Vcal_I(\lambda))\neq 0\right\}. \label{equ-CZip}
\end{align}
We will use the notation $C_{\zipsf}$ (introduced in \cite{Koskivirta-automforms-GZip}) instead of $C_{\GF^{\mu}}$. When $Y$ is given as \eqref{equ-Y} above, we call $Y$ \emph{the flag space} of $(X,\zeta)$. Furthermore, we make the slight abuse of saying that $(X,\zeta)$ satisfies the Hasse property if $(Y,\zeta_{\flag})$ does. In particular, let $X=S_K$ be the geometric special fiber modulo $p$ of a Hodge-type Shimura variety with good reduction at $p$. By Zhang's result, there is a smooth, surjective morphism $\zeta \colon X\to \GZip^\mu$, and so we obtain $(Y,\zeta_{\flag})$ as above. Our goal is to investigate which Hodge-type Shimura varieties satisfy the Hasse property.

We now return to a general pair $(X,\zeta)$. Since $\zeta$ is surjective, pullback by $\zeta_{\flag}$ induces an inclusion
\begin{equation}
    H^0(\GZip^\mu,\Vcal_I(\lambda)) \subset H^0(X,\Vcal_I(\lambda)).
\end{equation}
Hence we have inclusions $C_{\zipsf}\subset C_Y$ and $\Ccal_{\zipsf}\subset \Ccal_Y$. Furthermore, Hasse invariants already exist on the stack $\GF^{\mu}$ by Section~\ref{sec-partialHasse}; hence the cone $C_{\pha}$ generated by their weights satisfies $C_{\pha}\subset C_{\zipsf}$. Therefore, we have in general
\begin{equation}
\Ccal_{\pha}\subset \Ccal_{\zipsf}\subset \Ccal_Y.    
\end{equation}
In particular, if the pair $(X,\zeta)$ satisfies the Hasse property, then all three cones above coincide. In other words, a necessary condition for $X$ to satisfy the Hasse property is that $\GZip^{\mu}$ itself satisfies this property, which is equivalent to the condition $\Ccal_{\zipsf}=\Ccal_{\pha}$. This is an obstruction for a potential generalization of Theorem~\ref{thmDK}\eqref{thm-DK-item2} to other Shimura varieties.

\begin{rmk}
When we start with a pair $(X,\zeta)$ and construct $(Y,\zeta_{\flag})$ by fiber product as in \eqref{equ-Y}, Formula \eqref{equ-CY-Vlambda} shows  immediately that 
\begin{equation}\label{equ-Ldom-inclusion}
\Ccal_Y\subset X^*_{+,I}(T).
\end{equation}
Indeed, this follows simply from the fact that if $\lambda$ is not $I$-dominant, then $\Vcal_I(\lambda)=0$. Thus, in the example of Shimura varieties, we have the inclusion \eqref{equ-Ldom-inclusion}.
\end{rmk}

\subsection{Previous results}
We review previous results from \cite{Goldring-Koskivirta-global-sections-compositio}. Let $(X,\zeta)$ be a pair consisting of a $k$-scheme $X$ and a smooth, surjective morphism of stacks $\zeta\colon X\to \GZip^\mu$, and let $(Y,\zeta_{\flag})$ be the flag space of $X$. We make the following assumption. 

\begin{assumption}\label{assume} \leavevmode
\begin{Alist}
\item\label{assume-A} For any $w\in W$ with $\ell(w)=1$, the closed stratum $\overline{Y}_w=\zeta_{\flag}^{-1}(\overline{\Fcal}_w)$ is pseudo-complete (\textit{i.e.}, any element of $H^0(\overline{Y}_w,\Ocal_{\overline{Y}_w})$ is locally constant on $\overline{Y}_w$ for the Zariski topology).
\item\label{assume-B} The restriction of $\zeta$ to any connected component $X^\circ \subset X$ is smooth and surjective.
\end{Alist}
\end{assumption}

For example, condition~\eqref{assume-A} is satisfied if $X$ is a proper $k$-scheme. In general, it can happen that the inclusion $\Ccal_{\pha}\subset \Ccal_{\zipsf}$ is strict. In this case, it is impossible for $Y$ to satisfy the Hasse property. However, Goldring and the second-named author conjectured the following in general. 

\begin{conjecture}\label{conj-Czip}
Under Assumption~\ref{assume}, we have $\Ccal_Y = \Ccal_{\zipsf}$.
\end{conjecture}

Let $S_K$ be the special fiber of a Hodge-type Shimura variety at a prime $p$ of good reduction. In this case, we write $C_K(\overline{\FF}_p)$ for the cone $C_Y$; \textit{i.e.}, 
\begin{equation}\label{cone-CKFp}
    C_K(\overline{\FF}_p) \colonequals \left\{\lambda\in X^*(T) \relmiddle| H^0(S_K,\Vcal_I(\lambda))\neq 0\right\}.
\end{equation}
By \cite[Corollary 1.5.3]{Koskivirta-automforms-GZip}, the saturation of $C_K(\overline{\FF}_p)$ is independent of $K$, so we simply denote it by $\Ccal(\overline{\FF}_p)$. Let $\zeta\colon S_K\to \GZip^\mu$ be the map \eqref{zeta-Shimura}. We do not know whether the pair $(X,\zeta)$ always satisfies condition~\eqref{assume-A}  of Assumption~\ref{assume}. However, by \cite[Theorem 6.2.1]{Goldring-Koskivirta-Strata-Hasse}, the map $\zeta\colon S_K\to \GZip^\mu$  admits an extension to a toroidal compactification
\begin{equation}
    \zeta^\Sigma\colon S_{K}^{\Sigma}\lra \GZip^\mu, 
\end{equation}
where $\Sigma$ is a sufficiently fine cone decomposition. By construction, the pullback $\Vcal_I^\Sigma(\lambda)\colonequals \zeta^{\Sigma,*}(\Vcal_I(\lambda))$ is the canonical extension of $\Vcal_I(\lambda)$ to $S_K^\Sigma$. Furthermore, by \cite[Theorem 1.2]{Andreatta-modp-period-maps}, the map $\zeta^\Sigma$ is smooth. Since $\zeta$ is surjective, $\zeta^\Sigma$ is also surjective. By \cite[Proposition 6.20]{Wedhorn-Ziegler-tautological}, any connected component $S^\circ\subset S_{K}^{\Sigma}$ intersects the unique zero-dimensional stratum. Since $\zeta^\Sigma \colon S^\circ \to \GZip^\mu$ is smooth, it has an open image; therefore,  it must be surjective. In particular, the pair $(S_{K}^{\Sigma},\zeta^\Sigma)$ satisfies conditions~\eqref{assume-A} and~\eqref{assume-B}. Furthermore, in most cases, Koecher's principle holds by \cite[Theorem 2.5.11]{Lan-Stroh-stratifications-compactifications}; \textit{i.e.}, we have an equality
\begin{equation}
    H^0\left(S^{\Sigma}_K, \Vcal^{\Sigma}_I(\lambda)\right) = H^0\left(S_K,\Vcal_I(\lambda)\right).
\end{equation}
In particular, the cone attached to the pair $(S_{K}^{\Sigma},\zeta^\Sigma)$ is the same as the cone attached to $(S_K,\zeta)$, namely $C_K(\overline{\FF}_p)$. Therefore, by the above discussion, we deduce that Conjecture~\ref{conj-Czip} applies to Shimura varieties and predicts the following. 

\begin{conjecture}\label{conj-shimura}
If $S_K$ is the special fiber of a Hodge-type Shimura variety at a prime $p$ of good reduction, we have $\Ccal(\overline{\FF}_p) = \Ccal_{\zipsf}$.
\end{conjecture}

In \cite[Theorem D]{Goldring-Koskivirta-global-sections-compositio}, the authors proved that certain Shimura varieties satisfy the Hasse property. Specifically, they showed the following. 

\begin{theorem}[\textit{cf.} {\cite[Theorem D]{Goldring-Koskivirta-global-sections-compositio}}]\label{thmcones}
Let $(X,\zeta)$ be a pair that satisfies Assumption~\ref{assume}, and let $(Y,\zeta_{\flag})$ be the flag space of $X$. Suppose that $(G,\mu)$ is one of the following three pairs:
\begin{assertionlist}
    \item\label{thmcones-1} $G$ is an $\FF_p$-form of\, $\GL_2^n$ for some $n\geq 1$, and $\mu$ is nontrivial on each factor;
    \item\label{thmcones-2} $G=\GL_{3,\FF_p}$ and $\mu\colon z\mapsto \diag(z,z,1)$; 
    \item\label{thmcones-3} $G=\GSp(4)_{\FF_p}$ and $\mu\colon z\mapsto \diag(z,z,1,1)$.
\end{assertionlist}
Then $(X,\zeta)$ satisfies the Hasse property. In other words, we have $\Ccal_Y = \Ccal_{\zipsf} = \Ccal_{\pha}$
\end{theorem}

Theorem~\ref{thmcones} also holds if we change the group $G$ to a group with the same adjoint group. By the above discussion, Theorem~\ref{thmcones} applies to Hilbert--Blumenthal Shimura varieties, Picard surfaces at a split prime, Siegel modular threefolds and shows that Conjecture~\ref{conj-Czip} holds in each case. Goldring and the second-named author proved Conjecture~\ref{conj-Czip} in \cite{Goldring-Koskivirta-divisibility} for certain Shimura varieties for which the inclusion $\Ccal_{\pha}\subset \Ccal_{\zipsf}$ is strict. Namely, they showed Conjecture~\ref{conj-Czip} for the Siegel modular variety $\Acal_3$ as well as unitary Shimura varieties of signature $(r,s)$ with $r+s\leq 4$ at split or inert primes, except when $r=s=2$ and $p$ is inert. With the exception of the case $r=s=2$ and $p$ split, the inclusion $\Ccal_{\pha}\subset \Ccal_{\zipsf}$ is strict in each of these cases.

\subsection{First properties of $\boldsymbol{C_{\zipsf}}$}
Let $(G,\mu)$ be a cocharacter datum over $\FF_q$ and $\Zcal_{\mu}=(G,P,L,Q,M)$ the attached zip datum (see Section~\ref{subsec-cochar}). Fix a frame $(B,T,z)$ with $z=\sigma(w_{0,I})w_0$ (see Section~\ref{subsec-cochar}). Let $C_{\zipsf}\subset X^*(T)$ be the zip cone, defined in Equation~\eqref{equ-CZip}. We start with some elementary properties of $C_{\zipsf}$. As we already noted, we have $C_{\zipsf}\subset X^*_{+,I}(T)$. Furthermore, the cone $C_{\zipsf}$ has maximal rank in $X^*(T)$, in the sense that $\Span_\QQ(C_{\zipsf})=X^*(T)\otimes_\ZZ \QQ$. This was shown in \cite[Lemma 3.4.2]{Goldring-Koskivirta-Strata-Hasse} (with the notation of \loccitn, $\mathcal{C}_{w_0}\subset C_{\zipsf}$ and $\mathcal{C}_{w_0}$ has maximal rank). Note that the cocharacter datum is assumed to be of Hodge type in \cite[Section~3.4]{Goldring-Koskivirta-Strata-Hasse}, but this assumption is unnecessary for \cite[Lemma 3.4.2]{Goldring-Koskivirta-Strata-Hasse}.

Next, we consider line bundles on $\GZip^\mu$. Recall that $\Vcal_I(\lambda)$ is a line bundle if and only if $\lambda\in X^*(L)$ (viewed as a subgroup of $X^*(T)$). Define the following set:
\begin{equation}\label{char-L-reg}
X^*_{-}(L)_{\reg} = \left\{\lambda\in X^*(L) \relmiddle| \langle \lambda,\alpha^\vee \rangle<0, \ \forall \alpha\in \Delta^P \right\}.
\end{equation}
These characters were termed $L$-ample in \cite[Definition N.5.1]{Goldring-Koskivirta-Strata-Hasse}. The notation used in \eqref{char-L-reg} is more enlightening, since these characters are in particular in $X_{-}^*(T)$ (the cone of anti-dominant characters). 
An immediate consequence of \cite[Theorem 5.1.4]{Koskivirta-Wedhorn-Hasse} is the inclusion $X^*_{-}(L)_{\reg}\subset \Ccal_{\zipsf}$. Set $X^*_{-}(L)\colonequals X^*_{-}(T) \cap X^*(L)$. The stronger inclusion $X^*_{-}(L)\subset \Ccal_{\zipsf} $ is claimed in \cite[Proposition 1.6.1]{Koskivirta-automforms-GZip} with an incomplete proof, so we give one below. 

\begin{proposition}\label{prop-charL}
We have $X^*_{-}(L) \subset \Ccal_{\zipsf}$.
\end{proposition}

\begin{proof}
Let $\lambda \in X^*_{-}(T) \cap X^*(L)$. Applying \cite[Theorem 3.4.1]{Imai-Koskivirta-vector-bundles} to the $1$-dimensional $L$-representation $V_I(\lambda)$, we obtain 
\begin{equation}
H^0\left(\GZip^\mu,\Vcal_I(\lambda)\right)=V_I(\lambda)^{L_{\varphi}}\cap 
 \bigcap_{\alpha \in \Delta^P} 
 \fil_{\delta_{\alpha}}^{\Xi_{\alpha},\mathbf{a}_{\alpha},\mathbf{r}_{\alpha}} V_{\lambda}.
\end{equation}
Furthermore, $\fil_{\delta_{\alpha}}^{\Xi_{\alpha},\mathbf{a}_{\alpha},\mathbf{r}_{\alpha}} V_{\lambda}=V_\lambda=V_I(\lambda)$ if $\langle \lambda, \delta_\alpha \rangle \geq 0$; it is $0$ otherwise. Let $d_\alpha\geq 1$ be an integer such that~$\alpha$ is defined over $\FF_{q^{d_\alpha}}$. We find that $\delta_\alpha=-\frac{1}{q^{d_\alpha}-1}\sum_{i=0}^{d_\alpha-1} q^i \sigma^i(\alpha^\vee)$. Since $\lambda\in X^*_{-}(T)$, we have $\langle \lambda , \sigma^{i}(\alpha^\vee) \rangle \leq 0$ for all $i$, hence $\langle \lambda, \delta_\alpha \rangle \geq 0$. We deduce  $H^0(\GZip^\mu,\Vcal_I(\lambda))=V_I(\lambda)^{L_{\varphi}}$. Finally, if we change $\lambda$ to $N\lambda$, where~$N$ divides the order of the finite group scheme $L_\varphi$, we obtain $H^0(\GZip^\mu,\Vcal_I(\lambda))=V_I(\lambda)$. In particular, this space is nonzero, and this proves the result.
\end{proof}

\subsection{Norm of the highest-weight vector}\label{subsec-norm}

Recall that we always have $1\in U_\mu$, where $U_{\mu}\subset G_k$ is the unique open $E$-orbit. Recall the definition of the finite subgroup $L_\varphi\subset L$ given in \eqref{Lphi-equ}. Put $N_\varphi=|L_0(\FF_q)|q^m$, where $L_0$ is the Levi subgroup defined in \eqref{eqL0} and $m\geq 0$ is the smallest integer such that the finite unipotent group $L_\varphi^{\circ}$ is annihilated by $\varphi^m$. For $\lambda,\lambda'\in X^*(T)$ and $f\in V_I(\lambda)$, $f'\in V_I(\lambda')$, let $ff'\in V_I(\lambda+\lambda')$ be the image of $f\otimes f'$ by the map \eqref{Vlambda-natural-map}. For $\lambda\in X^*_{+,I}(T)$ and $f\in V_I(\lambda)$, define
\begin{equation} \label{norm-def-eq}
\Norm_{L_\varphi}(f)\colonequals \left(\prod_{s\in L_0(\FF_q)}s\cdot f \right)^{q^m}\in V_I(N_\varphi \lambda).
\end{equation}
It is clear that $\Norm_{L_\varphi}(f)$ is $L_\varphi$-invariant. In particular, it gives rise to an element in $H^0(\Ucal_\mu,\Vcal(N_\varphi \lambda))$ by Lemma~\ref{lem-Umu-sections}. In general, it is difficult to determine whether $\Norm_{L_\varphi}(f)$ extends to a global section. However, this is possible when $f$ is a highest-weight vector, as we now explain.

Let $f_\lambda\in V_I(\lambda)$ be a nonzero element in the highest-weight line of $V_I(\lambda)$. The following result generalizes \cite[Theorem 2]{Koskivirta-automforms-GZip} (where $P$ was assumed to be defined over $\FF_p$; here we do not make this assumption). For $\alpha\in \Delta^P$, denote by $r_\alpha$ the smallest integer $r\geq 1$ such that $\sigma^r(\alpha)=\alpha$.

\begin{proposition}\label{prop-Norm}
The section $\Norm_{L_\varphi}(f_\lambda)$ extends to a global section over $\GZip^\mu$ if and only if for all $\alpha \in \Delta^P$, the following holds:
\begin{equation}\label{formula-norm}
\sum_{w\in W_{L_0}(\FF_q)} \sum_{i=0}^{r_\alpha-1} q^{i+\ell(w)} \ \left\langle w\lambda, \sigma^i\left(\alpha^\vee\right) \right\rangle\leq 0.
\end{equation}
\end{proposition}

Before giving the proof, we need to recall some facts from \cite[Section~3.1]{Imai-Koskivirta-vector-bundles}. First, we have
\begin{equation}\label{GminusU}
    G_k \setminus U_{\mu} = \bigcup_{\alpha \in \Delta^P} Z_\alpha, \quad Z_\alpha=\overline{E\cdot s_\alpha}, 
\end{equation}
where $E\cdot s_\alpha$ denotes the $E$-orbit of $s_{\alpha}$ and the bar denotes the Zariski closure. This follows easily from Theorem~\ref{thm-E-orb-param}. For any $\alpha\in \Delta^P$, define an open subset 
\begin{equation}
X_\alpha  \colonequals  G_k \setminus \bigcup_{\beta\in \Delta^P,\, \beta\neq \alpha}Z_\beta.
\end{equation}
Then $U_{\mu}\subset X_\alpha$, and $X_{\alpha}\setminus U_{\mu}$ is irreducible. Choose a realization
$(u_{\alpha})_{\alpha \in \Phi}$, and let $\phi_{\alpha}\colon \SL_2\to G$ be the map attached to $\alpha$ (see Section~\ref{subsec-notation}). Set $Y\colonequals E \times \AA^1$ and $Y_0\colonequals E\times \GG_{\mathrm{m}}$. For $\alpha\in \Delta^P$, define $\psi_\alpha \colon Y \to G$ by
\begin{equation}\label{phia}
\psi_\alpha \colon ((x,y),t)\longmapsto x \phi_{\alpha} \left(A(t)\right)y^{-1}, \quad \textrm{where} \ A(t)=\left(
\begin{matrix}
t & 1 \\ -1 & 0 \end{matrix}
\right)\in \SL_{2,k}.
\end{equation}
It satisfies $\psi_\alpha((x,y),t)\in X_\alpha$ for all $((x,y),t)\in Y$ and $\psi_\alpha((x,y),t)\in U_\mu$ if and only if $t\neq 0 $ (see \cite[Proposition 3.1.4]{Imai-Koskivirta-vector-bundles}).

We now prove Proposition~\ref{prop-Norm}. We use a similar argument to that in \cite[Theorem 3.5.3]{Koskivirta-automforms-GZip}. Set $\Ucal'_\mu\colonequals\pi^{-1}(\Ucal_\mu)$, where $\pi\colon \GF^\mu\to \GZip^\mu$ is the natural projection. One  clearly has $\Ucal'_\mu\simeq [E'\backslash U_\mu]$ via the isomorphism $\GF^\mu\simeq [E'\backslash G]$ explained in Section~\ref{subsec-zipflag}. We have an identification
\begin{equation}
    H^0\left(\Ucal_\mu,\Vcal_I\left(N_\varphi\lambda\right)\right)=H^0\left(\Ucal'_\mu,\Vcal_{\flag}\left(N_{\varphi}\lambda\right)\right)
\end{equation}
similarly to \eqref{ident-H0-GF}. In particular, we can view $\Norm_{L_\varphi}(f_{\lambda})$ as a function $h\colon U_\mu\to \AA^1$ satisfying the relation $h(axb^{-1})=\lambda(a)h(x)$ for all $(a,b)\in E'$ and $x\in U_\mu$ (using \eqref{globquot}). Specifically, the function $h$ is given by 
\begin{equation}\label{functionh-eq}
  h\left(x_1x_2^{-1}\right)=\Norm_{L_\varphi}(f_\lambda)\left(\theta_L^P(x_1)^{-1}\right)
\end{equation}
for all $(x_1,x_2)\in E$ using \eqref{eq:H0expid}. 
The function $h$ is well defined because $\Norm_{L_\varphi}(f_\lambda)$ is $L_\varphi$-invariant. Furthermore, $\Norm_{L_\varphi}(f_\lambda)$ extends to $\GZip^\mu$ if and only if $h$ extends to a function $G\to \AA^1$. By the strategy explained in \cite[Section~3.2]{Koskivirta-automforms-GZip} and in \cite[Section~3.1]{Imai-Koskivirta-vector-bundles}, the function $h$ extends to $G$ if and only if for each $\alpha\in \Delta^P$, the function $h\circ \psi_\alpha \colon Y_0\to \AA^1$ extends to a function $Y\to \AA^1$. It remains to compute the $t$-valuation of the function $h\circ \psi_\alpha$, viewed as an element of $R[t,\frac{1}{t}]$, where $R=k[E]$ is the ring of functions of $E$. Put 
\begin{equation}\label{malpha-equ}
 m_{\alpha} =\min \left\{ m \geq 1 \mid 
 \sigma^{-m}(\alpha) \notin I \right\}, \quad \alpha \in \Delta^P
\end{equation}
and $t_{\alpha}=t^{-1}\alpha(\varphi(\delta_{\alpha}(t)))^{-1} =t \alpha(\delta_{\alpha}(t))^{-1} \in t^{\QQ}$, where $t$ is an indeterminate and $\delta_{\alpha}=\wp_*^{-1}(\alpha^{\vee})$ as defined in Section~\ref{subsec-global-sections}. Set
\[
 u_{t,\alpha}=\prod_{i=1}^{m_{\alpha}-1} \phi_{\sigma^{-i}(\alpha)} 
 \left( \begin{pmatrix}
1&-t_{\alpha}^{\frac{1}{q^i}}\\0&1
\end{pmatrix} \right),  
\]
where the product is taken in increasing order of indices. By the proof of \cite[Proposition 3.1.4]{Imai-Koskivirta-vector-bundles}, for all $(x,y)\in E$ and $t\in \GG_{\mathrm{m}}$, we can write $\psi_\alpha((x,y),t)=x_1x_2^{-1}$ with $(x_1,x_2)\in E$ and
\begin{equation}
    x_1=x\phi_{\alpha} \left(\begin{pmatrix}
1&0\\-t^{-1}&1
\end{pmatrix} \right) \delta_{\alpha}(t) 
u_{t,\alpha}.
\end{equation}
By the definition of $m_\alpha$, all the roots $\sigma^{-i}(\alpha)$ (for $1\leq i\leq m_\alpha-1$) appearing in the formula of $u_{t,\alpha}$ lie in $I$. Using \eqref{functionh-eq}, we deduce 
\begin{align}
h\circ \psi_\alpha((x,y),t)&=\Norm_{L_\varphi}(f_\lambda)\left(u_{t,\alpha}^{-1}\delta_\alpha(t)^{-1}\theta_L^P(x)^{-1}\right) \\
&=\left(\prod_{s\in L_0(\FF_q)} f_\lambda\left(su_{t,\alpha}^{-1}\delta_\alpha(t)^{-1}\theta_L^P(x)^{-1}\right) \right)^{q^m}.
\end{align}
Consider the element $f_\lambda(su_{t,\alpha}^{-1}\delta_\alpha(t)^{-1}\theta_L^P(x)^{-1})$, which lies in $R[t^{\QQ}]$. We can still speak of the $t$-valuation of this element, which is a rational number. Equivalently, to simplify notation, we change $\theta_L^P(x)^{-1}$ to a generic element $g\in L$, and we compute the $t$-valuation of $F_s(t,g)\colonequals f_\lambda(su_{t,\alpha}^{-1}\delta_\alpha(t)^{-1}g)$, viewed as an element of $k[L][t^\QQ]$. Let $v_\alpha(s)$ be this valuation. 
We put $B^+_L=B^+ \cap L$. 
Define a parabolic subgroup of $L$ by $Q_0\colonequals L_0 B^+_L$. It is clear that $u_{t,\alpha}$ lies in $R_{\mathrm{u}}(Q_0)$; thus for all $s\in L_0(\FF_q)$, we have $su_{t,\alpha}^{-1}s^{-1}\in R_{\mathrm{u}}(Q_0)$. Since $f_\lambda$ is invariant by $R_{\mathrm{u}}(B^+_L)$, we obtain 
$F_s(t,g)=f_\lambda(s\delta_\alpha(t)^{-1}g)$. Now, the rest of the proof is completely similar to that of \cite[Theorem 3.5.3]{Koskivirta-automforms-GZip}. We recall it briefly.

Let $B_{L_0}\colonequals B\cap L_0$ and $B^+_{L_0}\colonequals B^+\cap L_0$. If we change $s$ to $bs$ with $b\in B^{+}_{L_0}(\FF_q)$, then $v_\alpha(bs)=v_\alpha(s)$. Indeed, this follows from $f_\lambda(bs\delta_\alpha(t)^{-1}g)=\lambda(b)^{-1} f_\lambda(s\delta_\alpha(t)^{-1}g)$ since $f_\lambda$ is a $B^+_L$-eigenfunction. Similarly, we claim that $v_\alpha(sb)=v_\alpha(s)$ for all $b\in B^+_{L_0}(\FF_q)$. By symmetry, it suffices to show $v_\alpha(sb)\geq v_\alpha(s)$. We can write
$F_{sb}(t,g)=F_s(t,\Gamma(t)g)$, where
\begin{equation}
    \Gamma(t)\colonequals
    \delta_\alpha(t) b \delta_\alpha(t)^{-1}.
\end{equation}
We view $\Gamma$ as a map $\Gamma\colon \spec(k[t^\QQ])\to L$. Since $\alpha\in \Delta^P$, the cocharacter $\alpha^\vee$ is anti-$L$-dominant. It follows that for all $j\in \ZZ$, $\sigma^j(\alpha^\vee)$ is an anti-$L_0$-dominant quasi-cocharacter. It is easy to see that $\delta_\alpha$ is explicitly given by the formula
\begin{equation}\label{eq-deltaalpha}
    \delta_\alpha=-\frac{1}{q^{r_\alpha}-1}\sum_{j=0}^{r_\alpha-1}q^j\sigma^j\left(\alpha^\vee\right). 
\end{equation}
In particular, $\delta_\alpha$ is $L_0$-dominant. We deduce that the function $\Gamma$ extends to a map $\spec(k[t^{\QQ_{\geq 0}}])\to L$. This follows simply from the fact that $\delta_\alpha(t)$ acts on the root space $U_\beta$ (for $\beta\in \Phi$) by $t^{\langle \beta,\delta_\alpha \rangle}$, using \eqref{eq:phiconj}. Write $F_s(t,g)=t^{v_\alpha(s)}F_{s,0}(t,g)$, where $F_{s,0}(t,g)$ is an element of $k[L][t^\QQ]$ whose $t$-valuation is $0$. Then $F_{sb}(t,g)=t^{v_\alpha(s)}F_{s,0}(t,\Gamma(t)g)$ and 
$F_{s,0}(t,\Gamma(t)g) \in k[L][t^{\QQ_{\geq 0}}]$. 
Hence $v_\alpha(sb)\geq v_\alpha(s)$ as claimed.

Now, consider the Bruhat decomposition of $L_0(\FF_q)$:
\begin{equation}
    L_0\left(\FF_q\right)=\bigsqcup_{w\in W_{L_0}(\FF_q)} B^+_{L_0}\left(\FF_q\right) w B^+_{L_0}\left(\FF_q\right)
\end{equation}
as in \cite[Lemma 3.4.4]{Koskivirta-automforms-GZip}. By \cite[Lemma 3.4.5]{Koskivirta-automforms-GZip}, one has
\begin{equation}
  \left|B^+_{L_0}\left(\FF_q\right) w B^+_{L_0}\left(\FF_q\right)\right| =  \left|T\left(\FF_q\right)\right|q^{\dim(R_{\mathrm{u}}(B_{L_0}))+\ell(w)}.
\end{equation} 
Thus, we can determine $v_\alpha$ completely from the values $v_\alpha(w)$ for $w\in W_{L_0}(\FF_q)$. Similarly to \cite[Proposition 3.5.2]{Koskivirta-automforms-GZip}, we have $v_\alpha(w)=\langle w\lambda, \delta_\alpha \rangle$. We deduce that the $t$-valuation of $h\circ \psi_\alpha((x,y),t)$ is
\begin{equation}
 q^m\sum_{s\in L_0(\FF_q)} v_\alpha(s) =   q^m \left|T\left(\FF_q\right)\right|q^{\dim(R_{\mathrm{u}}(B_{L_0}))} \sum_{w\in W_{L_0}(\FF_q)} q^{\ell(w)}\langle w\lambda, \delta_\alpha \rangle.
\end{equation}
The statement of Proposition~\ref{prop-Norm} then follows by replacing $\delta_\alpha$ by the expression in \eqref{eq-deltaalpha}.

\begin{definition}
We denote by $\Ccal_{\hwsf}\subset X_{+,I}^*(T)$ the subset of characters $\lambda$ satisfying the inequalities \eqref{formula-norm} and call $\Ccal_{\hwsf}$ the highest-weight cone.
\end{definition}

By construction, for all $\lambda\in \Ccal_{\hwsf}$, the section $\mathbf{f}_{\lambda}\colonequals \Norm_{L_\varphi}(f_\lambda)$ is a nonzero section of $\Vcal_I(N_\varphi\lambda)$ over $\GZip^\mu$. In particular, we deduce $N_\varphi\lambda\in C_{\zipsf}$ and hence $\lambda \in \Ccal_{\zipsf}$. We deduce that $\Ccal_{\hwsf}\subset \Ccal_{\zipsf}$. If $S_K$ is the good reduction special fiber of a Hodge-type Shimura variety and $\zeta\colon S_K\to \GZip^\mu$ is the map \eqref{zeta-Shimura}, we obtain a family of mod $p$ automorphic forms $\zeta^*(\mathbf{f}_{\lambda})_{\lambda\in \Ccal_{\hwsf}}$. We also have, by Section~\ref{sec-partialHasse}, the family $\zeta^*(\Ha_{\chi})_{\chi\in X^*_{+}(T)}$. The vanishing locus of $\Ha_{\chi}$ is a union of Ekedahl--Oort strata of codimension~$1$. On the other hand, the vanishing locus of $\mathbf{f}_{\lambda}$ is highly nontrivial. It is an interesting closed subvariety stable by Hecke operators.

\subsection{Partial Hasse invariant cone, Griffiths--Schmid cone}
As mentioned in Section~\ref{sec-gen-set}, 
we give an unambiguous definition of $C_{\pha}$.

\begin{definition}[\textit{cf.} {\cite[Definition 1.7.1]{Koskivirta-automforms-GZip}}] \label{definition-CHasse}
Define $C_{\pha}$ as the image of $X^*_+(T)$ by
\begin{equation}\label{equ-maph}
h_{\Zcal}\colon X^*(T)\lra X^*(T), \quad \lambda \longmapsto \lambda-q w_{0,I} (\sigma^{-1} \lambda).
\end{equation}
\end{definition}

We write $\Ccal_{\pha}$ for the saturation of $C_{\pha}$. One has $\Ccal_{\pha}\subset X_{+,I}^*(T)$ since $- w_{0,I} \sigma^{-1}(\lambda)\in X_{+,I}^*(T)$ for $\lambda\in X^*_+(T)$. 
If $G$ is split over $\FF_q$, we have an equivalence
\begin{equation}\label{equivSbt}
\lambda\in  \Ccal_{\pha}  \ \Longleftrightarrow \ q w_{0,I}\lambda+\lambda \in X_-^*(T).
\end{equation}

\begin{definition}\label{GSdef}
Let $\Ccal_{\GSsf}$ denote the set of characters $\lambda\in X^*(T)$ satisfying
\begin{align}
\langle \lambda, \alpha^\vee \rangle &\geq 0 \ \textrm{ for }\alpha\in I, \\
\langle \lambda, \alpha^\vee \rangle &\leq 0 \ \textrm{ for }\alpha\in \Phi^+ \setminus \Phi^{+}_{L}. 
\end{align}
\end{definition}
One easily sees that $\lambda\in \Ccal_{\GSsf}$ if and only if $-w_{0,I}\lambda$ is dominant. Clearly, $\Ccal_{\GSsf}$ is a saturated subcone of $X^*(T)$ and contains $X_{-}^*(L)$. We explain the significance of $\Ccal_{\GSsf}$. Consider a Hodge-type Shimura variety $\shgx_K$ over the reflex field $\mathbf{E}$, with good reduction at the prime $p$, as in Section~\ref{subsec-Shimura}. Similarly to \eqref{cone-CKFp}, we define a cone $C_K(\CC)$ by
\begin{equation}
C_K(\CC)=\{\lambda\in X^*(T) \mid H^0(\shgx_K\otimes_{\mathbf{E}} \CC,\Vcal_I(\lambda))\neq 0\}.
\end{equation}
Again, the saturation of $C_K(\CC)$ is independent of $K$, so we denote it by $\Ccal(\CC)$. Based on the results of \cite{Griffiths-Schmid-homogeneous-complex-manifolds}, it is expected that $\Ccal(\CC) =\Ccal_{\GSsf}$, but we could not find a reference for this conjecture. The inclusion $\Ccal(\CC) \subset \Ccal_{\GSsf}$ is proved for general Hodge-type Shimura varieties in \cite[Theorem 2.6.4]{Goldring-Koskivirta-GS-cone}.

By reduction modulo $p$, one can show that $\Ccal(\CC) \subset \Ccal(\overline{\FF}_p) $ (see \cite[Proposition 1.8.3]{Koskivirta-automforms-GZip}). Combining the expectation $\Ccal(\CC) =\Ccal_{\GSsf}$ with Conjecture~\ref{conj-shimura}, one should expect an inclusion $\Ccal_{\GSsf}\subset \Ccal_{\zipsf}$ (at least for groups attached to Shimura varieties). In Theorem~\ref{thmGSzip}, we confirm this expectation and prove $\Ccal_{\GSsf}\subset \Ccal_{\zipsf}$ in general (this was previously shown in \cite{Koskivirta-automforms-GZip} only in the case when $P$ is defined over $\FF_p$). This result gives evidence for Conjecture~\ref{conj-shimura}.

\subsection{Inclusion relations of cones}

Let us briefly summarize in a diagram the cones that appear in our construction. We explain below the various inclusion relations between these cones as well as the conjectures pertaining to these objects.

\begin{equation}\label{conediag}
\xymatrix@M=5pt{
&\Ccal_{\pha} \ar@{^{(}->}[rd] &  \\
X^*_{-}(L) \ar@{^{(}->}[r]  \ar@{^{(}->}[rd]  & \Ccal_{\hwsf} \ar@{^{(}->}[r] &  \Ccal_{\zipsf}  \\
& \Ccal_{\GSsf} \ar@{^{(}->}[ru] \ar@{^{(}->}[r] & \Ccal_{\lwsf}\rlap{,} \ar@{^{(}.>}[u]_{\textrm{Cond. }\ref{cond-commute}} }
\qquad \quad
\xymatrix@M=5pt{
&&\\
&  \Ccal_{\zipsf} \ar@{^{(}->}[r]^{?} & \Ccal(\overline{\FF}_p) \\
 \Ccal_{\GSsf} \ar@{^{(}->}[ru]  & \Ccal(\CC)\rlap{.} \ar@{_{(}->}[l]^{?} \ar@{^{(}->}[ru] }
\end{equation}
All arrows of these diagrams are inclusions, and all cones are contained in $X^*_{+,I}(T)$. All plain arrows are proved inclusions that hold unconditionally. The left diagram is entirely group-theoretic and holds for arbitrary pairs $(G,\mu)$. The lowest-weight cone $\Ccal_{\lwsf}$ is defined in Section~\ref{sec-low}. The inclusion $\Ccal_{\lwsf}\subset \Ccal_{\zipsf}$ is shown only under Condition~\ref{cond-commute} (hence the dotted arrow in the above diagram). 

The right diagram applies to Shimura varieties of Hodge (or abelian) type. The arrows labeled with a question mark are conjecturally equalities.

\begin{lemma}\label{XLreg-contained}
One has $X_{-}^*(L)\subset \Ccal_{\hwsf}$.
\end{lemma}

\begin{proof}
For $\lambda \in X_{-}^*(L)$, we have $w\lambda = \lambda$ for all $w\in W_{L}$. Hence $\langle w\lambda, \sigma^i\alpha^\vee \rangle \leq 0$ for all $i\in \ZZ$, $w\in W_{L_{0}}(\FF_q)$ and $\alpha\in \Delta^P$. Thus $\lambda \in \Ccal_{\hwsf}$.
\end{proof}

We postpone the proof of $\Ccal_{\GSsf}\subset \Ccal_{\zipsf}$ in the general case, which is quite involved. The following was proved in \cite[Corollary 3.5.6]{Koskivirta-automforms-GZip}.

\begin{lemma}\label{CGS-contained} 
Assume that $P$ is defined over $\FF_q$. Then one has $\Ccal_{\GSsf}\subset \Ccal_{\hwsf}$.
\end{lemma}
This shows $\Ccal_{\GSsf}\subset \Ccal_{\zipsf}$ in the case when $P$ is defined over $\FF_q$. However, the inclusion $\Ccal_{\GSsf}\subset \Ccal_{\hwsf}$ is false in general. This happens for example in the case of Picard modular surfaces of signature $(2,1)$ at an inert prime, where the group $G$ is a unitary group of rank $3$ over $\FF_p$. In this example, all cones $\Ccal_{\pha}$, $\Ccal_{\GSsf}$, $\Ccal_{\hwsf}$ and $\Ccal_{\zipsf}$ are distinct, and there is no inclusion relation between the first three. These four cones are also distinct for $G=\Sp(6)$ (see \cite[Section~5.5]{Koskivirta-automforms-GZip}), and more generally for $G=\Sp(2n)$, $n\geq 3$. In particular, in those cases the inclusion $\Ccal_{\pha} \subset \Ccal_{\zipsf}$ is strict; hence $\GF^\mu$ does not satisfy the Hasse property. As a consequence, the Siegel-type Shimura variety $\Acal_n$ does not satisfy the Hasse property for $n\geq 3$.

\section{Hasse-type zip data}\label{sec4}

\subsection{Topology of $\boldsymbol{C_{\zipsf,\RR_{\geq 0}}}$}
Let $(G,\mu)$ be a cocharacter datum. We showed $X^*_{-}(L)\subset \Ccal_{\zipsf}$ in Proposition~\ref{prop-charL}. For $X_{-}^*(L)_{\reg}$ (see Equation~\eqref{char-L-reg}), we have a more precise result (see~\cite[Theorem 5.1.4]{Koskivirta-Wedhorn-Hasse}).

\begin{theorem}\label{theo-Hasse-inv}
For all $\lambda\in X^*_{-}(L)_{\reg}$, there is a section $h\in H^0(\GZip^\mu,\Vcal_I(N_\varphi\lambda))$ whose nonvanishing locus is exactly $\Ucal_\mu$.
\end{theorem}
Here $N_\varphi\geq 1$ is the integer defined in Section~\ref{subsec-norm}. Since $\lambda\in X^*(L)$, the vector bundle $\Vcal_I(\lambda)$ is a line bundle, and thus $\Vcal_I(N_\varphi\lambda)=\Vcal_I(\lambda)^{\otimes N_\varphi}$. A subset of an $\RR$-vector space stable under linear combination with coefficients in $\RR_{\geq 0}$ will be called an $\RR_{\geq 0}$-subcone. We endow $X^*_{+,I}(T)_{\RR_{\geq 0}}$ with the subspace topology of $X^*(T)_\RR$.

\begin{lemma}\label{lem-cone-topo}
Let $C\subset X^*_{+,I}(T)_{\RR_{\geq 0}}$ be an $\RR_{\geq 0}$-subcone, and let $\lambda\in C$. Then $C$ is a neighborhood of $\lambda$ in $X^*_{+,I}(T)_{\RR_{\geq 0}}$ if and only if for all $\lambda'\in X^*_{+,I}(T)_{\RR_{\geq 0}}$, there exists an $r\in \RR_{>0}$ such that $\lambda'+r\lambda \in C$.
\end{lemma}

\begin{proof}
First, assume that $C$ is a neighborhood of $\lambda$ in $X^*_{+,I}(T)_{\RR_{\geq 0}}$. There is an open subset $V$ of $X^*(T)_{\RR}$ such that $\lambda\in V\cap X^*_{+,I}(T)_{\RR_{\geq 0}} \subset C$. Fix $\lambda'\in X^*_{+,I}(T)_{\RR_{\geq 0}}$. For large $r\in \RR_{>0}$, we have $\lambda+\frac{\lambda'}{r}\in V$, and this element is also in $X^*_{+,I}(T)_{\RR_{\geq 0}}$. Thus $\lambda'+r\lambda \in C$.

We prove the converse. We claim that for all $\lambda'\in X^*_{+,I}(T)_{\RR_{\geq 0}}$, there exists an $r> 1$ such that $\lambda+\frac{\lambda'-\lambda}{r}\in C$. Indeed, let $r\in \RR_{>0}$ be such that $\lambda+\frac{\lambda'}{r}\in C$. Then for all $\gamma>0$, we have $\gamma\lambda + \frac{\gamma \lambda'}{r}=\lambda+\frac{\gamma(\lambda'-\lambda)}{r}+(\gamma-1+\frac{\gamma}{r})\lambda\in C$. For $\gamma=\frac{r}{r+1}$, we have $\gamma-1+\frac{\gamma}{r}=0$, hence $\lambda+\frac{\lambda'-\lambda}{r+1}\in C$. Hence, by taking $\lambda$ as the origin, we are reduced to the following: 
\begin{quote}
Let $X\subset \RR^n$ be an intersection of closed half-spaces 
containing $0$, and $Y \subset X$ a convex subset containing~$0$, satisfying: for all $x\in X$, $\exists r\in \RR_{>0}$, $\frac{x}{r}\in Y$. Then $Y$ is a neighborhood of $0$ in $X$. 
\end{quote}
Taking intersections with a neighborhood of $0$ in $\RR^n$ that is a convex polytope, we may assume that $X$ is a convex polytope. Since $X$ is the convex hull of finitely many points, there exists an $r>1$ such that $\frac{1}{r}X=\{\frac{x}{r} \mid x\in X\} \subset Y$. Hence, it suffices to show that $\frac{1}{r}X$ is a neighborhood of $0$ in $X$. There are linear forms $u_1, \dots , u_d$ on $\RR^n$ and $m_1, \dots , m_d\in \RR_{\geq 0}$ such that $x\in X$ if and only if $u_i(x)\leq  m_i$ for all $i=1, \dots, d$. Hence $u=(u_1, \dots , u_d )$ maps $X$ to $Z=\prod_{i=1}^d]-\infty, m_i]$. For $r>1$, $\frac{1}{r}Z$ is clearly a neighborhood of $0$ in $Z$; hence $\frac{1}{r}X=u^{-1}(\frac{1}{r}Z)$ is a neighborhood of $0$ in $X$.
\end{proof}

The following lemma was proved in a slightly restricted setting in \cite[Proposition 2.2.1]{Koskivirta-automforms-GZip}, so we restate it below. 

\begin{lemma}\label{lemma-zip-neighborhood}
The cone $C_{\zipsf, \RR_{\geq 0}}$ is a neighborhood of\, $X^*_{-}(L)_{\reg}$ in $X^*_{+,I}(T)_{\RR_{\geq 0}}$.
\end{lemma}

\begin{proof}
For $\lambda \in X^*_{-}(L)_{\reg}$, we show that $C_{\zipsf, \RR_{\geq 0}}$ is a neighborhood of $\lambda$ in $X^*_{+,I}(T)_{\RR_{\geq 0}}$. 
By Lemma~\ref{lem-cone-topo}, it suffices to show that for all $\lambda'\in X^*_{+,I}(T)_{\RR_{\geq 0}}$, there is an $r\in \RR_{>0}$ such that $\lambda'+r\lambda \in C_{\zipsf, \RR_{\geq 0}}$. 
We may assume $\lambda'\in X^*_{+,I}(T)$ by scaling. 
Let $h\in H^0(\GZip^\mu,\Vcal_I(N_\varphi\lambda))$ be the section provided by Theorem~\ref{theo-Hasse-inv}. By Lemma~\ref{lem-Umu-sections}, $H^0(\Ucal_\mu,\Vcal_I(N_\varphi \lambda'))$ is nonzero; let $h'$ be a nonzero element therein. This section may have poles on the complement of $\Ucal_\mu$. However, since $h$ vanishes on the complement of $\Ucal_\mu$, there exists a $d\geq 1$ such that $h^d h'$ has no poles. Hence $h^d h'\in H^0(\GZip^\mu, \Vcal_I(N_{\varphi}\lambda' + d N_{\varphi} \lambda))$, and thus $N_{\varphi}(\lambda' + d \lambda) \in C_{\zipsf}$, so $\lambda'+ d\lambda \in \Ccal_{\zipsf}$. The result follows.
\end{proof}

\begin{lemma}\label{lemma-GS-hw-nbh}
The cones $\Ccal_{\GSsf,\RR_{\geq 0}}$ and $\Ccal_{\hwsf, \RR_{\geq 0}}$ are neighborhoods of\, $X^*_{-}(L)_{\reg}$ in $X^*_{+,I}(T)_{\RR_{\geq 0}}$.
\end{lemma}

\begin{proof}
The open subset of $X^*_{+,I}(T)_{\RR_{\geq 0}}$ defined by the equations $\langle \lambda, \alpha^\vee \rangle <0$ for all $\alpha\in \Phi^+ \setminus \Phi^{+}_{L}$ 
is contained in $\Ccal_{\GSsf,\RR_{\geq 0}}$ and contains $X^*_{-}(L)_{\reg}$, which proves the first part of the assertion. Replacing $\leq$ by $<$ in the inequalities \eqref{formula-norm}, we get an open subset of $X^*_{+,I}(T)_{\RR_{\geq 0}}$ containing $X^*_{-}(L)_{\reg}$ (same proof as that of Lemma~\ref{XLreg-contained}), which proves the second part. 
\end{proof}

We may ask whether $C_{\pha, \RR_{\geq 0}}$ is also a neighborhood of $X_{-}^*(L)_{\reg}$. The proof of the following result is similar to that of \cite[Lemma 2.3.1]{Koskivirta-automforms-GZip}, where the cocharacter datum $(G,\mu)$ was assumed to be of Hodge type, but this assumption is superfluous. We partly reproduce  the proof to explain the appropriate changes (we replace the character $\eta_\omega$ in \cite[Lemma 2.3.1]{Koskivirta-automforms-GZip} by the set $X^*_{-}(L)_{\reg}$). The following holds for an arbitrary cocharacter datum $(G,\mu)$. 

\begin{proposition}\label{prop-Hassetype-equivalent}
The following are equivalent:
\begin{equivlist}
\item\label{pHe-1} The cone $C_{\pha, \RR_{\geq 0}}$ is a neighborhood of\, $X^*_{-}(L)_{\reg}$ in $X_{+,I}^*(T)_{\RR_{\geq 0}}$.
\item\label{pHe-2} One has $\Ccal_{\GSsf} \subset \Ccal_{\pha}$.
\item
\label{item-root-data-hasse-type}
The subgroup $P$ is defined over $\FF_q$, and the Frobenius $\sigma$ acts on $I$ by $\sigma(\alpha)=-w_{0,I}\alpha$ for all $\alpha\in I$.
\end{equivlist}
\end{proposition}

\begin{proof}
Since $\Ccal_{\GSsf,\RR_{\geq 0}}$ is a neighborhood of $X^*_{-}(L)_{\reg}$ in $X^*_{+,I}(T)_{\RR_{\geq 0}}$, we have~\eqref{pHe-2} $\Rightarrow$~\eqref{pHe-1}. Assume that~\eqref{pHe-1} holds. In particular, $X^*_{-}(L)_{\reg}\subset \Ccal_{\pha}$; hence $h_\Zcal^{-1}(X^*_{-}(L)_{\reg})\subset X^*_{+}(T)_{\RR_{\geq 0}}$. Let $\lambda\in X^*_{-}(L)_{\reg}$, and write $\lambda = h_\Zcal(\chi)$ for $\chi \in X^*_{+}(T)_{\RR_{\geq 0}}$.  Hence for all $\alpha\in I$, we have $\langle h_\Zcal(\chi),\alpha^\vee\rangle=0$, which amounts to $\langle \chi, \alpha^\vee \rangle= q\langle \chi, \sigma(w_{0,I}\alpha^\vee) \rangle$. Since $\alpha \in I$, $w_{0,I} \alpha$ is a negative root, and so is $\sigma(w_{0,I}\alpha)$. We deduce that $\langle \chi, \alpha^\vee \rangle= \langle \chi, \sigma(w_{0,I}\alpha^\vee) \rangle = 0$ (in particular, $\chi\in X^*(L)$). Since $X^*_{-}(L)_{\reg}$ generates $X^*(L)$, this shows that $h_\Zcal^{-1}$ maps $X^*(L)_\RR$ to itself, and all elements in the image satisfy $\langle \chi, \sigma(w_{0,I}\alpha^\vee) \rangle = 0$ for all $\alpha\in I$. For dimension reasons, $h^{-1}_\Zcal(X^*(L)_\RR)=X^*(L)_\RR$; hence any character $\chi \in X^*(L)$ is orthogonal to $\sigma(\alpha^\vee)$ for all $\alpha\in I$. Hence we must have $\sigma(I)=I$; thus~$P$ is defined over $\FF_q$.  Next, for $\alpha\in I$, take $\lambda_\alpha\in X_{+,I}^*(T)$ such that $\langle \lambda_\alpha , \beta^{\vee} \rangle=0$ for all $\beta \in \Delta \setminus \{\alpha\}$ and $\langle \lambda_\alpha, \alpha^\vee \rangle >0$. Let $\lambda\in X_{-}^*(L)_{\reg}$.  There exist an $r \in \mathbb{R}_{>0}$ and a $\chi_\alpha\in X_{+}^*(T)_{\mathbb{R}_{\geq 0}}$ such that $h_\Zcal(\chi_\alpha)=r\lambda+\lambda_\alpha$.  As before, we deduce $\langle \chi_\alpha, \beta^\vee \rangle = \langle \chi_\alpha, \sigma(w_{0,I}\beta^\vee) \rangle = 0$ for all $\beta\in I\setminus \{\alpha\}$. The character $\chi_\alpha$ cannot be orthogonal to all $\beta^\vee$ for $\beta\in I$; hence $\langle \chi_\alpha, \alpha^\vee \rangle \neq 0$. Furthermore, since the map $I\to I$, $\beta\mapsto -\sigma(w_{0,I}\beta)$ is a bijection, we must have $-\sigma(w_{0,I}\alpha)=\alpha$. This shows~\eqref{pHe-1} $\Rightarrow$ \eqref{item-root-data-hasse-type}. Finally, the implication~\eqref{item-root-data-hasse-type} $\Rightarrow$~\eqref{pHe-2} is completely similar to the implication (3) $\Rightarrow$ (4) in the proof of \cite[Lemma 2.3.1]{Koskivirta-automforms-GZip} (after changing $p$ to $q$).
\end{proof}

\begin{definition}\label{def-Hasse-type-cochardatum}
We say that a cocharacter datum $(G,\mu)$ is of Hasse type if the equivalent conditions of Proposition~\ref{prop-Hassetype-equivalent} are satisfied.
\end{definition}

The main result of this section is that~\eqref{pHe-1},~\eqref{pHe-2},~\eqref{item-root-data-hasse-type} above are also equivalent to the equality $\Ccal_{\pha}=\Ccal_{\zipsf}$. For the time being, the following is an immediate consequence of Lemma~\ref{lemma-zip-neighborhood}. 

\begin{corollary}\label{cor-Hassetype}
Assume that $\Ccal_{\pha} = \Ccal_{\zipsf}$ holds. Then $(G,\mu)$ is of Hasse type.
\end{corollary}

Recall that $\Ccal_{\pha} = \Ccal_{\zipsf}$ means by definition that $\GF^{\mu}$ satisfies the Hasse property (Definition~\ref{def-Hasse-property}). This shows that Theorem~\ref{thmDK}\eqref{thm-DK-item2} can only potentially generalize to Hodge-type Shimura varieties $S_K$ such that the associated zip datum $(G,\mu)$ is of Hasse type. Indeed, if the flag space of $S_K$ satisfies the Hasse property, then so does $\GF^{\mu}$, and hence $(G,\mu)$ must be of Hasse type by Corollary~\ref{cor-Hassetype}. In Theorem~\ref{thmcones}, all three cases~\eqref{thmcones-1},~\eqref{thmcones-2} and~\eqref{thmcones-3} are of Hasse type.

\subsection{Maximal flag stratum}\label{sec-max-flag-stratum}

We prove some technical results used in the proof of Theorem~\ref{main-thm-Hasse-type}. Let $(G,\mu)$ be an arbitrary cocharacter datum, and let $(B,T,z)$ be a frame with $z=\sigma(w_{0,I}) w_0$ (see Remark~\ref{rmkmuFp}). Recall that $H^0(\GZip^\mu,\Vcal_I(\lambda))$ is identified with $H^0(\GF^{\mu},\Vcal_{\flag}(\lambda))$ by \eqref{ident-H0-GF}. Via the isomorphism $\GF^{\mu} \simeq [E'\backslash G]$ (see Section~\ref{subsec-zipflag}) and Equation~\eqref{globquot}, an element of the space $H^0(\GF^{\mu},\Vcal_{\flag}(\lambda))$ can be viewed as a function $f\colon G\to \AA^1$ satisfying
\begin{equation}\label{equ-Gzipflag-lambda}
    f(agb^{-1})=\lambda(a) f(g), \quad \forall (a,b)\in E', \ \forall g\in G.
\end{equation}
Recall that $\GF^{\mu}$ admits a stratification $(\Fcal_w)_{w\in W}$ (see Section~\ref{subsec-zipflag}), where $\Fcal_w\colonequals [E'\backslash F_w]$ and $F_w=BwBz^{-1}$ is the $B\times {}^z B$-orbit of $wz^{-1}$. The unique open stratum is $\Ucal_{\max}=\Fcal_{w_0}$. Also write  $U_{\max}\colonequals F_{w_0}=B w_0 Bz^{-1}$ (the $B\times {}^z B$-orbit of $w_0 z^{-1}=\sigma(w_{0,I})^{-1}$). The codimension~$1$ $B\times {}^z B$-orbits are the $F_{s_\alpha w_0}$ for $\alpha \in \Delta$. Define $\Ucal'_\mu\colonequals \pi^{-1}(\Ucal_\mu)\simeq [E'\backslash U_\mu]$.

\begin{lemma} \label{lemma-Umax} \ 
\begin{assertionlist}
    \item \label{lemma-Umax-item1} The stabilizer of $\sigma(w_{0,I})^{-1}$ in $B\times {}^z B$ is $S\colonequals \{(t,\sigma(w_{0,I}) t\sigma(w_{0,I})^{-1}) \ | \ t\in T\}$.
    \item \label{lemma-Umax-item2} The map $B_M\to U_{\max}$, $b\mapsto \sigma(w_{0,I})b^{-1}$ induces an isomorphism $[B_M/T]\simeq \Ucal_{\max}$, where $T$ acts on $B_M$ on the right by the action $B_M\times T\to B_M$, $(b,t)\mapsto \varphi(t)^{-1} b \sigma(w_{0,I}) t \sigma(w_{0,I})^{-1}$.
    \item \label{lemma-Umax-item3} Assume that $P$ is defined over $\FF_q$. Then $U_{\max}\subset U_{\mu}$ and\, $\Ucal_{\max} \subset \Ucal'_{\mu}$.
\end{assertionlist}
\end{lemma}

\begin{proof}
We prove~\eqref{lemma-Umax-item1}. Let $(x,y)\in B\times {}^zB$ be such that $x\sigma(w_{0,I})^{-1} y^{-1}=\sigma(w_{0,I})^{-1}$. Write $y=zy'z^{-1}$ with $y'\in B$. Since $z=\sigma(w_{0,I})w_0$, we obtain $x w_0 y'^{-1}w_0^{-1}\sigma(w_{0,I})^{-1}=\sigma(w_{0,I})^{-1}$, hence $x = w_0 y' w_0^{-1}$. It follows that $x\in B\cap w_{0}Bw_{0}^{-1}=T$. We can write $y=\sigma(w_{0,I}) x\sigma(w_{0,I})^{-1}$, which proves~\eqref{lemma-Umax-item1}.
To show~\eqref{lemma-Umax-item2}, note that the map $B\times {}^zB\to U_{\max}$, $(x,y)\mapsto x \sigma(w_{0,I}) y^{-1}$ induces an isomorphism $(B\times {}^zB)/S \to U_{\max}$, where $S$ is as in~\eqref{lemma-Umax-item1}. 
Hence $\Ucal_{\max}$ is isomorphic to $[E' \backslash B\times {}^z B / S]$. We have an isomorphism 
\begin{equation} \label{isomBM}
E'\backslash (B\times {}^z B) \lra B_M, \quad  E'\cdot (x,y)\longmapsto \varphi\left(\theta^P_L(x)\right)^{-1}\theta^Q_M(y),
\end{equation}
whose inverse is $B_M \to E'\backslash B\times {}^zB$, $b\mapsto E'\cdot (1,b)$. We identify $T$ and $S$ via the isomorphism $T\to S$, $t\mapsto (t,\sigma(w_{0,I}) t\sigma(w_{0,I})^{-1})$. The action of $S$ on $E'\backslash B\times {}^zB$ by multiplication on the right transforms via the isomorphism \eqref{isomBM} to the right action of $T$ defined by $B_M \times T \to B_M$, $(b,t)\mapsto \varphi(t)^{-1} b \sigma(w_{0,I}) t \sigma(w_{0,I})^{-1}$. This proves~\eqref{lemma-Umax-item2}. 
Finally, we show~\eqref{lemma-Umax-item3}. Assume that $P$ is defined over $\FF_q$. Then $U_{\mu}$ coincides with the unique open $P\times Q$-orbit by \cite[Corollary 2.15]{Wedhorn-bruhat}. 
Since $B\times {}^z B\subset P\times Q$, the set $U_{\mu}$ is a union of $B\times {}^z B$-orbits, hence contains $U_{\max}$. Since $\Ucal'_{\mu}=[E'\backslash U_{\mu}]$, we have $\Ucal_{\max} \subset \Ucal'_{\mu}$.
\end{proof}

For $\lambda\in X^*(T)$, let $S(\lambda)$ denote the space of functions $h\colon B_M\to \AA^1$ satisfying
\begin{equation}\label{equ-Slambda}
h\left(\varphi(t)^{-1}b\sigma(w_{0,I}) t\sigma(w_{0,I})^{-1}\right)=\lambda(t)^{-1} h(b), \quad \forall t\in T, \ \forall b\in B_M.
\end{equation}

\begin{corollary}
The isomorphism from item~\eqref{lemma-Umax-item2} of Lemma~\ref{lemma-Umax} induces an isomorphism
\begin{equation}
  \vartheta \colon  H^0\left(\Ucal_{\max},\Vcal_{\flag}(\lambda)\right)\lra S(\lambda).
\end{equation}
\end{corollary}
We describe this isomorphism explicitly. Let $f\in H^0(\Ucal_{\max},\Vcal_{\flag}(\lambda))$, 
viewed as a function $f\colon U_{\max}\to \AA^1$ satisfying Equation~\eqref{equ-Gzipflag-lambda}. The corresponding element $\vartheta(f)\in S(\lambda)$ is the function $B_M \to \AA^1$, $b\mapsto f(\sigma(w_{0,I})b^{-1})$. Conversely, if $h\colon B_M\to \AA^1$ is an element of $S(\lambda)$, the function $f=\vartheta^{-1}(h)$ is given by
\begin{equation}\label{equ-relationf}
f\left(b_1 \sigma(w_{0,I}) b_2^{-1}\right) = \lambda(b_1) h\left(\varphi\left(\theta^P_L(b_1)\right)^{-1} \theta^Q_M(b_2)\right), \quad (b_1,b_2)\in B\times {}^zB.
\end{equation}
By the property of $h$, the function $f$ is well defined.

In particular, given a section of $\Vcal_{\flag}(\lambda)$ over $\GF^{\mu}$, we can restrict it to the open substack $\Ucal_{\max}$ and then apply $\vartheta$ to obtain an element of $S(\lambda)$. Now assume  that $P$ is defined over $\FF_q$. In particular, we have $\sigma(w_{0,I})=w_{0,I}$ and $z=w_{0,I}w_0$. We also have $\Ucal_{\max}\subset \Ucal_{\mu}'$ (\cf Lemma~\ref{lemma-Umax}\eqref{lemma-Umax-item3} and inclusions
\begin{equation}\label{equ-injection-spaces}
H^0\left(\GF^{\mu},\Vcal_{\flag}(\lambda)\right)\subset H^0\left(\Ucal_{\mu}',\Vcal_{\flag}(\lambda)\right) \subset H^0\left(\Ucal_{\max},\Vcal_{\flag}(\lambda)\right).
\end{equation}
Write $S_{\flag}(\lambda) \subset S_{\mu}(\lambda) \subset S(\lambda)$, respectively, for the images of these three spaces under $\vartheta$. Choose a realization
$(u_{\alpha})_{\alpha \in \Phi}$ (see Section~\ref{subsec-notation}). For $\alpha\in \Delta$, define a map $\Gamma_\alpha \colon B_L\times \AA^1\to G$ by
\[
\Gamma_\alpha \colon (b,t)\longmapsto b \phi_\alpha(A(t))w_{0,I}, \quad \textrm{where } 
A(t) \colonequals \left(\begin{matrix}
t&1\\-1&0
\end{matrix} \right)\in \SL_2
\]
and $\phi_{\alpha}\colon \SL_2\to G$ is the map attached to $\alpha$. For $\alpha\in \Delta$, define an open subset 
\begin{equation}
G_\alpha \colonequals G\setminus \bigcup_{\substack{\beta\in \Delta\\ \beta\neq \alpha}} \overline{F}_{s_{\beta} w_0} =U_{\max}\cup F_{s_\alpha w_{0}}.    
\end{equation}
Since $U_\mu$ coincides with the open $P\times Q$-orbit, one sees that $G_\alpha\subset U_{\mu}$ if and only if $\alpha \in I$. 
In this setting, one has an analogue of \cite[Proposition 3.1.4]{Imai-Koskivirta-vector-bundles}, as follows. 

\begin{proposition}\label{psiadapted}
The following properties hold: 
\begin{assertionlist}
\item \label{item-imagepsi} The image of\, $\Gamma_\alpha$ is contained in $G_\alpha$.
\item \label{item-psit} For all $b\in B_L$ and $t\in \AA^1$, one has $\Gamma_\alpha(b,t)\in U_{\max}$ if and only if $t\neq 0$.
\end{assertionlist}
\end{proposition}

\begin{proof}
We have $U_{\max}=Bw_0Bz^{-1}=BB^+ w_{0,I}$. As in \cite[Equation~(3.1.3)]{Imai-Koskivirta-vector-bundles}, one has a decomposition
\begin{equation} \label{At}
A(t)=\left(\begin{matrix}
1&0\\-t^{-1}&1
\end{matrix} \right) \left(\begin{matrix}
t&1\\0&t^{-1}
\end{matrix} \right)=\left(\begin{matrix}
1&0\\-t^{-1}&1
\end{matrix} \right) \left(\begin{matrix}
t&0\\0&t^{-1}
\end{matrix} \right)\left(\begin{matrix}
1&t^{-1}\\0&1
\end{matrix} \right).
\end{equation}
Thus for $t\neq 0$, we have $\phi_\alpha(A(t))\in BB^+$, hence $\Gamma(b,t)\in U_{\max}$. For $t=0$, we have $\phi_{\alpha} (A(0))=s_\alpha$ and $\Gamma_{\alpha}(b,t)\in B s_{\alpha}w_{0,I} \subset B s_{\alpha}w_{0,I} {}^z B =F_{s_\alpha w_0}$. This shows~\eqref{item-imagepsi} and~\eqref{item-psit}.
\end{proof}

Let $f\in H^0(\Ucal_{\max},\Vcal_{\flag}(\lambda))$, viewed as a function $f\colon U_{\max}\to \AA^1$ satisfying Equation~\eqref{equ-Gzipflag-lambda}. Let $h \colonequals \vartheta(f)$ be the corresponding element of $S(\lambda)$. Using Equation~\eqref{equ-relationf},  for $\alpha \in \Delta^P$ and $(b,t)\in B_L\times \GG_{\mathrm{m}}$, we have 
\begin{align}
f\circ \Gamma_\alpha(b,t)&= f\left(b \phi_{\alpha}\left(\begin{matrix}
1&0\\-t^{-1}&1\end{matrix} \right)w_{0,I} \left(w_{0,I} \phi_{\alpha}\left(\begin{matrix}
t&1\\0&t^{-1}\end{matrix} \right)w_{0,I} \right)\right)\\
&= \lambda(b) \ h\left(\varphi (b)^{-1} 
w_{0,I} \alpha^{\vee}(t)^{-1} w_{0,I}\right).
\end{align}

Similarly, for $\alpha\in I$ and $(b,t)\in B_L\times \GG_{\mathrm{m}}$, one can show the following (we leave out the computation, since we will only need the case $\alpha\in \Delta^P$ in Section~\ref{sec-main-result}):
\begin{equation}
 f\circ \Gamma_\alpha(b,t)=   \lambda(b) \ h\left(\phi_{\sigma(\alpha)}\left(\begin{matrix}
1&0\\t^{-q}&1\end{matrix} \right) \varphi(b)^{-1} 
 \phi_{-w_{0,I}\alpha}\left(\begin{matrix}
t&0\\-1&t^{-1}\end{matrix} \right)\right).
\end{equation}

For $\alpha\in \Delta$, define a function $F_{h,\alpha}\colon B_L\times \GG_{\mathrm{m}}\to \AA^1$ by
\begin{align}
F_{h,\alpha}(b,t) &\colonequals h\left(\phi_{\sigma(\alpha)}\left(\begin{matrix}
1&0\\t^{-q}&1\end{matrix} \right) b
 \phi_{-w_{0,I}\alpha}\left(\begin{matrix}
t&0\\-1&t^{-1}\end{matrix} \right)\right) &\textrm{ if }\alpha\in I, \\
F_{h,\alpha}(b,t) &\colonequals h\left(b w_{0,I} \alpha^{\vee}(t)^{-1} w_{0,I}\right) &\textrm{ if }\alpha\in \Delta^P.
\end{align}
The function $F_{h,\alpha}(b,t)$ lies in $k[B_L][t,\frac{1}{t}]$, where $k[B_L]$ denotes the ring of functions of $B_L$. Moreover, $F_{h,\alpha}(b,t)\in k[B_L][t]$ if and only if $f\circ \Gamma_{\alpha}(b,t)$ extends to a map $B_L\times \AA^1\to G$. 

\begin{proposition}\label{prop-Slambda}
Let $h\in S(\lambda)$.
\begin{assertionlist}
\item \label{item-Sflcd} We have $h\in S_{\flag}(\lambda)$ if and only if\, $F_{h,\alpha}\in k[B_L][t]$ for all $\alpha\in \Delta$.
\item \label{item-Smucd} We have $h\in S_\mu(\lambda)$ if and only if\, $F_{h,\alpha}\in k[B_L][t]$ for all $\alpha\in I$.
\end{assertionlist}
\end{proposition}

\begin{proof}
Let $f=\vartheta^{-1}(h)\in H^0(\Ucal_{\max}, \Vcal_{\flag}(\lambda))$. In the terminology of \cite[Definition 3.2.1]{Koskivirta-automforms-GZip}, the map $\Gamma_\alpha$ is adapted to $f$ by \cite[Lemma 3.2.4]{Koskivirta-automforms-GZip}, 
because $f$ is an eigenfunction for the action of $E'$ and we have $E'\cdot \Gamma_\alpha(B_L\times \{0\})= F_{s_\alpha w_0}$ using $B\times {}^zB = E' (B_L\times \{1\})$. By \cite[Lemma 3.2.2]{Koskivirta-automforms-GZip}, the map $f$ extends to $G$ if and only if $f\circ \Gamma_{\alpha}$ extends to $B_L\times \AA^1$ for all $\alpha\in \Delta$, which shows (\ref{item-Sflcd}). Assertion (\ref{item-Smucd}) is proved similarly.
\end{proof}

\subsection{Main result} \label{sec-main-result}

We state the main result of this section, which is the reciprocal of Corollary~\ref{cor-Hassetype}. 

\begin{theorem}\label{main-thm-Hasse-type}
Let $(G,\mu)$ be a cocharacter datum of Hasse type. Then $\GF^\mu$ satisfies the Hasse property. Combining that with Corollary~\ref{cor-Hassetype}, we have 
\begin{equation}
(G,\mu) \textrm{ is of Hasse type } \ \Longleftrightarrow \ \Ccal_{\zipsf} = \Ccal_{\pha}.
\end{equation}
\end{theorem}

We prove Theorem~\ref{main-thm-Hasse-type} in the rest of this section. Fix a cocharacter datum $(G,\mu)$, with zip datum $\Zcal=\Zcal_\mu=(G,P,L,Q,M)$. For now, we only assume that $P$ is defined over $\FF_q$ (hence $L=M$). Also fix  a frame $(B,T,z)$ with $z=w_{0,I}w_0$.

\begin{proposition}[\textit{cf.} {\cite[Proposition XXII.5.5.1]{SGA3}}]
Let $G$ be a reductive group over $k$, and let $(B,T)$ be a Borel pair. Choose a total order on $\Phi^-$. The $k$-morphism
\begin{equation} \label{equ-map-gamma}
\gamma \colon T\times \prod_{\alpha\in \Phi^-}U_\alpha \lra G
\end{equation}
defined by taking the product with respect to the chosen order is a closed immersion with image $B$.
\end{proposition}
We apply this proposition to $(L,B_L)$. Choose an order on $\Phi^{-}_{L}$, and consider the corresponding map $\gamma$ as in Equation~\eqref{equ-map-gamma}, with image $B_L$. For a function $h\colon B_L\to \AA^1$, put $P_h \colonequals h\circ \gamma$. Via the isomorphism $u_{\alpha}\colon \GG_{a} \to U_{\alpha}$, we can view $P_h$ as a polynomial $P_h \in k[T][(x_\alpha)_{\alpha\in \Phi^{-}_{L}}]$, where the $x_\alpha$ are indeterminates indexed by $\Phi^{-}_{L}$. For $m=(m_\alpha)_\alpha \in \NN^{\Phi^{-}_{L}}$ and $\lambda\in X^*(T)$, denote by $P_{m,\lambda}$ the monomial
\begin{equation}\label{equ-monomial}
P_{m,\lambda}=\lambda(t)\prod_{\alpha \in \Phi^{-}_{L}} x_{\alpha}^{m_\alpha} \ \in \  k[T]\left[(x_\alpha)_{\alpha\in \Phi^{-}_{L}}\right].   
\end{equation}
We can write any element $P$ of $k[T][(x_\alpha)_{\alpha\in \Phi^{-}_{L}}]$ as a sum of monomials
\begin{equation}\label{equ-decomp-monomials}
    P=\sum_{i=1}^N c_{i} P_{m_i,\lambda_i}, 
\end{equation}
where for all $1\leq i \leq N$, we have $m_i\in \NN^{\Phi^{-}_{L}}$, $\lambda_i\in X^*(T)$ and $c_{i}\in k$. Furthermore, we may assume that the $(m_i,\lambda_i)$ are pairwise distinct. Under this assumption, the expression \eqref{equ-decomp-monomials} is uniquely determined (up to permutation of the indices). 
For $P\in k[T][(x_\alpha)_{\alpha\in \Phi^{-}_{L}}]$, write $h_P\colon B_L\to \AA^1$ for the function $P\circ \gamma^{-1}$. For $m=(m_\alpha)_\alpha \in \NN^{\Phi^{-}_{L}}$ and $\lambda\in X^*(T)$, write $h_{m,\lambda}\colonequals h_{P_{m,\lambda}}$.

\begin{lemma}\label{lemma-formula-hmlambda}
Let $(m,\lambda)\in \NN^{\Phi^{-}_{L}}\times X^*(T)$. For all $a\in T$ and $b\in B_L$, we have
\begin{align}
h_{m,\lambda}(ab)&=\lambda(a)h_{m,\lambda}(b), \\
h_{m,\lambda}(ba)&=\left(\lambda(a) \prod_{\alpha\in \Phi^{-}_{L}}\alpha(a)^{-m_{\alpha}}\right) h_{m,\lambda}(b).
\end{align}
\end{lemma}
\begin{proof}
The first formula is an immediate computation. For the second, let $b=\gamma(t,(u_\alpha(x_\alpha))_\alpha)$ with $t\in T$ and $(x_\alpha)_\alpha\in \GG_{\mathrm{a}}^{\Phi^{-}_{L}}$. Then 
\begin{align}
h_{m,\lambda}(ba)&=h_{m,\lambda}\left(ta\prod_{\alpha\in \Phi^{-}_{L}} a^{-1}u_\alpha(x_\alpha) a \right)= h_{m,\lambda}\left(ta\prod_{\alpha\in \Phi^{-}_{L}} u_\alpha(\alpha(a)^{-1} x_\alpha)\right) \\
&= \lambda(ta) \prod_{\alpha\in \Phi^{-}_{L}} (\alpha(a)^{-1} x_\alpha)^{m_\alpha} =\left(\lambda(a) \prod_{\alpha\in \Phi^{-}_{L}}\alpha(a)^{-m_{\alpha}}\right) h_{m,\lambda}(b),
\end{align}
where we used the formula $a^{-1} u_\alpha(x) a = u_\alpha(\alpha(a)^{-1}x)$ for all $x\in \AA^1$ and all $a\in T$.
\end{proof}

For $(m,\lambda)\in \NN^{\Phi^{-}_{L}}\times X^*(T)$ as above, define the weight $\omega(m,\lambda)$ as
\begin{equation}\label{omega-weight}
\omega(m,\lambda)\colonequals q \sigma^{-1}(\lambda)- w_{0,I}\lambda +\sum_{\beta \in \Phi^{-}_{L}} m_\alpha (w_{0,I}\beta) \in X^*(T).
\end{equation}
It follows immediately from Lemma~\ref{lemma-formula-hmlambda} that $h_{m,\lambda}\colon B_L\to \AA^1$ lies in $S(\omega(m,\lambda))$.

\begin{lemma}\label{lemma-same-weight}
Let $\lambda\in X^*(T)$ and $h\in k[B_L]$ be nonzero. Write $P_h=\sum_{i=1}^N c_i P_{m_i,\lambda_i}$ as in Equation~\eqref{equ-decomp-monomials}, with $(m_i,\lambda_i)$ pairwise disjoint and $c_i\neq 0$ for all $1\leq i \leq N$. Then we have
\begin{equation}
h\in S(\lambda) \Longleftrightarrow \omega(m_i,\lambda_i)=\lambda \textrm{ for all } i=1, \dots, N.
\end{equation}
\end{lemma}

\begin{proof}
  The implication ``$\Leftarrow$'' is obvious. Conversely, if $h\in S(\lambda)$, then for all $t\in T$, $b\in B$, we have
  $\lambda(t)h(b)=h(\varphi(t)bw_{0,I}t^{-1} w_{0,I})=\sum_{i=1}^N \omega(m_i,\lambda_i)(t) c_i  h_{m_i,\lambda_i}(b)$. 
 The result follows by the linear independence of characters.
\end{proof}

For $m\in \NN^{\Phi^{-}_{L}}$, $\lambda\in X^*(T)$ and $\alpha\in \Delta^P$, we write $F_{m,\lambda,\alpha}\colonequals F_{h_{m,\lambda},\alpha}$ (see Section~\ref{sec-max-flag-stratum}). For all $\alpha\in \Delta^P$ and all $(b,t)\in B_L\times \GG_m$, we find 
\begin{equation}\label{equ-Fmlambda}
F_{m,\lambda,\alpha}(b,t)=t^{-q\langle \lambda,\sigma\alpha^\vee \rangle + \langle \omega(m,\lambda),\alpha^\vee \rangle}h_{m,\lambda}(b).
\end{equation}
In particular, $F_{m,\lambda,\alpha}$ is in $k[B_L][t]$ if and only if $-q\langle \lambda, \sigma \alpha^\vee \rangle + \langle \omega(m,\lambda),\alpha^\vee \rangle \geq 0$. Using Equation~\eqref{omega-weight}, this inequality can also be written as 

\begin{equation}\label{equ-equivalent-ineq}
\left\langle w_{0,I}\lambda,\alpha^\vee \right\rangle \leq \sum_{\beta\in \Phi^{-}_{L}}m_\beta \left\langle w_{0,I}\beta,\alpha^\vee \right\rangle.
\end{equation}

\begin{corollary}\label{cor-monomials-DeltaP}
Let $\lambda\in X^*(T)$ and $h\in S(\lambda)$ be nonzero. Write $P_h=\sum_{i=1}^N c_{i} P_{m_i,\lambda_i}$ as in Equation~\eqref{equ-decomp-monomials}, with the $(m_i,\lambda_i)$ pairwise distinct and $c_i\neq 0$ for $1\leq i \leq N$. Let $\alpha\in \Delta^P$.
\begin{assertionlist}
\item \label{cor-monomials-DeltaP-item1} We have
\begin{align}
F_{h,\alpha}\in k[B_L][t] \ &\Longleftrightarrow \ \forall i=1, \dots ,N, \ F_{m_i,\lambda_i,\alpha}\in k[B_L][t].\\
&\Longleftrightarrow \ \forall i=1, \dots ,N, \ -q\langle \lambda_i,\sigma\alpha^\vee \rangle +\langle \lambda, \alpha^\vee \rangle\geq 0.
\end{align}
\item\label{cor-monomials-DeltaP-item2} Moreover, if\, $F_{h,\alpha}\in k[B_L][t]$, then $\langle w_{0,I}\lambda_i,\alpha^\vee \rangle \leq 0$ for all $1\leq i \leq N$.
\end{assertionlist}
\end{corollary}

\begin{proof}
By Equation~\eqref{equ-Fmlambda}, we have $F_{m_i,\lambda_i,\alpha}(b,t)=t^{d_i} h_{m_i,\lambda_i}(b)$ for some integer $d_i\in \ZZ$. 
Hence, the first equivalence of~\eqref{cor-monomials-DeltaP-item1} follows from the assumption that the $(m_i,\lambda_i)$ for $1\leq i \leq N$ are pairwise distinct. 
The second equivalence follows from the previous discussion, using $\omega(m_i,\lambda_i)=\lambda$ (see Lemma~\ref{lemma-same-weight}). Assertion \eqref{cor-monomials-DeltaP-item2} follows from the inequality \eqref{equ-equivalent-ineq} and the fact that $\langle w_{0,I}\beta , \alpha^\vee \rangle \leq 0$ for all $\beta\in \Phi^{-}_{L}$. Indeed, recall that $\langle \beta,\alpha^\vee \rangle \leq 0$ for any two distinct simple roots $\alpha,\beta\in \Delta$. Since $\beta\in \Phi^{-}_{L}$, we have $w_{0,I}\beta\in \Phi^{+}_{L}$; hence $w_{0,I}\beta$ is a sum of simple roots in $I$. Since $\alpha\in \Delta^P$, the result follows.
\end{proof}

We now study the partial Hasse invariant cone $C_{\pha}$ (see Definition~\ref{definition-CHasse}). Fix a positive integer $n$ such that $G$ is split over $\FF_{q^n}$. By inverting the map $h_\Zcal\colon \lambda \mapsto \lambda -qw_{0,I}(\sigma^{-1}\lambda)$, we can write $\Ccal_{\pha}$ as the set of $\lambda\in X^*(T)$ such that
\begin{equation}\label{equ-CHasse}
\sum_{i=0}^{2n-1} q^i \left\langle (w_{0,I})^i \sigma^{-i} \lambda,\alpha^\vee \right\rangle \leq 0, \quad \forall \alpha \in \Delta.
\end{equation}
For $\alpha \in \Delta$ and $\lambda\in X^*(T)$, define $K_{\alpha}(\lambda) \colonequals \sum_{i=0}^{2n-1} q^i \langle (w_{0,I})^i \sigma^{-i} \lambda,\alpha^\vee \rangle$.

\begin{lemma}\label{lemma-HT-CHasse}
Assume that $(G,\mu)$ is of Hasse type. For all $\lambda\in X_{+,I}^*(T)$ and $\alpha\in I$, we have $K_\alpha(\lambda)\leq 0$. In particular, we have 
\begin{equation}
  \Ccal_{\pha}=\left\{\lambda\in X_{+,I}^*(T) \relmiddle| \ \forall \alpha \in \Delta^P, \ K_\alpha(\lambda)\leq 0 \right\}.
\end{equation}
\end{lemma}

\begin{proof}
For all $\alpha\in I$, we have
\begin{align}
\sum_{i=0}^{2n-1} q^i \left\langle (w_{0,I})^i \sigma^{-i} \lambda,\alpha^\vee \right\rangle & = \sum_{i=0}^{2n-1} q^i \left\langle  \lambda,\sigma^i((w_{0,I})^i \alpha^\vee) \right\rangle = \sum_{i=0}^{2n-1} (-1)^i q^i \left\langle  \lambda, \alpha^\vee \right\rangle
=-\left\langle  \lambda, \alpha^\vee \right\rangle \left(\frac{q^{2n}-1}{q+1} \right) \leq 0,
\end{align}
where we used that $(G,\mu)$ is of Hasse type in the second equality and the fact that $\lambda$ is $I$-dominant in the last inequality. 
This shows the result.
\end{proof}

For example, if $P$ is a maximal parabolic, we have $|\Delta^P|=1$, so $\Ccal_{\pha}$ is given inside $X_{+,I}^*(T)$ by a single inequality. This is in contrast to cases that are not of Hasse type. For example, if $G=\Sp(6)$ as explained in \cite[Section~5.5]{Koskivirta-automforms-GZip}, the cone $\Ccal_{\pha}$ is defined by $|\Delta|=3$ inequalities inside $X_{+,I}^*(T)$.

From now on, assume that $(G,\mu)$ is of Hasse type. We prove Theorem~\ref{main-thm-Hasse-type} by showing that if $H^0(\GZip^\mu,\Vcal_I(\lambda))\neq 0$, then $\lambda\in \Ccal_{\pha}$. First, recall that $H^0(\GZip^\mu,\Vcal_I(\lambda))$ is identified with the space $H^0(\GF^\mu,\Vcal_{\flag}(\lambda))$, and also with $S_{\flag}(\lambda)\subset S(\lambda)$. Let $h\in S_{\flag}(\lambda)$ be nonzero. By Proposition~\ref{prop-Slambda}\eqref{item-Smucd}, $F_{h,\alpha}\in k[B_L][t]$ for all $\alpha\in \Delta$. We will only need this information for $\alpha\in \Delta^P$. Again, write  $P_h=\sum_{i=1}^N c_{i} P_{m_i,\lambda_i}$ as in Equation~\eqref{equ-decomp-monomials}, with the $(m_i,\lambda_i)$ pairwise distinct and $c_i\neq 0$ for $1\leq i \leq N$. By Lemma~\ref{lemma-same-weight} and Formula \eqref{omega-weight}, we have in particular
\begin{equation}\label{equ-lambda1}
\lambda= q \sigma^{-1}(\lambda_1)- w_{0,I}\lambda_1 +\sum_{\beta\in \Phi^{-}_{L}} m_{1,\beta} (w_{0,I}\beta).
\end{equation}
We want to show $\lambda\in \Ccal_{\pha}$, which amounts to $K_\alpha(\lambda)\leq 0$ for all $\alpha\in \Delta^P$ by Lemma~\ref{lemma-HT-CHasse}. We first compute $K_\alpha(\beta)$ for any $\beta\in \Phi_L$. We find 
\begin{equation}
K_{\alpha}(\beta)=\sum_{i=0}^{2n-1} q^i \left\langle (w_{0,I})^i \sigma^{-i} \beta,\alpha^\vee \right\rangle = \sum_{i=0}^{2n-1} (-1)^i q^i \left\langle \beta,\alpha^\vee \right\rangle =-\left\langle\beta,\alpha^\vee \right\rangle \left(\frac{q^{2n}-1}{q+1} \right).
\end{equation}
On the other hand, we have
\begin{align}
K_\alpha\left(q \sigma^{-1}(\lambda_1)-w_{0,I}\lambda_1\right)&=\sum_{i=0}^{2n-1} q^i \left\langle (w_{0,I})^i \sigma^{-i} \left( q \sigma^{-1}(\lambda_1)-w_{0,I}\lambda_1 \right),\alpha^\vee \right\rangle \\
&=\sum_{i=0}^{2n-1} q^{i+1} \left\langle (w_{0,I})^i \sigma^{-(i+1)} (\lambda_1),\alpha^\vee \right\rangle - \sum_{i=0}^{2n-1} q^i \left\langle (w_{0,I})^{i+1} \sigma^{-i} (\lambda_1 ),\alpha^\vee \right\rangle\\
&=(q^{2n}-1) \left\langle w_{0,I} \lambda_1,\alpha^\vee \right\rangle,
\end{align}
where we used that $\sigma^{2n}\lambda_1 =\lambda_1$. Hence, for all $\alpha\in \Delta^P$, we find 
\begin{align}
K_\alpha(\lambda)&=K_\alpha\left(q \sigma^{-1}(\lambda_1)-w_{0,I}\lambda_1\right) + \sum_{\beta\in \Phi^{-}_{L}}m_{1,\beta} K_\alpha(w_{0,I} \beta) \\
&=\frac{q^{2n}-1}{q+1}\left((q+1)\left\langle w_{0,I}\lambda_1, \alpha^\vee \right\rangle-\sum_{\beta\in \Phi^{-}_{L}}m_{1,\beta}\left\langle w_{0,I}\beta, \alpha^\vee \right\rangle \right).
\end{align}
One has $K_\alpha(\lambda)\leq 0$, using the fact that $\langle w_{0,I}\lambda_1,\alpha^\vee \rangle \leq 0$ (see Corollary~\ref{cor-monomials-DeltaP}\eqref{cor-monomials-DeltaP-item2}) and Equation \eqref{equ-equivalent-ineq} applied to $F_{m_1,\lambda_1}$. This concludes the proof of Theorem~\ref{main-thm-Hasse-type}.

\section{\texorpdfstring{$\boldsymbol{R_{\textrm{u}}(P_0)}$}{R\textunderscore u (P\textunderscore 0)}-invariant subspace{}}\label{sec5}

Let $(G,\mu)$ be an arbitrary cocharacter datum, and let $\Zcal_\mu=(G,P,L,Q,M)$ be the attached zip datum. Fix a frame $(B,T,z)$ with $z=\sigma(w_{0,I})w_0$. Let $(V,\rho)$ be an $L$-representation. For $f\in V^{L_\varphi}$, we can view $f$ as a section of $\Vcal_I(\lambda)$ over $\Ucal_\mu$, by Lemma~\ref{lem-Umu-sections}. When $P$ is defined over $\FF_q$, this section extends to $\GZip^\mu$ if and only if $f\in V(\lambda)^{\Delta^P}_{\geq 0}$, by Corollary~\ref{cor-Fq-Levi}. For general $P$, the condition on $f$ is given by the Brylinski--Kostant filtration on $V_I(\lambda)$ (see \cite[Theorem 3.4.1]{Imai-Koskivirta-vector-bundles}). Unfortunately, this condition is too complex to understand explicitly. However, let $P_0$ be the parabolic $L_0B$ with $L_0$ as in \eqref{eqL0}, and assume further that $f\in V_I(\lambda)^{R_{\textrm{u}}(P_0)}$. In this case, we can give a more explicit condition for when $f$ extends. In particular, the lowest-weight vector of $V_I(\lambda)$ satisfies this condition. This makes it possible to define a ``lowest-weight cone'' $\Ccal_{\lwsf}$ (see Section~\ref{sec-low} below) similar to the highest-weight cone $\Ccal_{\hwsf}$. When $P$ is not defined over $\FF_q$, one sees on examples that $\Ccal_{\hwsf}$ is usually very small. On the other hand, the lowest-weight cone will be quite large.

\subsection{Statement}
As in the proof of Proposition~\ref{prop-Norm},  for $\alpha\in \Delta^P$, define 
\begin{equation}\label{malpha-equ2}
 m_{\alpha} =\min \{ m \geq 1 \mid 
 \sigma^{-m}(\alpha) \notin I \}
\end{equation}
and $t_{\alpha}=t^{-1}\alpha(\varphi(\delta_{\alpha}(t)))^{-1} =t \alpha(\delta_{\alpha}(t))^{-1} \in t^{\QQ}$, where $t$ is an indeterminate. Also set 
\begin{equation}\label{utalpha}
 u_{t,\alpha}=\prod_{i=1}^{m_{\alpha}-1} \phi_{\sigma^{-i}(\alpha)} 
 \left( \begin{pmatrix}
1&-t_{\alpha}^{\frac{1}{q^i}}\\0&1
\end{pmatrix} \right).    
\end{equation}
For $\alpha\in \Phi$, write $G_\alpha\subset G$ for the image of the map $\phi_\alpha\colon \SL_2\to G$. For simplicity, we consider the following condition.

\begin{condition}\label{cond-commute}
For all $1\leq i, j \leq m_\alpha-1$ with $i\neq j$, we have $\langle \sigma^{-i}(\alpha), \sigma^{-j}(\alpha^\vee) \rangle=0$, and the subgroups $G_{\sigma^{-i}(\alpha)}$ and $G_{\sigma^{-j}(\alpha)}$ commute with each other.
\end{condition}

\begin{rmk}
Condition~\ref{cond-commute} is satisfied in many cases. For example, if $G$ splits over $\FF_{q^2}$, then $m_\alpha\in \{1,2\}$ and the condition is trivially satisfied. In particular, all absolutely simple unitary groups satisfy it. The condition also holds for $G=\Res_{\FF_{q^n}/\FF_q}(G_{0,\FF_{q^n}})$, where $G_0$ is a split reductive over $\FF_q$.
\end{rmk}

Let $(V,\rho)$ be an $L$-representation, and let $V=\bigoplus_{\nu \in X^*(T)}V_\nu$ be its $T$-weight decomposition. For $\alpha\in \Delta$, set $\delta_{\alpha}=\wp_*^{-1}(\alpha^{\vee})$ (where $\wp_*$ was defined in Equation~\eqref{equ-Plowstar}). Put $P_1 \colonequals \sigma^{-(m_\alpha-1)}(P)$. We have $\Delta^{P_1}=\sigma^{-(m_\alpha-1)}(\Delta^P)$. Since $P_0\subset P_1$, we have $\Delta^{P_1}\subset \Delta^{P_0}$. Define $V^{\Delta^{P_1}}_{\geq 0}$ similarly to Equation~\eqref{equ-VDeltaP} by 
\begin{equation}\label{equ-VDeltaP0}
V_{\geq 0}^{\Delta^{P_1}} = \bigoplus_{\langle \nu,\delta_\beta \rangle \geq 0, \ 
 \forall \beta\in \Delta^{P_1}} V_\nu.
\end{equation}

\begin{proposition}\label{prop-RuP0}
Assume that Condition~\ref{cond-commute} holds. Then we have
\begin{equation}
V^{R_{\textrm{u}}(P_0)} \cap V^{L_\varphi} \cap V^{\Delta^{P_1}}_{\geq 0} \subset H^0(\GZip^\mu,\Vcal(\rho)).
\end{equation}
\end{proposition}

\begin{proof}
Let $f\in V^{R_{\textrm{u}}(P_0)} \cap V^{L_\varphi} $, and let $\widetilde{f}\colon U_\mu\to V$ be the function corresponding to $f$ by Lemma~\ref{lem-Umu-sections}. It suffices to check that $\widetilde{f}$ extends to $G$. By the proof of \cite[Theorem 3.4.1]{Imai-Koskivirta-vector-bundles}, it is enough to show that for all $\alpha\in \Delta^P$, the function
\begin{equation}
    F_{\alpha}\colon t\longmapsto \rho \left( \phi_\alpha \left(\left(\begin{matrix}
    1 & 0 \\ -t^{-1} & 1
    \end{matrix} \right)\right)\delta_\alpha(t)u_{t,\alpha} \right)f
\end{equation}
lies in $k[t]\otimes V$. Since it lies in $k[t,t^{-1}]\otimes V$ by the proof of \cite[Theorem 3.4.1]{Imai-Koskivirta-vector-bundles}, it suffices to show that it also lies in $k[(t^r)_{r \in \mathbb{Q}_{\geq 0}}]\otimes V$. 
Since $\rho$ is trivial on $R_{\textrm{u}}(P)$ and $\alpha \in \Delta^P$, one has simply $F_\alpha(t)=\rho (\delta_\alpha(t)u_{t,\alpha})f$. Using Equation~\eqref{utalpha}, we can write
\begin{equation}
  F_\alpha(t)= \rho \left(\delta_\alpha(t) \prod_{i=1}^{m_{\alpha}-1} \phi_{\sigma^{-i}(\alpha)} 
 \left( \begin{pmatrix}
1&-t_{\alpha}^{\frac{1}{q^i}}\\0&1
\end{pmatrix} \right) \right)f  = \rho \left( \prod_{i=1}^{m_{\alpha}-1} \phi_{\sigma^{-i}(\alpha)} 
 \left( \begin{pmatrix}
1&\gamma_i\\0&1
\end{pmatrix} \right) \delta_\alpha(t) \right) f, 
\end{equation}
where $\gamma_i=-t^{\langle \sigma^{-i}(\alpha), \delta_\alpha\rangle } t_{\alpha}^{\frac{1}{q^i}}$. We have $q^{-1}\sigma^{-1}(\delta_\alpha)=\delta_\alpha+q^{-1}\sigma^{-1}\alpha^\vee$, and hence by induction we have $q^{-i}\sigma^{-i}(\delta_\alpha)=\delta_\alpha+(q^{-1}\sigma^{-1}\alpha^\vee + \dots +q^{-i}\sigma^{-i}\alpha^\vee)$. Let $1\leq i \leq m_\alpha-1$. By Condition~\ref{cond-commute}, we deduce $\langle \sigma^{-i}(\alpha),\delta_\alpha \rangle=q^{-i}(\langle \alpha, \delta_\alpha \rangle-2)$. Thus
\begin{equation}
\gamma_i=-t^{\langle \sigma^{-i}(\alpha), \delta_\alpha\rangle+q^{-i}(1-\langle \alpha,\delta_\alpha\rangle)}=-t^{-1/q^i}.
\end{equation}
Let $f=\sum_\nu f_\nu$ be the $T$-weight decomposition of $f$. By assumption, we have  
\begin{align*}
    F_\alpha(t)&=\sum_{\nu}\rho \left( \prod_{i=1}^{m_{\alpha}-1} \phi_{\sigma^{-i}(\alpha)} 
 \left( \begin{pmatrix}
1&-t^{-1/q^i}\\0&1
\end{pmatrix} \right) \delta_\alpha(t) \right) f_\nu \\ &=  \sum_{\nu} t^{\langle \nu, \delta_\alpha \rangle}\rho \left( \prod_{i=1}^{m_{\alpha}-1} \phi_{\sigma^{-i}(\alpha)} 
 \left( \begin{pmatrix}
1&-t^{-1/q^i}\\0&1
\end{pmatrix} \right) \right) f_\nu
\\ &=  \sum_{\nu} t^{\langle \nu, \delta_\alpha \rangle}\rho \left( \prod_{i=1}^{m_{\alpha}-1} \phi_{\sigma^{-i}(\alpha)} 
 \left( \begin{pmatrix}
t^{-1/q^i}&-1\\0&t^{1/q^i}
\end{pmatrix} \right) \sigma^{-i}(\alpha)^\vee(t^{1/q^i}) \right) f_\nu\\
 &=  \sum_{\nu} t^{\langle \nu, \delta_\alpha +\sum_{i=1}^{m_\alpha-1} q^{-i}\sigma^{-i}(\alpha)^{\vee} \rangle}\rho \left( \prod_{i=1}^{m_{\alpha}-1} \phi_{\sigma^{-i}(\alpha)} 
 \left( \begin{pmatrix}
t^{-1/q^i}&-1\\0&t^{1/q^i}
\end{pmatrix} \right)\right) f_\nu.
\end{align*}
As before, we have $\delta_\alpha+\sum_{i=1}^{m_\alpha-1} q^{-i}\sigma^{-i}(\alpha)^{\vee}=q^{-(m_\alpha-1)}\sigma^{-(m_\alpha-1)}(\delta_\alpha)$. Furthermore, we have
\begin{equation}
 \begin{pmatrix}
t^{-1/q^i}&-1\\0&t^{1/q^i}
\end{pmatrix} =  \begin{pmatrix}
0&-1\\1&t^{1/q^i}
\end{pmatrix}  \begin{pmatrix}
1&0\\-t^{-1/q^i}&1
\end{pmatrix}. 
\end{equation}
Since $P_0$ is defined over $\FF_q$, we have $\sigma^{-i}(\alpha)\notin I_{P_0}$ 
for all $i\in \ZZ$. Using the invariance of $f$ under $R_{\textrm{u}}(P_0)$, we deduce
\begin{equation}
    F_\alpha(t)= \sum_{\nu} t^{\langle \nu,\sigma^{-(m_\alpha-1)}(\delta_\alpha) \rangle / q^{m_\alpha-1} }\rho \left( \prod_{i=1}^{m_{\alpha}-1} \phi_{\sigma^{-i}(\alpha)} 
 \left( \begin{pmatrix}
0&-1\\1&t^{1/q^i}
\end{pmatrix} \right)\right) f_\nu.
\end{equation}
Since $f\in V^{\Delta^{P_1}}_{\geq 0} $, we have $\langle \nu,\sigma^{-(m_\alpha-1)}(\delta_\alpha) \rangle=\langle \nu,\delta_{\sigma^{-(m_\alpha-1)}(\alpha)} \rangle\geq 0$. Hence, the $t$-valuation of $F_\alpha(t)$ is $\geq 0$. The result follows.
\end{proof}

\subsection{Lowest-weight cone}\label{sec-low}

We examine the case $V=V_I(\lambda)$ for $\lambda\in X^*_{+,I}(T)$. The $L_0$-representation $V_I(\lambda)^{R_{\textrm{u}}(P_0)}$ is isomorphic to $V_{I_0}(w_{0,I_0}w_{0,I}\lambda)$ by \cite[Proposition 6.3.1]{Imai-Koskivirta-partial-Hasse}. Put $\lambda_0=w_{0,I_0}w_{0,I}\lambda$.

Let $f_{\low, \lambda} \in V_I(\lambda)$ be a nonzero element in the lowest-weight line of $V_I(\lambda)$. Consider the element $\Norm_{L_\varphi}(f_{\low, \lambda}) \in V_I(N_\varphi \lambda)$, defined in \eqref{norm-def-eq}, where $N_\varphi=|L_0(\FF_q)|q^m$. By construction, this element lies in $V_I(N_{\varphi}\lambda)^{L_\varphi}$. 
For $\alpha\in \Delta$, write $r_\alpha$ for the smallest integer $r\geq 1$ such that $\sigma^r(\alpha)=\alpha$.

\begin{theorem}\label{thm-norm-low}
Assume Condition~\ref{cond-commute}. Suppose that for all $\alpha \in \Delta^{P_0}$, one has
\begin{equation}\label{formula-norm-low}
\sum_{w\in W_{L_0}(\FF_q)} \sum_{i=0}^{r_\alpha-1} q^{i+\ell(w)} \ \langle w \lambda_0, \sigma^i(\alpha^\vee) \rangle\leq 0.
\end{equation}
Then $\Norm_{L_\varphi}(f_{\low,\lambda})$ extends to $\GZip^\mu$.
\end{theorem}

\begin{rmk}
Formulas \eqref{formula-norm-low} and \eqref{formula-norm} (in the case of $f_{\high, \lambda}$) differ in two aspects: $\lambda$ changes to $\lambda_0=w_{0,I_0}w_{0,I}\lambda$, and ``for all $\alpha\in \Delta^P$'' changes to ``for all $\alpha \in \Delta^{P_0}$.''
\end{rmk}

\begin{proof}
The lowest-weight vector $f_{\low, \lambda}$ is contained in the $L_0$-subrepresentation $V_I(\lambda)^{R_{\textrm{u}}(P_0)} \cong V_{I_0}(\lambda_0)$, 
which has highest weight $\lambda_0$, lowest-weight vector $f_{\low, \lambda}$ and highest-weight vector $f_{\high,\lambda_0}\colonequals w_{0, I_0}(f_{\low, \lambda})$. Since $w_{0,I_0}\in W_{L_0}(\FF_q)$, we have
\begin{equation}\label{norm-eq}
\Norm_{L_\varphi}\left(f_{\low,\lambda}\right)=\Norm_{L_\varphi}\left(f_{\high,\lambda_0}\right)=\Norm_{L_0(\FF_q)}\left(f_{\high,\lambda_0}\right)^{q^m}.
\end{equation}
Consider the zip datum $\Zcal_0=(G,P_0,L_0,Q_0,L_0)$, where $Q_0$ is the opposite parabolic to $P_0$ with Levi subgroup $L_0$. By Remark~\ref{rmk-opposite-cochar}, we have $\Zcal_0=\Zcal_{\mu_0}$ for some cocharacter $\mu_0\colon \GG_{\textrm{m},k}\to G_k$. Since $P_0$ is defined over $\FF_q$,  by Corollary~\ref{cor-Fq-Levi}, we have 
\begin{equation}
H^0\left(\GZip^{\mu_0},\Vcal_{I_0}(\lambda_0)\right)=V_{I_0}(\lambda_0)^{L_0(\FF_q)}\cap V_{I_0}(\lambda_0)^{\Delta^{P_0}}_{\geq 0}.   
\end{equation}
Applying Proposition~\ref{prop-Norm} to $\GZip^{\mu_0}$ and the $L_0$-representation $V_{I_0}(\lambda_0)$, we deduce 
\[\Norm_{L_0(\FF_q)}\left(f_{\high,\lambda_0}\right)\in V_{I_0}(N_0\lambda_0)^{\Delta^{P_0}}_{\geq 0}, 
\] 
where $N_0=|L_0(\FF_q)|$. Combining this with \eqref{norm-eq}, and using that $\Delta^{P_1}\subset \Delta^{P_0}$, we find
\begin{equation}
    \Norm_{L_\varphi}\left(f_{\low,\lambda}\right) \in  V_I\left(N_\varphi\lambda\right)_{\geq 0}^{\Delta^{P_0}} \subset V_I(N_\varphi\lambda)_{\geq 0}^{\Delta^{P_1}}.
\end{equation}
The result follows from Proposition~\ref{prop-RuP0} applied to $V_I(N_\varphi \lambda)$.
\end{proof}

\begin{definition}
Define $\Ccal_{\lwsf}$ as the set of $\lambda \in X^*_{+,I}(T)$ satisfying the inequalities \eqref{formula-norm-low}.
\end{definition}

We call $\Ccal_{\lwsf}$ the lowest-weight cone. Under Condition~\ref{cond-commute}, one has $\Ccal_{\lwsf}\subset \Ccal_{\zipsf}$ by Theorem~\ref{thm-norm-low}. We do not know if this inclusion holds in general. When $P$ is defined over $\FF_q$, one has $P_0=P$ and hence $\Ccal_{\lwsf}=\Ccal_{\hwsf}$.

\begin{lemma}\label{cor-CGS-Cond}
One has $\Ccal_{\GSsf}\subset  \Ccal_{\lwsf}$. 
\end{lemma}

\begin{proof}
For $\lambda \in \Ccal_{\GSsf}$, the character $w_{0,I}\lambda$ is anti-dominant. For all $w\in W_{L_0}(\FF_q)$, we have $\langle w \lambda_0, \sigma^i(\alpha^\vee) \rangle = \langle  w_{0,I}\lambda, w_{0,I_0} w^{-1}\sigma^i(\alpha^\vee) \rangle$. Since $w_{0,I_0} w^{-1}\in W_{L_0}$ and $\alpha \in \Delta^{P_0}$, the root $w_{0,I_0} w^{-1}\sigma^i(\alpha)$ is positive. Hence $\langle w \lambda_0, \sigma^i(\alpha^\vee) \rangle \leq 0$ for all $w\in W_{L_0}(\FF_q)$, and the result follows.
\end{proof}

In particular, if Condition~\ref{cond-commute} holds, we deduce $\Ccal_{\GSsf}\subset \Ccal_{\zipsf}$ from Lemma~\ref{cor-CGS-Cond}. We will prove this inclusion in the general case in the next section.

\section{Weil restriction}\label{sec-Weil}

When Condition~\ref{cond-commute} does not hold, we cannot use Proposition~\ref{prop-RuP0} to show $\Ccal_{\GSsf}\subset \Ccal_{\zipsf}$. We show here that a version of Proposition~\ref{prop-RuP0} holds in general (see Theorem~\ref{thm-Weilr} below). To eliminate the need for Condition~\ref{cond-commute}, we first study the case of a Weil restriction. More generally, we will prove a useful result that makes it possible to reduce certain questions pertaining to the cone $\Ccal_{\zipsf}$ to the case of a split group.

\subsection{Zip strata of a Weil restriction}

We recall some results from \cite[Section~4]{Koskivirta-Wedhorn-Hasse}. Note that \loccit uses the convention $B\subset Q$, whereas we assume $B\subset P$. We make the appropriate changes in this section. Let $r\geq 1$, and let $G_1$ be a connected, reductive group over $\FF_{q^r}$. Put $G=\Res_{\FF_{q^r}/\FF_q} G_1$. Over $k$, we can decompose
\begin{equation}
    G_k=G_{1,k}\times G_{2,k} \times \dots \times G_{r,k}, 
\end{equation}
where $G_i=\sigma^{i-1}(G_1)$. The Frobenius homomorphim $\varphi \colon G\to G$ maps the tuple $(x_1, \dots , x_r)\in G_k$ to the tuple $(\varphi(x_r),\varphi(x_1), \dots , \varphi(x_{r-1}))$. We choose a cocharacter $\mu\colon \GG_{\mathrm{m},k}\to G_k$ written as $(\mu_1, \dots , \mu_r)$ with $\mu_i\colon \GG_{\mathrm{m},k}\to G_{i,k}$. Consider the attached zip datum $(G,P,L,Q,M)$. Assume that there is a Borel pair $(B,T)$ defined over $\FF_q$ and $B\subset P$. For all $\square=P,L,Q,M,B,T$, one can decompose $\square=\prod_{i=1}^r \square_i$. Note that $\sigma(B_i) = B_{i+1}$ and $\sigma(T_i) = T_{i+1}$ and $\sigma(L_i)=M_{i+1}$, where indices are taken modulo $r$. Moreover, $\sigma(P_i)$ and $Q_{i+1}$ are opposite in $G_{i+1,k}$. Write $\Delta_i$ for the set of simple roots of $G_i$. The Weyl group $W:=W(G_k,T)$ also decomposes  as $W = W_{1} \times \cdots \times W_{r}$, where $W_i:=W(G_{i,k},T_i)$. Let $w_{0,i}$ be the longest element in $W_i$. The Frobenius induces an automorphism of $W$ again denoted by $\sigma$, and we have $\sigma(W_i) = W_{i+1}$. Similarly, we have ${}^I W={}^{I_1} W_1\times\cdots\times {}^{I_r} W_r$ and $W^J=W_1^{J_1}\times\cdots\times W_r^{J_r}$, 
where $I_{i}, J_i\subset \Delta_{i}$ are the types of the parabolic subgroups $P_{i}$ and $Q_i$, respectively.

We obtain a frame $(B,T,z)$ by setting $z\colonequals \sigma(w_{0,I}) w_0=w_0 w_{0,J}$ (see Lemma~\ref{lem-framemu}). Thus $z=(z_1,\dots ,z_r)$ with $z_i=w_{0,i} w_{0,J_i}$ for all $i=1,\dots ,r$. By the dual parametrization \eqref{dualorbparam} 
and the dimension formula for $E$-orbits in Theorem~\ref{thm-E-orb-param}, 
the $E$-orbits of codimension~$1$ in $G$ are
\begin{equation} \label{codim1-Weil}
C_{i,\alpha}:=E\cdot\left(1,\dots,1,w_{0,i} s_{\alpha} w_{0,i},1,\dots,1\right), \quad 1\leq i \leq r, \ \alpha \in \Delta_i \setminus J_i.
\end{equation}
For each $1\leq j \leq r$, define parabolic subgroups in $G_{j,k}$ by
\[
P'_{j}=\bigcap_{i=0}^{r-1}\sigma^{-i}(P_{i+j}) \quad\mbox{and}\quad Q'_{j}= \bigcap_{i=0}^{r-1}\sigma^{i}(Q_{j-i}), 
\]
where the indices are taken modulo $r$. The unique Levi subgroups of $P'_j$ and $Q'_j$ containing $T_j$ are, respectively, 
\[
L'_{j}=\bigcap_{i=0}^{r-1}\sigma^{-i}(L_{i+j}) \quad\mbox{and}\quad M'_{j}= \bigcap_{i=0}^{r-1}\sigma^{i}(M_{j-i}). 
\]
By \cite[Lemma 4.2.1]{Koskivirta-Wedhorn-Hasse}, the tuple $\Zcal_j \colonequals (G_j,P'_j,L'_j,Q'_j,M'_j,\varphi^r)$ is a zip datum over $\FF_{q^r}$. Clearly, $B_j\subset P'_j$ and $B_j^+ \subset Q'_j$, since $B$ is defined over $\FF_{q}$. It follows that $\sigma^r(P'_j)$ and $Q'_j$ are opposite parabolics of $G_{j,k}$. By Remark~\ref{rmk-opposite-cochar}, $\Zcal_j$ is of cocharacter type. We denote the zip group of $\Zcal_j$ by $E_j\subset P'_j \times Q'_j$ (in \cite{Koskivirta-Wedhorn-Hasse}, this group is denoted by $E'_j$, but we want to avoid confusion with the group $E'$ defined in Section~\ref{subsec-zipflag}).

Write $\iota_j\colon G_{j,k}\to G_k$ for the natural embedding $x\mapsto (1, \ldots , x, \ldots ,1)$. Denote by $\Xcal$ the set of $E$-orbits in $G_k$, and by $\Xcal_j\subset \Xcal$ the set of $E$-orbits that intersect $G_{j,k}$ (viewed as a subset of $G_k$ via $\iota_j$). We have the following result (\textit{cf.} \cite[Theorem 4.3.1]{Koskivirta-Wedhorn-Hasse}).

\begin{theorem}
\label{bijorbits}The map $C\mapsto C\cap G_{j,k}$ defines a
bijection between $\Xcal_{j}$ and the set of\, $E_{j}$-orbits in $G_{j,k}$. Furthermore,  one has $\codim_{G_k}(C) = \codim_{G_{j,k}}(C\cap G_{j,k})$ for all $C\in \Xcal_j$.
\end{theorem}

Note that $\Xcal_j$ always contains the open $E$-orbit, since this orbit contains $1\in G_k$. Furthermore, by Equation \eqref{codim1-Weil}, any $E$-orbit of codimension 1 lies in at least one of the $\Xcal_j$. There is a natural group homomorphism $\gamma_j\colon E_j \to E$, defined as follows. For $(x,y)\in E_j$, write $\overline{x}:= \theta_{L'_j}^{P'_j}(x)$ and set
\begin{equation}\label{injEprime}
\begin{aligned}
u_j(x,y) &:= (\varphi^{r-j+1}(\overline{x}),\dots ,\varphi^{r-1}(\overline{x}), x,\varphi(\overline{x}),\dots,\varphi^{r-j}(\overline{x}))\in P, \\
v_j(x,y) &:=(\varphi^{r-j+1}(\overline{x}),\dots ,\varphi^{r-1}(\overline{x}), y,\varphi(\overline{x}),\dots,\varphi^{r-j}(\overline{x}))\in Q,\\
\gamma_j(x,y)&:=(u_j(x,y),v_j(x,y)) \in E.
\end{aligned}
\end{equation}
The pair $(\iota_j, \gamma_j)$ induces a morphism of stacks
\begin{equation}
\theta_j \colon [E_j \backslash G_{j,k}] \lra [E \backslash G_k].
\end{equation}
By the previous discussion, the image of $\theta_j$ contains a nonempty open subset, 
and each codimension 1 stratum in $\GZip^\mu$ is contained in the image of at least one $\theta_j$. Note that $u_j(x,y)$ only depends on $x$. By abuse of notation, we denote again by $\gamma_j$ the map
\begin{equation}
   \gamma_j \colon P'_j \lra P, \quad x\longmapsto \left(\varphi^{r-j+1}(\overline{x}),\dots ,\varphi^{r-1}(\overline{x}), x,\varphi(\overline{x}),\dots,\varphi^{r-j}(\overline{x})\right).
\end{equation}
We have a commutative diagram
\begin{equation}
\xymatrix@1@M=7pt{
E_j \ar[d]_{\pr_1} \ar[r]^{\gamma_j} & E \ar[d]^{\pr_1} \\
P'_j \ar[r]_{\gamma_j} & P\rlap{.}
}
\end{equation}
For $x\in L'_j$, we have $\gamma_j(x)\in L$. Hence, we also have a map $\gamma_j\colon L'_j\to L$.

\subsection{Space of global sections}
For each $1\leq i \leq r$, let $(V_i,\rho_i)$ be an $L_i$-representation, and let $(V,\rho)$ be the $L$-representation $\textstyle \sigmaop{}_{i=1}^r \rho_i$. For example, if $\lambda=(\lambda_1, \dots , \lambda_r)$ is in $X^*(T)=X^*(T_1)\times \dots \times X^*(T_r)$, then we have $V_I(\lambda)=\textstyle \sigmaop{}_{i=1}^r V_{I_i}(\lambda_i)$. View $\rho_i$ as a map $P_i\to \GL(V_i)$ trivial on $R_{\textrm{u}}(P_i)$. Using the maps $\gamma_j\colon P'_j \to P$, we have
\begin{equation}
    \theta_j^* \left( \Vcal(\rho) \right) = \bigotimes_{i=1}^r \Vcal\left(\rho_{j+i}^{[i]}\right), 
\end{equation}
where $\rho_{j+i}^{[i]}$ denotes the $P'_j$-representation $P'_j \xrightarrow{\varphi^i} P_{j+i} \xrightarrow{\rho_{j+i}} \GL(V_{j+i})$ (indices modulo $r$). By the definition of~$P'_j$, this composition is well defined. Note that $\rho_{j+i}^{[i]}$ may not be trivial on the unipotent radical of $P'_j$. Let $L_\varphi$ be the stabilizer of $1\in G$ in $E$, as defined in Section~\ref{subsec-global-sections}, and fix $f\in V^{L_\varphi}$. By Lemma~\ref{lem-Umu-sections}, we may view $f$ as a section of $\Vcal(\rho)$ over the open substack $\Ucal_\mu \subset \GZip^\mu$.  Similarly, since $\theta_j$ maps $\Ucal_{\mu_j}$ into $\Ucal_\mu$ (see Theorem~\ref{bijorbits}), we have $\theta_j^*(f) \in H^0(\Ucal_{\mu_j}, \theta_j^*(\Vcal(\rho)))$.

\begin{lemma}\label{thetaj-extends}
The section $f$ extends to $\GZip^\mu$ if and only if\, $\theta_j^*(f)$ extends to $\GjZip^{\Zcal_j}$ for all $1\leq j \leq r$.
\end{lemma}

\begin{proof}
The ``only if''  implication is clear. Conversely, assume that $\theta_j^*(f) \in H^0(\GZip^{\Zcal_j}, \theta_j^*(\Vcal(\rho)))$ for all $1\leq j \leq r$. 
Viewing $f$ as a section over $\Ucal_\mu$, consider the unique regular map $\widetilde{f}\colon U_\mu \to V$ satisfying $\widetilde{f}(1)=f$ and $\widetilde{f}(axb^{-1})=\rho(a)\widetilde{f}(x)$ for all $x\in U_{\mu}$ and all $(a,b)\in E$. It suffices to show that the map $\widetilde{f}$ extends to a regular map $\widetilde{f}\colon G\to V$ (by density, this regular map will automatically satisfy the $E$-equivariance condition).

Consider a codimension~$1$ $E$-orbit $C_{i,\alpha}$ for some $1\leq i \leq r$ and $\alpha\in \Delta_i\setminus J_i$ (where $C_{i,\alpha}$ was defined in Equation \eqref{codim1-Weil}). Set $Y\colonequals U_\mu \cup C_{i,\alpha}$. It is the complement in $G$ of the union of the Zariski closures of all other codimension~$1$ $E$-orbits. In particular, $Y$ is open in $G$. Define $X\colonequals \iota_i^{-1}(Y)$, and consider the map $\iota_i\colon X\to Y$. This map satisfies conditions~\eqref{le-1} and~\eqref{le-2} of Lemma~\ref{lem-extends} below (for the group $H=E$). By assumption, the function $\iota_i^*(\widetilde{f})=\widetilde{f} \circ \iota_i \colon U_{\mu_i} \to V$ extends to a function $G_i\to V$ (in particular, to a map $X\to V$). Therefore, we can apply Lemma~\ref{lem-extends} to deduce that $\widetilde{f}$ extends to a regular map $Y\to V$. To show that $\widetilde{f}$ extends to $G$, let $\widetilde{f}_0\colon U_\mu\to \AA^1$ be a coordinate function of $f$ in some basis of $V$. By the above discussion, $\widetilde{f}_0$ cannot have a pole along any codimension~$1$ $E$-orbit of $G$, hence extends to $G$ by normality. Hence $\widetilde{f}$ itself extends to $G$, and the result follows.
\end{proof}

\begin{lemma}\label{lem-extends}
Let $Y,X$ be irreducible normal $k$-varieties, and assume that $Y$ is endowed with an action of an algebraic group $H$. Suppose that $Y$ has an open subset $U_Y\subset Y$ stable by $H$. Set $Z_Y \colonequals Y\setminus U_Y$. Let $(V,\rho)$ be an $H$-representation, and let $f\colon U_Y\to V$ be an $H$-equivariant regular map on $U_Y$. Let $\iota\colon X\to Y$ be a regular map satisfying the following:
\begin{assertionlist}
\item\label{le-1} $\iota(X)\cap U_Y \neq \emptyset$,
\item\label{le-2} $H\cdot(\iota(X)\cap Z_Y)$ is Zariski dense in $Z_Y$.
\end{assertionlist}
Define $U_X\colonequals \iota^{-1}(U_Y)$. Then the morphism $f$ extends to an $H$-equivariant regular map $Y\to V$ if and only if $\iota^*(f)\colon U_X\to V$ extends to a regular map $X\to V$.
\end{lemma}

\begin{proof}
The ``only if'' direction is obvious. Conversely, assume that $\iota^*(f)\colon U_X\to V$ extends to a regular map $X\to V$. Consider the map
\begin{equation*}
    \phi \colon H\times X\lra Y, \quad (h,x)\longmapsto h\cdot \iota(x).
\end{equation*}
We have $\phi^{-1}(U_Y)=H\times U_X$. Then $f$ extends to a regular map $Y\to V$ if and only if $\phi^*(f)\colon H\times U_X\to V$ extends to  a regular map $H\times X\to V$. Indeed, choose a basis of $V$. Let $f_i \colon U_Y \to \mathbb{A}_k^1$ for $1 \leq i \leq \dim V$ be coordinate maps of $f$ with respect to that basis. Since the image of $\phi$ is dense in $Z_Y$ by assumption, $f_i$ cannot have a pole along $Z_Y$, hence extends to $Y$ by normality. Thus, it suffices to show that if $\iota^*(f)$ extends, then so does $\phi^*(f)$. But since $f$ is $H$-equivariant,  for all $h\in H$, $x\in U_X$, we have 
\begin{equation*}
    \phi^*(f)(h,x)=f(h\cdot \iota(x))=h\cdot \left( \iota^*(f)(x)\right).
\end{equation*}
Hence if $\iota^*(f)$ extends to $X$, we can define a function $H\times X\to V$ using the above formula, and it must coincide with $\phi^*(f)$ on the open subset $H\times U_X$. The result follows.
\end{proof}

Now, assume that for all $1\leq j \leq r$, $P_j$ is defined over $\FF_{q^r}$ (for example, this is the case if $T_1$ is split over $\FF_{q^r}$). It is clear that $P'_j$ is then also defined over $\FF_{q^r}$. We apply Corollary~\ref{cor-Fq-Levi} to the $\FF_{q^r}$-zip datum $\Zcal_j$. We deduce that for any $L'_j$-representation $(W,\rho_W)$, we have
\begin{equation}\label{form-Pprime}
    H^0(\GjZip^{\Zcal_j}, \Vcal(\rho_W)) = W^{L'_j(\FF_{q^r})}\cap W^{\Delta^{P'_j}}_{\geq 0}.
\end{equation}
However, since $\gamma_j^*(\rho)=\rho \circ \gamma_j \in \Rep(P'_j)$ may be nontrivial on $R_{\textrm{u}}(P'_j)$, we cannot apply this formula directly to $\gamma_j^*(\rho)$. Denote by $V^{\#}\subset V$ the subspace of $f\in V$ that are invariant under $\gamma_j(R_{\textrm{u}}(P'_j))$ for all $1\leq j \leq r$. We deduce the following from Equation~\eqref{form-Pprime} and Lemma~\ref{thetaj-extends}. 

\begin{corollary}
Let $f\in V^{L_\varphi}\cap V^{\#}$. Then $f$ extends to $\GZip^\mu$ if and only if $f\in \left( V|_{L'_j} \right)^{\Delta^{P'_j}}_{\geq 0}$ for all $1\leq j \leq r$, where $V|_{L'_j}$ denotes the $L'_j$-representation $\gamma_j^*(\rho) \colon L'_j \xrightarrow{\gamma_j} L \xrightarrow{\rho} \GL(V)$.
\end{corollary}

Write $V=\bigoplus_{\chi\in X^*(T)} V_\chi$ for the $T$-weight space decomposition of $V$, and write $\chi=(\chi_1, \dots , \chi_r)$, where $\chi_i \in X^*(T_i)$. Similarly, let $f=\sum_{\chi} f_\chi$ be the decomposition of $f$. We determine the $T_j$-weight decomposition of $V|_{L'_j}$. For $\chi \in X^*(T)$, define 
\[S_j(\chi)\colonequals \sum_{i=0}^{r-1} q^i \sigma^{-i}(\chi_{j+i}) \in X^*(T_j)\]
(indices taken modulo $r$). Then, the $T_j$-weight decomposition of $V|_{L'_j}$ is given by
\begin{equation}
    V|_{L'_j} = \bigoplus_{\eta \in X^*(T_j)} V_\eta, \quad \textrm{where }  V_\eta = \bigoplus_{\substack{ \chi \in X^*(T) \\ S_j(\chi)=\eta }} V_\chi.
\end{equation}
Define $V^{\cap}_{\geq 0}\subset V$ as the intersection of all $\left( V|_{L'_j} \right)^{\Delta^{P'_j}}_{\geq 0}$ for $1\leq j \leq r$ inside $V$. Put
\begin{equation}
 \wp^{(r)}_{j,*} \colon X_*(T_j)_\RR \lra X_*(T_j)_\RR, \quad \delta \longmapsto \delta-q^r\sigma^{r}(\delta)
\end{equation}
as in Equation~\eqref{equ-Plowstar} (but changing $\varphi$ to $\varphi^r$). For $\alpha \in \Delta_j$, define $\delta^{(r)}_{j,\alpha}\colonequals(\wp^{(r)}_{j,*})^{-1}(\alpha^\vee)\in X_*(T_j)_{\RR}$. By definition, $( V|_{L'_j})^{\Delta^{P'_j}}_{\geq 0}$ is the direct sum of the $V_{\eta}$ for $\eta\in X^*(T_j)$ satisfying $\langle \eta, \delta^{(r)}_{j,\alpha} \rangle\geq 0$ for all $\alpha \in \Delta^{P'_j}$. Hence $V^{\cap}_{\geq 0}\subset V$ is the direct sum of the weight spaces $V_\chi$ satisfying $\langle S_j(\chi), \delta^{(r)}_{j,\alpha} \rangle \geq 0$ for all $\alpha \in \Delta^{P'_j}$ and all $1\leq j \leq r$. We have shown that $f$ extends to $\GZip^\mu$ if and only if $f\in V^{\cap}_{\geq 0}$. In other words, we have  shown the following. 

\begin{proposition} \label{glob-sec-Weil}
Let $\Gamma(\rho)$ be the set of all $\chi \in X^*(T)$ such that $\langle S_j(\chi), \delta^{(r)}_{j,\alpha} \rangle \geq 0 $ for all $1\leq j\leq r$ and all $\alpha\in \Delta^{P'_j}$. For $f\in V^{L_\varphi}\cap V^{\#}$, $f$ extends to $\GZip^\mu$ if and only if $f\in V^{\cap}_{\geq 0}=\bigoplus_{\chi\in \Gamma(\rho)} V_\chi$.
\end{proposition}

Now, assume that $T_1$ is split over $\FF_{q^r}$. 
Then for all $1\leq j \leq r$, $T_j$ is split over $\FF_{q^r}$; hence $\delta^{(r)}_{j,\alpha} = -\frac{1}{q^r-1} \alpha^\vee$ for all $\alpha \in \Delta_j$. Therefore, in this case, $\Gamma(\rho)$ is the set of $\chi\in X^*(T)$ satisfying $\langle S_j(\chi), \alpha^\vee \rangle \leq 0$ for all $\alpha \in \Delta^{P'_j}$ and all $1\leq j \leq r$.

\subsection{Consequence for $\boldsymbol{H^0(\GZip^\mu, \Vcal(\rho))}$} \label{subseq-Weyl-conseq}

We derive consequences from the above considerations. Let $G$ be a connected, reductive group over $\FF_q$, $\mu\colon \GG_{\mathrm{m},k}\to G_k$ a cocharacter and $\Zcal=(G,P,L,Q,M)$ the associated zip datum over $\FF_q$. Choose a frame $(B,T,z)$ as in Section~\ref{sec-frames}. For $r\geq 1$, consider the diagonal embedding
\begin{equation}
    \Delta \colon G \lra \widetilde{G} \colonequals \Res_{\FF_{q^r}/\FF_q}\left(G_{\FF_{q^r}}\right).
\end{equation}
The cocharacter $\widetilde{\mu} \colonequals \Delta\circ \mu$ induces a zip datum $\widetilde{\Zcal}=(\widetilde{G},\widetilde{P},\widetilde{L},\widetilde{Q},\widetilde{M},\widetilde{\varphi})$, where for each $\square=G,P,L,Q,M$ we have $\widetilde{\square}_k=\square_k \times \dots \times \square_k$. Write $\widetilde{E}$ for the zip group of $\widetilde{\Zcal}$. We obtain a morphism of stacks
\begin{equation}
\Delta \colon \GZip^\mu \lra \GtilZip^{\widetilde{\mu}}. 
\end{equation}
For all $1\leq i \leq r$, let $(V_i, \rho_i)$ be an $L$-representation, and write $\widetilde{\rho} \colonequals \textstyle \sigmaop{}_{i=1}^r \rho_i$, viewed as an $\widetilde{L}$-representation. We have
\begin{equation}
 \Delta^*(\Vcal(\widetilde{\rho}))= \bigotimes_{i=1}^r \Vcal(\rho_i).
\end{equation}
Since $\Delta\colon G \to \widetilde{G}$ is a group homomorphism, it satisfies $\Delta(1)=1$; hence the induced map $\Delta$ from $\GZip^\mu$ to $\GtilZip^{\widetilde{\mu}}$ is dominant ($1$ lies in the open zip stratum). Therefore, pullback via $\Delta$ induces an injection on the spaces of global sections:
\begin{equation}
    \Delta^* \colon H^0\left(\GtilZip^{\widetilde{\mu}},\Vcal(\widetilde{\rho})\right) \lra H^0\left(\GZip^\mu, \bigotimes_{i=1}^r \Vcal(\rho_i)\right).
\end{equation}
In particular, let $(V,\rho)$ be an $L$-representation, and let $\rho_0\colon L\to \{1\}$ be the trivial character of $L$. Put $\rho_1 = \rho$ and $\rho_i=\rho_0$ for all $2\leq i \leq r$. We obtain an injection
\begin{equation}\label{Deltastar}
\Delta^* \colon  H^0\left(\GtilZip^{\widetilde{\mu}}, \Vcal(\pr_1^*\rho)\right) \lra H^0(\GZip^\mu, \Vcal(\rho)), 
\end{equation}
where $\pr_1\colon \widetilde{L}\to L$ is the first projection and $\pr_1^*\rho$ is the $\widetilde{L}$-representation $\rho \circ \pr_1$. Fix an $r\geq 1$ such that~$P$ is defined over $\FF_{q^r}$. We apply Proposition~\ref{glob-sec-Weil} to $\pr_1^*\rho$. In this case, for each $1\leq j \leq r$, the parabolic subgroup $P'_j$ is equal to $P_0=\bigcap_{i\in \ZZ} \sigma^i(P)$, the largest parabolic subgroup defined over $\FF_q$ contained in $P$. Let $L_0\subset P_0$ be the Levi subgroup containing $T$, as in Equation~\eqref{eqL0}. The space $V^{\#}$ is clearly $V^{R_{\textrm{u}}(P_0)}$. Any weight of the $\widetilde{T}$-representation $\pr_1^*\rho$ is of the form $\widetilde{\chi} = (\chi,0, \dots , 0)$, where $\chi$ is a $T$-weight of $V$. Hence, for each $1\leq j \leq r$, we have $S_j(\widetilde{\chi})=q^{r-j+1}\sigma^{-(r-j+1)}\chi$. Thus, $V_{\geq 0}^{\cap}$ is the direct sum of $T$-weight spaces $V_\chi$ satisfying $\langle \sigma^{-(r-j+1)}\chi, \delta^{(r)}_{\alpha} \rangle \leq 0$ for all $\alpha \in \Delta^{P_0}$ and all $1\leq j \leq r$ (here $\delta^{(r)}_{j,\alpha}$ is independent of $j$, so we denote it simply by $\delta^{(r)}_{\alpha}$). But since $P_0$ is defined over $\FF_q$, this condition is also equivalent to $\langle \chi, \delta^{(r)}_\alpha \rangle \leq 0$ for all $\alpha \in \Delta^{P_0}$. Note that $V^{\cap}_{\geq 0}$ is very close to the space $V^{\Delta^{P_0}}_{\geq 0}$, the only difference being that $\delta_\alpha$ is replaced by $\delta^{(r)}_\alpha$ in the definition. In other words, we could say that $V^{\cap}_{\geq 0}= V^{\Delta^{P_0\otimes \FF_{q^r}}}_{\geq 0}$, where we changed $P_0$ to $P_{0}\otimes \FF_{q^r}$. To simplify notation, for any $L$-representation $(V,\rho)$,  define
\begin{equation}\label{equ-VDeltarP0}
V^{\Delta^{P_0},(r)}_{\geq 0}\colonequals  \bigoplus_{\substack{\langle \nu,\delta^{(r)}_{\alpha} \rangle \geq 0, \ 
 \forall \alpha\in \Delta^{P_0}}} V_\nu.
\end{equation}
We showed that $V^{\cap}_{\geq 0}=V^{\Delta^{P_0},(r)}_{\geq 0}$. Denote by $L_{\varphi}^{(r)}$ the image of $\Stab_{\widetilde{E}}(1)$ via the composition of the projection $\widetilde{E}\to \widetilde{P}$ and the first projection $\pr_1\colon \widetilde{P}\to P$. By Lemma~\ref{lemLphi}, we have $L_{\varphi}^{(r)}\subset L$. 
From Proposition~\ref{glob-sec-Weil}, we deduce 
\begin{equation}\label{equ-Gtil-pr1}
    V^{L_{\varphi}^{(r)}} \cap V^{\Delta^{P_0},(r)}_{\geq 0} \cap V^{R_{\textrm{u}}(P_0)} \subset H^0\left(\GtilZip^{\widetilde{\mu}},\Vcal(\pr_1^*(\rho)\right).
\end{equation}
The largest Levi subgroup of $\widetilde{G}$ defined over $\FF_q$ contained in $\widetilde{L}$ is $\widetilde{L}_0 \colonequals \Res_{\FF_{q^r}/\FF_q} L_0$. Since $\widetilde{L}_0(\FF_q) = L_0(\FF_{q^r})$, we have $L_{\varphi}^{(r)} = L_{\varphi}^{(r),\circ}\rtimes L_0(\FF_{q^r})$ by Lemma~\ref{lemLphi}. 
Furthermore, $\Delta$ induces an injection $\Delta\colon L_\varphi \to L_{\varphi}^{(r)}$. Combining Equations~\eqref{equ-Gtil-pr1} and \eqref{Deltastar}, we deduce the following. 

\begin{theorem} \label{thm-Weilr}
Let $r\geq 1$ be such that $P$ is defined over $\FF_{q^r}$. One has
\begin{equation} \label{thm-Weilr-eq}
    V^{L_{\varphi}^{(r)}} \cap V^{\Delta^{P_0},(r)}_{\geq 0}\cap V^{R_{\textrm{u}}(P_0)} \subset H^0(\GZip^\mu, \Vcal(\rho)).
\end{equation}
\end{theorem}
This theorem is slightly weaker than Proposition~\ref{prop-RuP0}, but it holds in general, independently of Condition~\ref{cond-commute}. Put $V_{\Weil}^{(r)}\colonequals V^{L_{\varphi}^{(r)}} \cap V^{\Delta^{P_0},(r)}_{\geq 0}\cap V^{R_{\textrm{u}}(P_0)}$.

\subsection{Applications to $\boldsymbol{C_{\zipsf}}$}

Consider the $L$-representation $V=V_I(\lambda)$ for $\lambda\in X_{+,I}^*(T)$. Let $r\geq 1$ be such that $P$ is defined over $\FF_{q^r}$. Consider the sub-$L_0$-representation $V_{I_0}(\lambda_0)\subset V_I(\lambda)$ with $\lambda_0\colonequals w_{0,I_0}w_{0,I}\lambda$. Then, we have $V^{R_{\textrm{u}}(P_0)}=V_{I_0}(\lambda_0)$. Let $Q_0$ be the opposite parabolic to $P_0$ with Levi subgroup $L_0$. Let $\mu_0\colon \GG_{\textrm{m},k}\to G_k$ be any dominant cocharacter with centralizer $L_0$ (hence $\mu_0$ defines the parabolics $P_0$, $Q_0$). If we base change $G$ to $\FF_{q^r}$, then by Corollary~\ref{cor-Fq-Levi}, we have 
\begin{align}
    H^0(\GFqrZip^{\mu_0},\Vcal_{I_0}(\lambda_0)) &= V_{I_0}(\lambda_0)^{L_0(\FF_{q^r})}\cap V_{I_0}(\lambda_0)^{\Delta^{P_0},(r)}_{\geq 0} \\ &=V^{L_0(\FF_{q^r})} \cap V^{\Delta^{P_0},(r)}_{\geq 0}\cap V^{R_{\textrm{u}}(P_0)}. \label{H0Fqr}
\end{align}
Hence, the space $V_{\Weil}^{(r)}$ given in Equation~\eqref{thm-Weilr-eq} is very close to the space in Equation~\eqref{H0Fqr}. The only difference is that we take invariants under $L_{\varphi}^{(r)} = L_{\varphi}^{(r),\circ}\rtimes L_0(\FF_{q^r})$ instead of $L_0(\FF_{q^r})$.

Fix $m\geq 1$ such that the finite unipotent group $L_{\varphi}^{(r),\circ}$ is annihilated by $\varphi^m$. 
If $f\in H^0(\GFqrZip^{\mu_0},\Vcal_{I_0}(\lambda_0))$, then $f^{q^m}$ is stable by $L_{\varphi}^{(r)}$, and hence lies in $V_I(q^m\lambda)_{\Weil}^{(r)}$. 
We deduce the following: Assume that $\lambda\in X_{+,I}^*(T)$ satisfies $\lambda_0\in C_{\zipsf}(G_{\FF_{q^r}},\mu_0)$, where $C_{\zipsf}(G_{\FF_{q^r}},\mu_0)$ is the zip cone of the zip datum $(G_{\FF_{q^r}},\mu_0)$. Then $\lambda\in \Ccal_{\zipsf}$. We have shown the following.

\begin{theorem}\label{thm-GFqr-cone}
Assume that $P$ is defined over $\FF_{q^r}$. Then
\begin{equation}
  X_{+,I}^*(T) \cap \left(w_{0,I} w_{0,I_0} \Ccal_{\zipsf}(G_{\FF_{q^r}},\mu_0) \right) \subset \Ccal_{\zipsf}.
\end{equation}
\end{theorem}

\begin{rmk}\label{rmk-reduce-split}
    We can apply all results and constructions about the zip cone to $(G_{\FF_{q^r}},\mu_0)$.  For example, consider the highest-weight cone of $(G_{\FF_{q^r}},\mu_0)$. We deduce from Theorem~\ref{thm-GFqr-cone} and Proposition~\ref{prop-Norm} that if $\lambda\in X_{+,I}^*(T)$ satisfies
\begin{equation}\label{formula-norm-Weil}
\sum_{w\in W_{L_0}(\FF_{q})} q^{r\ell(w)} \ \langle w\lambda_0, \alpha^\vee \rangle\leq 0, \quad \forall \alpha \in \Delta^{P^0},
\end{equation}
then $\lambda\in \Ccal_{\zipsf}$. This is slightly weaker than Theorem~\ref{thm-norm-low}, but it holds without any assumption on $(G,\mu)$. 
\end{rmk}

We can finally prove the statement in general.

\begin{theorem}\label{thmGSzip}
One has $\Ccal_{\GSsf}\subset \Ccal_{\zipsf}$.
\end{theorem}

\begin{proof}
Write $\Ccal_{\GSsf,I}=\Ccal_{\GSsf}$ and $\Ccal_{\GSsf,I_0}$ for the Griffiths--Schmid cones of $I$ and $I_0$, respectively. By Lemma~\ref{CGS-contained}, we have $\Ccal_{\GSsf,I_0}\subset \Ccal_{\zipsf}(G_{\FF_{q^r}},\mu_0)$. Since $w_{0,I} w_{0,I_0}\Ccal_{\GSsf,I_0}=\Ccal_{\GSsf,I}$, the result follows from Theorem~\ref{thm-GFqr-cone}.
\end{proof}

\section{Examples}\label{sec7}

\subsection{The case $\boldsymbol{G=U(2,1)}$ with $\boldsymbol{p}$ inert}
We consider the example of Picard modular surfaces. More precisely, let $\mathbf{E}/\QQ$ be a quadratic totally imaginary extension and $(\mathbf{V},\psi)$ a Hermitian space over $\mathbf{E}$ of dimension $3$ such that $\psi_\RR$ has signature $(2,1)$. There is a Shimura variety of dimension $2$ of PEL type attached to $\mathbf{G}=\GU(\mathbf{V},\psi)$. It parametrizes abelian varieties of dimension $3$ with a polarization, an action of $\Ocal_{\mathbf{E}}$ and a level structure. Let $p$ be a prime of good reduction, and let $X$ be the special fiber of the Kisin--Vasiu (canonical) integral model of the Shimura variety. By Equation~\eqref{zeta-Shimura}, we have a smooth, surjective morphism $\zeta \colon X\to \GZip^\mu$, where $G$ is the special fiber of a reductive $\ZZ_p$-model of $\mathbf{G}_{\QQ_p}$. In this section, we study the cones attached to $\GZip^\mu$ when $p$ is inert in $\mathbf{E}$. To simplify, we consider the case of a unitary group $G=\U(V,\psi)$ (the case of $G=\GU(V,\psi)$ is very similar).

Let $(V,\psi)$ be a $3$-dimensional vector space over $\FF_{q^2}$ endowed with a nondegenerate Hermitian form $\psi \colon V\times V\to \FF_{q^2}$ (in the context of Shimura varieties, take $q=p$). Write $\Gal(\FF_{q^2}/\FF_q)=\{\id,\sigma\}$. Choose a basis $\Bcal=(v_1,v_2,v_3)$ of $V$ where $\psi$ is given by the matrix 
\[
 J= \begin{pmatrix}
&&1\\&1&\\1&&\end{pmatrix}. 
\]
We define a reductive group $G$ by
\begin{equation}
G(R) = \left\{f\in \GL_{\FF_{q^2}}\left(V\otimes_{\FF_q} R\right) \relmiddle|  \psi_R(f(x),f(y))=\psi_R(x,y), \ \forall x,y\in V\otimes_{\FF_q} R \right\}
\end{equation}
for any $\FF_q$-algebra $R$.  There is an isomorphism $G_{\FF_{q^2}}\simeq \GL(V)\simeq \GL_{3,\FF_{q^2}}$. It is induced by the $\FF_{q^2}$-algebra isomorphism $\FF_{q^2}\otimes_{\FF_q} R\to R\times R$, $a\otimes x\mapsto (ax,\sigma(a)x)$ (where $\Gal(\FF_{q^2}/\FF_q)=\{\id, \sigma\}$). The corresponding action of $\sigma$ on $\GL_3(k)$ is given by $\sigma\cdot A = J \sigma({}^t \!A)^{-1}J$. Let $T$ denote the diagonal torus and $B$ the lower-triangular Borel subgroup of $G_k$ (note that $B$ and $T$ are defined over $\FF_q$). Identify $X^*(T)=\ZZ^3$ such that $(a_1,a_2,a_3)\in \ZZ^3$ corresponds to the character $\diag(x_1,x_2 ,x_3)\mapsto \prod_{i=1}^3 x_i^{a_i}$. The simple roots are $\Delta=\{e_1-e_{2}, e_2-e_3\}$, where $(e_1, e_2 ,e_3)$ is the canonical basis of $\ZZ^3$. Define a cocharacter $\mu  \colon  \GG_{\mathrm{m},k}\to G_{k}$ by $x\mapsto \diag(x,x,1)$ via the identification $G_{k}\simeq \GL_{3,k}$. Let $\Zcal_{\mu}=(G,P,L,Q,M)$ be the associated zip datum. We have $\Delta^P=\{e_2-e_3\}$. Note that the determinant $\det\colon \GL_{3,k}\to \GG_{\textrm{m},k}$ is an invertible section of the line bundle $\Vcal_I(p+1,p+1,p+1)$ on $\GZip^\mu$. Set $D\colonequals \ZZ(1,1,1)=X^*(G)$. 
We have $D \subset \Ccal_{\zipsf}$. Identify
\begin{equation}\label{ident}
\ZZ^3/D\simeq \ZZ^2, \quad (a_1,a_2,a_3)\longmapsto (a_1-a_3,a_2-a_3).
\end{equation}
Hence, subcones of $\ZZ^3$ containing $D$ correspond bijectively to subcones of $\ZZ^2$ via the bijection~\eqref{ident}. For a subcone $C$ of $\ZZ^3$ containing $D$ and a subcone $C'\subset \ZZ^2$, we write $C\leftrightarrow C'$ if they correspond via the bijection \eqref{ident}.

\begin{proposition}
Via this identification, we have
\begin{align*}
    X^*_{+,I}(T) & \leftrightarrow  \{ (a_1,a_2)\in \ZZ^2 \mid \  a_1\geq a_2\}, \\
    X_{-}^*(L) & \leftrightarrow  \NN(-1,-1), \\ 
    \Ccal_{\GSsf} & \leftrightarrow  \{ (a_1,a_2)\in X^*_{+,I}(T) \mid \  0\geq a_1 \}, \\
     \Ccal_{\zipsf} & \leftrightarrow  \{ (a_1,a_2)\in X^*_{+,I}(T) \mid \  (q-1)a_1+a_2\leq 0 \},    \\  
    \Ccal_{\pha} & \leftrightarrow  \{ (a_1,a_2)\in X^*_{+,I}(T) \mid \  qa_1-(q-1)a_2\geq 0 \textrm{ and } (q-1)a_1+a_2\leq 0\}, \\
    \Ccal_{\hwsf} & \leftrightarrow  \{ (a_1,a_2)\in X^*_{+,I}(T) \mid \  qa_1-(q-1)a_2\leq 0 \}, \\
    \Ccal_{\lwsf} & =  \Ccal_{\zipsf}.
\end{align*}
\end{proposition}

\begin{proof}
The cone $\Ccal_{\zipsf}$ was determined in \cite[Corollary 6.3.3]{Imai-Koskivirta-vector-bundles}. 
The rest is a straight-forward computation.
\end{proof}

This example is not of Hasse type since $P$ is not defined over $\FF_q$. As predicted by Proposition~\ref{prop-Hassetype-equivalent}, $C_{\pha, \RR_{\geq 0}}$ is not a neighborhood of $X^*_{-}(L)_{\reg}$ in $X_{+,I}^*(T)_{\RR_{\geq 0}}$. Condition~\ref{cond-commute} is satisfied, and we  indeed have $\Ccal_{\GSsf}\subset \Ccal_{\lwsf}$ (see Lemma~\ref{cor-CGS-Cond}). 
However, $\Ccal_{\GSsf}\subset \Ccal_{\hwsf}$ does not hold. For this group, Conjecture~\ref{conj-shimura} holds by \cite[Theorem 4.3.3]{Goldring-Koskivirta-divisibility}; \textit{i.e.}, we have $\Ccal(\overline{\FF}_p)=\Ccal_{\zipsf}$.

\subsection{The orthogonal group $\boldsymbol{\SO(2n+1)}$} \label{subsec-orthogonal}

We consider the case of odd orthogonal groups. This example arises in the theory of Shimura varieties of Hodge type attached to general spin groups $\GSpin(2n-1,2)$ ($n\geq 1$). This furnishes an interesting infinite family of examples of zip data of Hasse type (see Definition~\ref{def-Hasse-type-cochardatum}). To simplify, we only consider the case of odd special orthogonal groups $\SO(2n+1)$, which is completely similar. Assume $p>2$. Let $J$ be the symmetric square matrix of size $2n+1$ defined by
\begin{equation}
    J \colonequals \left(
    \begin{matrix}
    &&1 \\
    &\iddots& \\
    1&&
    \end{matrix}
    \right).
\end{equation}
Let $n\geq 1$, and let $G$ be the reductive, connected, algebraic group over $\FF_q$ defined by 
\begin{equation}
    G(R) \colonequals \{A\in \SL_{2n+1}(R) \ | \ {}^tA J A = J \}
\end{equation}
for all $\FF_q$-algebra $R$. Let $T$ be the maximal diagonal torus, which is given by matrices of the form $t=\diag(t_1, \dots ,t_n,1,t_n^{-1}, \dots ,t_1^{-1})$. Identify $X^*(T)\simeq \ZZ^n$ such that $(a_1, \dots ,a_n)\in \ZZ^n$ corresponds to $t\mapsto t_1^{a_1} \cdots t_n^{a_n}$. Let $e_1, \dots ,e_n$ be the canonical basis of $\ZZ^n$. Fix the Borel subgroup of lower-triangular matrices in $G$. The positive roots $\Phi^+$ and the simple roots $\Delta$ are, respectively, 
\begin{align}
    \Phi^+& \colonequals \{ e_i \pm e_j, \ 1\leq i< j \leq n\}\cup \{e_i, \ 1\leq i \leq n\},  \\
    \Delta & \colonequals \{e_1-e_2, \dots ,e_{n-1}-e_n,e_n\}.
\end{align}
The Weyl group identifies as the group of permutations $\sigma$ of $\{1, \dots ,2n+1\}$ satisfying $\sigma(i)+\sigma(2n+2-i)=2n+2$. In particular, we have $\sigma(n+1)=n+1$. Moreover, $\sigma$ is entirely determined by $\sigma(1), \dots ,\sigma(n)$. For $\sigma\in W$ such that $\sigma(i)=a_i$ for $i=1, \dots ,n$, write
$\sigma=[a_1 \  \dots  \ a_n]$. Hence, the identity element is $[1 \ 2 \ \dots  \ n]$, and the longest element is $w_0=[2n+1 \  2n\  \dots  \ n+2]$. The action of $w_0$ on $X^*(T)$ is given by $w_0\lambda =-\lambda$. Consider the cocharacter
\[\mu\colon z\longmapsto \diag(z,1, \dots ,1,z^{-1}).\]
Let $\Zcal_\mu\colonequals (G,P,L,Q,M)$ be the zip datum attached to $\mu$ (since $\mu$ is defined over $\FF_q$, we have $M=L$). For $n\geq 2$, one has 
\begin{equation}
    I=\Delta \setminus \{e_1-e_2\}, \quad \Delta^P=\{e_1-e_2\}
\end{equation}
(for $n=1$, one has $I=\emptyset$, $\Delta^P=\Delta=\{e_1\}$). The Levi $L$ is isomorphic to $\SO(2n-1)\times \GG_{\mathrm{m}}$. In particular, $w_{0,I}$ acts on $I$ by $w_{0,I}\alpha=-\alpha$. Since $T$ is $\FF_q$-split, one has $\sigma(\alpha)=\alpha=-w_{0,I}\alpha$ for all $\alpha\in I$. This shows that $(G,\mu)$ is of Hasse type. Put $z \colonequals w_{0,I}w_0 =[2n+1 \ 2 \  \dots  \ n]$. Then $(B,T,z)$ is a frame for $\Zcal_\mu$ (see Lemma~\ref{lem-framemu}). We determine the cones appearing in Diagram \eqref{conediag}.

\begin{proposition}\label{prop-SO}
For $n\geq 2$, we have
\allowdisplaybreaks
\begin{align*}
    X^*_{+,I}(T) & = \{ (a_1,\dots ,a_n)\in \ZZ^n \mid \  a_2\geq  \dots  \geq a_n \geq 0\}, \\
    X^*(L)_{-} &= \ZZ_{\leq 0}(1, 0,\dots ,0),  \\
    \Ccal_{\GSsf} &= \{(a_1,\dots ,a_n)\in \ZZ^n \in X^*_{+,I}(T) \mid \  a_1+a_2\leq 0 \},  \\
    \Ccal_{\pha} &= \{ (a_1,\dots ,a_n)\in X^*_{+,I}(T) \mid \   (q+1)a_1+(q-1)a_2\leq 0\}, \\
     \Ccal_{\zipsf} &= \Ccal_{\pha}, \\[-0.8em]
     \Ccal_{\hwsf} &= \Ccal_{\lwsf} =
     \left\{ (a_1,\dots ,a_n)\in X^*_{+,I}(T) \mid \  (q^{2n-2}-1) a_1 \leq (q-1)
     \sum\limits_{i=2}^n (q^{i-2}-q^{2n-1-i})  a_i\right\}.
\end{align*}
\end{proposition}

\begin{proof}
The equality $\Ccal_{\zipsf} = \Ccal_{\pha}$ follows from Theorem~\ref{main-thm-Hasse-type}. Since $P$ is defined over $\FF_q$, we have $\Ccal_{\hwsf}=\Ccal_{\lwsf}$. The only nontrivial computation is $\Ccal_{\hwsf}$. Since $T$ is split over $\FF_q$, we can use \cite[Section~3.6]{Koskivirta-automforms-GZip} (changing~$p$ to $q$). Put $\alpha=e_1-e_2$. Denote by $L_\alpha \subset L$ the centralizer in $L$ of $\alpha^\vee$ and by $I_\alpha\subset I$ its type. Then $\Ccal_{\hwsf}$ is the set of $\lambda\in X_{+,I}^*(T)$ satisfying
\begin{equation}\label{ineqfinal}
\sum_{w\in {}^{I_\alpha} W_{I}} q^{\ell(w)}  \langle w \lambda, \alpha^\vee \rangle\leq 0.
\end{equation}
We only carry out the case $n\geq 3$. The set ${}^{I_{\alpha}}W_I$ has cardinality $2(n-1)$. Any permutation $w\in {}^{I_{\alpha}}W_I$ is entirely determined by $w^{-1}(2)$, and it can be any integer satisfying $2\leq w^{-1}(2)\leq 2n$ different from $n+1$. Writing $w^{-1}(2)=i$, there are two cases to consider: $2\leq i\leq n$ and $n+2\leq i\leq 2n$. In the first case, the length of $w$ is $i-2$ and one has  $\langle  \lambda, w^{-1}\alpha^\vee \rangle = a_1 - a_i$ (where $\lambda=(a_1,\dots ,a_n)$). In the second case, the length of $w$ is $i-3$ and  $\langle \lambda, w^{-1}\alpha^\vee \rangle = a_1 + a_{2n+2-i}$. Hence we find that the sum in Equation~\eqref{ineqfinal} is equal to
\begin{equation}
\sum_{i=2}^n q^{i-2} (a_1- a_i) + \sum_{i=n+2}^{2n} q^{i-3} (a_1+ a_{2n+2-i}) = \frac{q^{2n-2}-1}{q-1} a_1 - \sum_{i=2}^n (q^{i-2}-q^{2n-1-i})  a_i.
\end{equation}
The result follows.
\end{proof}

As predicted by Theorem~\ref{main-thm-Hasse-type}, one sees that $\Ccal_{\pha}$ contains all cones of Proposition~\ref{prop-SO} (except of course $X_{+,I}^*(T)$). For example, assume that $\lambda \in \Ccal_{\hwsf}$. We find $\frac{q^{2n-2}-1}{q-1}a_1 \leq \sum_{i=2}^n (q^{i-2}-q^{2n-1-i})  a_i  \leq  (1-q^{2n-3}) a_2$, and hence $\frac{q^{2n-2}-1}{q^{2n-3}-1}a_1+(q-1)a_2\leq 0$. In particular, this implies $a_1 \leq 0$. 
Since $q+1\geq \frac{q^{2n-2}-1}{q^{2n-3}-1}$, we have $(q+1)a_1+(q-1)a_2\leq 0$. This shows $\Ccal_{\hwsf}\subset \Ccal_{\pha}$ (for $n=2$, one actually has $\Ccal_{\hwsf}=\Ccal_{\pha}$). In Figure~\ref{fig1}, we give a representation of the cones for $n=3$. We represent the intersections with the affine hyperplane $a_1=-(q-1)$. In other words, the weight $(-(q-1),x,y)$ appears as the point $(x,y)$.

\begin{figure}[ht]
\begin{center}
\footnotesize{
\begin{tikzpicture}[line cap=round,line join=round,x=.7cm,y=.7cm]
\clip(-2,-1) rectangle (12,10);
\draw [line width=1pt] (0,0)-- (9,9);
\draw [line width=1pt] (9,9)-- (9,0);
\draw [line width=1pt] (9,0)-- (0,0);
\draw [line width=1pt] (5,5)-- (5,0);
\draw [line width=1pt] (6.5,6.5)-- (8,0);
\draw [line width=1pt,dotted,domain=0:20.347000000000005] plot(\x,{(-0--9*\x)/9});
\draw [line width=1pt,dotted,domain=0:20.347000000000005] plot(\x,{(-0-0*\x)/9});
\draw (7.3,-0.1) node[anchor=north west] {$(b,0)$};
\draw (5.2,7.5) node[anchor=north west] {$(a,a)$};
\draw (2.2,6) node[anchor=north west] {$(q-1,q-1)$};
\draw (6.0448,9.9186) node[anchor=north west] {$(q+1,q+1)$};
\draw (8.8,-0.1) node[anchor=north west] {$(q+1,0)$};
\draw (4,-0.1) node[anchor=north west] {$(q-1,0)$};
\draw (-1.312,-0.1) node[anchor=north west] {\large $X^*(L)_-$};
\draw[color=black] (13,4.859198753441583) node {\large $X_{+,I}^{*}(T)$};
\draw[color=black] (3.5,1.6) node {\Large $\Ccal_{\GSsf}$};
\draw[color=black] (6.4,2.2) node {\Large $\Ccal_{\hwsf}$};
\draw[color=black] (7.8,6) node {\Large $\Ccal_{\pha}$};

\begin{scriptsize}
\draw [fill=black] (0,0) circle (2pt);
\draw [fill=black] (9,9) circle (2.5pt);
\draw [fill=black] (9,0) circle (2.5pt);
\draw [fill=black] (5,0) circle (2.5pt);
\draw [fill=black] (5,5) circle (2.5pt);
\draw [fill=black] (6.5,6.5) circle (2.5pt);
\draw [fill=black] (8,0) circle (2.5pt);
\end{scriptsize}
\end{tikzpicture}
}
\end{center}\caption{Representation of the cones for $\lambda \in \Ccal_{\hwsf}$, $n=3$,  where $a \colonequals \frac{q^4-1}{q^3+q^2-q-1}$ and $b \colonequals  \frac{q^4-1}{q^3-1}$ (hence we have $q-1<a<b<q+1$)}
\label{fig1}
\end{figure}

\renewcommand\thesection{\Alph{section}}
\setcounter{section}{0}

\section*{Appendix. Classification of Hasse-type zip data, \Large{\bf by Wushi Goldring}\footnote{Department of Mathematics, Stockholm University, Stockholm, SE-10691, Sweden\\
\textit{e-mail:} wushijig@gmail.com}}

\addcontentsline{toc}{section}{Appendix.  Classification of Hasse-type zip data, by Wushi Goldring}
\refstepcounter{section}

    This appendix  classifies Hasse-type pairs $(G,\mu)$, as defined in Definition~\ref{def-Hasse-type-cochardatum}; see Theorem~\ref{th-sigma-triv}. The componentwise-maximal ones are singled out in Corollary~\ref{cor-maximal}, while those arising from Shimura varieties (resp.\ Shimura varieties of Hodge and abelian type) are classified in Theorem~\ref{th-Hodge}. Proofs are given in Section~\ref{sec-app-proofs}.

\subsection*{Acknowledgments}
I am grateful to Arno Kret and Ludvig Olsson for enlightening discussions. I thank Jean-Stefan Koskivirta and the referee for their helpful comments.

\subsection{Hasse-type Dynkin triples}
Let $\Dfr$ be a Dynkin diagram, $\sigma \in \Aut(\Dfr)$ a diagram automorphism and $\Ifr \subset \Dfr$ a $\sigma$-stable sub-diagram. 
This appendix classifies such  Dynkin triples $(\Dfr, \Ifr, \sigma)$ satisfying the following condition.

\begin{condition}
\label{cond-hasse-type-root-data}
 The actions of $\sigma$ and the opposition involution $-w_{0,\Ifr}$ of $\Ifr$ on $\Ifr$ coincide.
\end{condition}

The case $\Ifr=\Dfr$ is allowed. If $\Ifr=\Dfr$ and $\sigma=1$, then Condition~\ref{cond-hasse-type-root-data} holds precisely when the opposition involution of $\Dfr$ is trivial: $-w_0:=-w_{0,\Dfr}=1$. The classification of such $\Dfr$ is recalled in Lemma~\ref{lem-opp-inv-triv}.

\subsection{Translation}
\label{sec-translation} In the setting of Proposition~\ref{prop-Hassetype-equivalent}, let $\Dfr$ denote the Dynkin diagram of the simple roots $\Delta$ associated to $(G,B,T)$, and let $\Ifr$ denote the Dynkin sub-diagram of the type $I \subset \Delta$ of the parabolic $P \supset B$. Then the triples $(\Dfr, \Ifr, \sigma)$ satisfying Condition~\ref{cond-hasse-type-root-data} are precisely those arising from Hasse-type zip data, as characterized by the root-data-theoretic condition~\eqref{item-root-data-hasse-type} of Proposition~\ref{prop-Hassetype-equivalent}.

\subsection{Classification}
After highlighting isolated vertices of $\Ifr$, the classification is stated in Theorem~\ref{th-sigma-triv}.

\begin{definition}
A vertex $v \in \Ifr$ is \emph{isolated} 
if its connected component in $\Ifr$ is $\{v\}$. Let $\Ifr^{\geq 2} \subset \Ifr$ be the sub-diagram consisting of all connected components with at least two vertices.
\end{definition}

That is, $\Ifr^{\geq 2}$ is the (possibly empty) sub-diagram with all isolated vertices removed.  

\begin{rmk}
\label{rmk-isolated}
An isolated vertex $v \in \Ifr$ is fixed by $w_{0,\Ifr}$. Hence $(\Dfr, \Ifr, \sigma)$ satisfies Condition~\ref{cond-hasse-type-root-data} if and only if $(\Dfr, \Ifr^{\geq 2}, \sigma)$ does and $\sigma$ fixes all isolated vertices of $\Ifr$. 
Thus the key triples $(\Dfr, \Ifr, \sigma)$ are those where $\Ifr=\Ifr^{\geq 2}$ contains no isolated vertices. 
\end{rmk}

\begin{theorem}
\label{th-sigma-triv}
A triple $(\Dfr, \Ifr, \sigma)$ satisfies Condition~\ref{cond-hasse-type-root-data} if and only if it satisfies the three conditions: 
\begin{alist}
\item 
\label{item-isolated}
All isolated vertices of\, $\Ifr$ are fixed by $\sigma$; 
\item If\, 
\label{item-sigma-stable}
$\Ifr \cap \Dfr_i \neq \emptyset$ for some connected component $\Dfr_i$ of\, $\Dfr$, then $\Dfr_i$ is $\sigma$-stable; and 
\item 
\label{item-table}
If\, $\Ifr^{\geq 2} \cap \Dfr_i \neq \emptyset$, then 
$(\Dfr_i, \Dfr_i \cap \Ifr^{\geq 2},\sigma)$ appears in Table~\ref{table-main}.
\end{alist}
\end{theorem}

\begin{table}[ht]
\caption{ Triples $(\Dfr, \Ifr^{\geq 2}, \sigma)$ satisfying Condition~\ref{cond-hasse-type-root-data} with $\Dfr$ connected}
\label{table-main}
\small
\addtolength{\tabcolsep}{-2pt}
\renewcommand{\arraystretch}{1.2}
\begin{center}
\begin{tabular}{|c|c|c|c|}

\hline & $\type(\Dfr)$ &     $\sigma \in \Aut(\Dfr)$ & $\Ifr^{\geq 2}$

\\ \hline
1 & $A_n$, $n \geq2 $ &  $-w_0$ &
\begin{tabular}{l}
 unique $\sigma$-stable $A_m$, some     \\
  $2 \leq m \leq n$, $m \equiv n \pmod 2$    
\end{tabular}

\\ \hline
2 & $B_n$, $n \geq 2$ & trivial & unique $B_m$, some $2 \leq m \leq n$   
\\ \hline
3 &$C_n$, $n \geq 2$ & trivial & unique $C_m$, some $2 \leq m \leq n$ 

\\  \hline
4 & $D_n$, $n \geq 4$  & trivial & unique $D_{2m}$, some $2 \leq m \leq n/2$ 
 \\ \hline
5 &  $D_n$, $n \geq 4$  & order $2$ & unique $D_{2m+1}$, some $2 \leq m \leq (n-1)/2$ 
 \\ \hline
6 & $G_2$ & trivial &
unique $G_2$   
\\ \hline
7 & $D_4$ & order $2$ &
 extremal $\sigma$-fixed point removed $\cong D_3$          

\\ \hline
8 & $F_4$ &  trivial &
 unique $B_2 \cong C_2$, $B_3$, $C_3$ or $F_4$.          

\\ \hline
9 & $E_6$ &  trivial &
 unique $D_4$          

\\ \hline
10 & $E_6$ &  $-w_0$ &
 unique $-w_0$-stable $A_3$, $A_5$ or $E_6$.         
\\ \hline
11 & $E_7$ &  trivial &
 unique  $D_4$, $D_6$ or $E_7$.         
\\ \hline
12 & $E_8$ &  trivial &
 unique  $D_4$, $D_6$, $E_7$ or $E_8$.         
\\ \hline

   \end{tabular}
\end{center}
\end{table}

\subsection{Special cases I: Componentwise-maximal triples}
\label{sec-special-max}

\begin{definition}
\label{def-maximal}
A pair $(\Dfr,\Ifr)$ is \emph{maximal} if $\card(\Ifr)=\card(\Dfr)-1$. 
It is   \emph{componentwise-maximal} if $\Ifr \subsetneqq \Dfr$ and, for every connected component $\Dfr_i$ of $\Dfr$, either  $\Dfr_i \cap \Ifr=\Dfr_i$ or $(\Dfr_i, \Dfr_i \cap \Ifr)$ is maximal. 
\end{definition}
A triple $(\Dfr, \Ifr, \sigma)$ is (componentwise-)maximal  if the underlying pair $(\Dfr, \Ifr)$ is.
\begin{rmk} If $(\Dfr,\Ifr, \sigma)$ arises from $(G,\mu,B,T)$ as in Section~\ref{sec-translation}, then $(\Dfr,\Ifr)$ is componentwise-maximal if and only if, for every nontrivial, minimal, normal, connected $k$-subgroup $G_i$ of $G$, the Levi subgroup $L_i=\cent(\mu) \cap G_i$ of $G_i$ is either all of $G_i$ or a proper, maximal Levi of $G_i$. 

\end{rmk}

A $\sigma$-orbit of a triple $(\Dfr, \Ifr, \sigma)$ is a triple $(\Dfr', \Ifr', \sigma')$ such that $\Dfr'$ is a $\sigma$-orbit of connected components of $\Dfr$, $\Ifr':=\Dfr' \cap \Ifr$ and $\sigma'=\sigma_{|\Dfr'} \in \Aut(\Dfr')$.
\begin{corollary}
\label{cor-maximal}
A componentwise-maximal triple $(\Dfr, \Ifr, \sigma)$ satisfies Condition~\ref{cond-hasse-type-root-data} if and only if every $\sigma$-orbit  $(\Dfr', \Ifr', \sigma')$ with $\Ifr'=\Dfr'$ has $\Dfr'$ connected, $\sigma'=1$ and is listed in Lemma~\ref{lem-opp-inv-triv}, while every  $\sigma$-orbit with $\Ifr' \subsetneqq \Dfr'$ appears in Table~\ref{table-max}.
\end{corollary}

\subsection{Special cases II: Hodge-, abelian- and Shimura-type triples}
\label{sec-special-hodge}
Let $\gx$ be a Shimura datum.
For every prime $p$ such that $\mathbf{G}_{\mathbb{Q}_p}$ is unramified, the process recalled in~Section~\ref{subsec-Shimura} produces\footnote{This does not require $\gx$ to be of Hodge or abelian type.} a connected, reductive $\FF_p$-group $G$ from $\mathbf{G}$ and a cocharacter $\mu \in X_*(G)$ from~$\mathbf{X}$.  
Then Section~\ref{sec-translation} associates a triple $(\Dfr, \Ifr, \sigma)$ to $(G,\mu)$. 
The datum $\gx$ is of Hodge type if there exists a symplectic embedding $\gx \hookrightarrow (\GSp(2g), \mathbf{X}_g)$ into a Siegel-type datum for some $g \geq 1$, where $\mathbf{X}_g$ is the Siegel double half-space. 

\begin{definition}
A triple $(\Dfr, \Ifr, \sigma)$ is of Shimura (resp.\ Hodge) type if it arises from a Shimura (resp.\ Hodge-type) datum $\gx$ and a prime $p$ by the process described above. 
\end{definition}

\begin{rmk}
Let $\mathbf{X}^{\ad}$ be the projection of $\mathbf{X}$ onto the adjoint group $\mathbf{G}^{\ad}$.
If $\gx$ is of abelian type, then by definition there exists a Hodge-type datum $(\mathbf{G}_1, \mathbf{X}_1)$ such that $\mathbf{G}^{\ad}=\mathbf{G}_1^{\ad}$ and $\mathbf{X}^{\ad}=\mathbf{X}_1^{\ad}$.
If $\mathbf{G}, \mathbf{G}_1$ are both unramified at $p$, the triples $(\Dfr, \Ifr, \sigma)$ associated to the two Shimura data at $p$ are naturally identified. Therefore, there is no point to define ``abelian-type triples'' $(\Dfr, \Ifr, \sigma)$, as they are just the Hodge-type ones.    
\end{rmk}

\begin{table}[ht]
\caption{Componentwise-maximal $\sigma$-orbits  $(\Dfr', \Ifr', \sigma')$ satisfying Condition~\ref{cond-hasse-type-root-data}}
\label{table-max}

\addtolength{\tabcolsep}{-4pt}
\renewcommand{\arraystretch}{1.2}
\begin{center}
\begin{tabular}{|c|c|c|c|c|}

\hline & $\type(\Dfr')$ &     $\sigma' \in \Aut(\Dfr')$ &
\begin{tabular}{cl}
   $\emptyset$  & if $\type(\Ifr')=\emptyset$ \\
   $( \type(\Ifr'), \alpha)$  & if $\Ifr'=\Dfr' \setminus \{\alpha\}$
\end{tabular}
& \begin{tabular}{c}
    Hodge   \\
     type? 
\end{tabular} 
\\ \hline
1 & $A_1^m$, $m \geq1 $ &  $m$-cycle & $\emptyset$      
& Yes

\\ \hline
2 & $A_2$ & trivial & $(A_1, \alpha_1)$, $(A_1, \alpha_2)$
& Yes

\\ \hline 
3 &
$A_3$ & trivial &  $(A_{1}^2, \alpha_2)$ 
& Yes

\\ \hline 
4 &
$B_n$, $n \geq 2$ & trivial &  $(B_{n-1}, \alpha_1)$ 
& Yes

\\ \hline
5 & $D_{2m+1}$, $m \geq 1$  & trivial & $(D_{2m}, \alpha_1)$ 
& Yes

\\ \hline 
6 &
$D_4$ & \begin{tabular}{c}
   any of the $3$ \\  involutions $\neq 1$  
\end{tabular}  &
\begin{tabular}{c}
     $(D_3, \alpha)$ \\
  $\alpha=$extremal $\sigma'$-fixed point     
\end{tabular}
& Yes            

\\ \hline
7 &
$D_{2m}$, $m \geq 3$  & unique involution $\neq 1$ &  $(D_{2m-1}, \alpha_1)$ 
& Yes

\\ \hline 
8 &
$C_n$, $n \geq 2$ & trivial &  $(C_{n-1}, \alpha_1)$
& No

\\ \hline 
9 &
$C_n$, $n \geq 3$ & trivial & 
$(A_1 \times C_{n-2}, \alpha_2)$
& No

\\ \hline 
10 &
 $D_{2m}$, $m \geq 2$  & trivial &  $(A_1 \times D_{2m-2}, \alpha_2)$ 
& No
\\ \hline 
11 &
$B_n$, $n \geq 3$ & trivial & $(A_1 \times B_{n-2}, \alpha_2)$ 
& No

 \\ \hline 
12 &
 $D_{2m+1}$, $m \geq 2$  & $-w_0$ & $(A_1 \times D_{2m-1}, \alpha_2)$ 
& No

 \\ \hline 
13 &
$G_2$ & trivial & $(A_1, \alpha_1)$, $(A_1, \alpha_2)$
   
& No

\\ \hline 
14 &
$F_4$ &  trivial &
$(B_3, \alpha_4)$, $(C_3, \alpha_1)$
  
& No

\\ \hline 
15 &
$E_6$ &  $-w_0$ &
 $(A_5, \alpha_2)$
& No

\\ \hline 
16 &
$E_7$ &  trivial &
  $(D_6, \alpha_1)$       & No
 
\\ \hline 
17 &
$E_8$ &  trivial &
$(E_7,\alpha_8)$       & No
 
\\ \hline
   \end{tabular}
\end{center}
\end{table}

Combining Deligne's classification of Shimura data, see \cite[Section~1.2.5]{Deligne-Shimura-varieties}, and their symplectic embeddings, see [\opcitn, Table~1.3.9, Sections~2.3.4, 2.3.5, 2.3.7 and 2.3.9, Propositions~2.3.6 and~2.3.10, and Table~2.3.8]  with Corollary~\ref{cor-maximal} gives the following. 

\begin{theorem}
\label{th-Hodge}
A triple $(\Dfr, \Ifr, \sigma)$ of either Shimura or Hodge type satisfies Condition~\ref{cond-hasse-type-root-data} if and only if $(\Dfr, \Ifr, \sigma)$ is componentwise-maximal and every $\sigma$-orbit $(\Dfr', \Ifr', \sigma')$ satisfies:
\begin{alist}
\item
\label{item-special-shimura}
If\, $\Ifr' \subsetneqq \Dfr'$, then $(\Dfr', \Ifr', \sigma')$ is one of Table~\ref{table-max}, entries 1--7.
 
\item 
\label{item-cpt-factors}
 If\, $\Ifr'=\Dfr'$, then $\sigma'=-w_{0,\Dfr'}$ and there exists another $\sigma$-orbit $(\Dfr'', \Ifr'', \sigma'')$ of\, $(\Dfr, \Ifr, \sigma)$ with $\Ifr'' \subsetneqq \Dfr''$ and $\type(\Dfr'')=\type(\Dfr')$.    
\end{alist}
\end{theorem}

In particular, the Shimura triples satisfying Condition~\ref{cond-hasse-type-root-data} are precisely the Hodge-type ones.

\begin{rmk}
\label{rmk-Siegel-threefold}
In \cite{Goldring-Koskivirta-global-sections-compositio}, it was shown that the cone conjecture (Conjecture~\ref{conj2} of the introduction, Conjecture~2.1.6 in \opcitn) holds when $\type(G)=C_2$ and the projection $\mu^{\ad}$ of $\mu$ onto $G^{\ad}$ is a multiple of a minuscule cocharacter. This includes the Siegel varieties associated to $\GSp(4)$. Under the coincidental isomorphism $B_2 \cong C_2$, this is Table~\ref{table-max}, entry 4, $n=2$, consistent with the Hodge-type classification (Theorem~\ref{th-Hodge}).
 \end{rmk}

\subsection{Proofs}
\label{sec-app-proofs}
The proofs of the general classification, Theorem~\ref{th-sigma-triv}, and the componentwise-maximal case, Corollary~\ref{cor-maximal}, are exercises in the \textit{Planches} of Bourbaki \cite[Chapitre~6, Planches~I-IX]{bourbaki-lie-4-6}. The proof of the Hodge-type classification is an exercise in Deligne's classification  of Shimura (resp.\ Hodge-type) data (see \cite[\textit{loc.~cit.} and Section~1.2.5]{Deligne-Shimura-varieties}). Consulting the \textit{Planches}, one finds the following. 

\begin{lemma}
\label{lem-opp-inv-triv}
A connected Dynkin diagram $\Dfr$ has trivial opposition involution $-w_{0,\Dfr}=1$ in $\Aut(\Dfr)$ if and only if\, $\type(\Dfr)=A_1$, $B_n$,  $C_n$, 
$D_{2n}$ $(n \geq 2)$, $G_2$, $F_4$, $E_7$ or $E_8$.
\end{lemma}

\begin{lemma}
\label{lem-components}  
If\, $(\Dfr, \Ifr, \sigma)$ satisfies Condition~\ref{cond-hasse-type-root-data}, then every connected component of\, $\Ifr$ is $\sigma$-stable.
\end{lemma}

\begin{proof}
    The parabolic subgroup $W_{\Ifr}$ of $W$ stabilizes connected components of $\Ifr$.
\end{proof}

\begin{corollary}
\label{cor-components}    
If a connected component $\Dfr_i$ of\, $\Dfr$ is not $\sigma$-stable, then $\Ifr \cap \Dfr_i =\emptyset$.
\end{corollary}

\begin{proof}[Proof of Theorem~\ref{th-sigma-triv}]
By Remark~\ref{rmk-isolated} on isolated vertices, it is equivalent to show that $(\Dfr, \Ifr^{\geq 2}, \sigma)$ satisfies Condition~\ref{cond-hasse-type-root-data} if and only if it satisfies Condition~\ref{cond-hasse-type-root-data}\eqref{item-sigma-stable},~\eqref{item-table}. 
A triple $(\Dfr, \Ifr, \sigma)$ satisfies Condition~\ref{cond-hasse-type-root-data} if and only if every $\sigma$-orbit does. By Corollary~\ref{cor-components}, it suffices to check that the $(\Dfr, \Ifr, \sigma)$ satisfying Condition~\ref{cond-hasse-type-root-data} with $\Dfr$ connected and no isolated vertices $\Ifr=\Ifr^{\geq 2}$ are those listed in Table~\ref{table-main}.

Assume $\sigma=1$. Then $-w_{0,\Ifr}=1$. By Lemma~\ref{lem-opp-inv-triv}, $\Ifr$ should contain none of the following:
\begin{alist}
    \item
    \label{item-no-type-A}
    a connected component of type $A_m$ with $m\geq 2$,
    \item
    \label{item-no-type-D}
    a connected component of type $D_{2k+1}$ with $k \geq 2$,
    \item
    \label{item-no-type-E6}
    a sub-diagram of type $E_6$.
\end{alist}
Since $\Dfr$ is connected, a disconnected sub-diagram without isolated vertices contains a type $A_m$ component with $m \geq 2$. 
By restriction~\eqref{item-no-type-A},  $\Ifr$ is connected. 
Thus restrictions~\eqref{item-no-type-A}--\eqref{item-no-type-D} establish Theorem~\ref{th-sigma-triv} when $\type(\Dfr) \neq E$.  
Type $E$ is handled the same way, except that, in addition, the unique sub-diagram of type $E_6$ is disqualified by Lemma~\ref{lem-opp-inv-triv}. 
The sub-diagrams of type $D_5$ in $E_6$ are excluded by~\eqref{item-no-type-D}. This proves Theorem~\ref{th-sigma-triv} when $\sigma=1$.

Assume $\sigma \neq 1$. Since $\Dfr$ is connected, $\type(\Dfr)=A_n$ ($n \geq 2$), $D_n$ ($n\geq 3$) or $E_6$ by Lemma~\ref{lem-opp-inv-triv}. 

Consider type $E_6$. Since $\sigma \neq 1$, $\sigma=-w_0$ is the opposition involution.  There are precisely six $-w_0$-stable sub-diagrams without isolated points, of types $A_2, A_3, A_2 \times A_2$, $D_4$, $A_5$ and $E_6$. For the unique $\Ifr$ with $\type(\Ifr)=D_4$, $\sigma$ acts nontrivially, while $-w_{0,\Ifr}=1$.  For the $\sigma$-stable $\Ifr$ of type $A_2$, $\sigma$ acts trivially, while $-w_{0,\Ifr} \neq 1$. So Condition~\ref{cond-hasse-type-root-data} fails for both.  By Lemma~\ref{lem-components}, it also fails for $A_2 \times A_2$.  The remaining three sub-diagrams $A_3, A_5$ and $E_6$ do satisfy Condition~\ref{cond-hasse-type-root-data}. This proves Theorem~\ref{th-sigma-triv} in type $E_6$.

In type $A$ with $\sigma \neq 1$, again $\sigma=-w_0$. So Theorem~\ref{th-sigma-triv} holds by Lemma~\ref{lem-components}.
In type $D$, $\sigma$ acts trivially on $\sigma$-stable, type $A$ sub-diagrams $\Afr \subset \Dfr$ with more than one point, while these have $-w_{0,\Afr} \neq 1$. On the other hand, $\sigma \neq 1$ will act nontrivially on a type $D$ sub-diagram, so the latter must have odd rank, meaning type $D_{2m+1}$  rather than $D_{2m}$. 
\end{proof}

\begin{proof}[Proof of Corollary~\ref{cor-maximal}]
A triple $(\Dfr, \Ifr, \sigma)$ is componentwise-maximal if and only if some $\sigma$-orbit is componentwise-maximal and every  $\sigma$-orbit $(\Dfr', \Ifr', \sigma')$ is either componentwise-maximal or has $\Ifr'=\Dfr'$. By Corollary~\ref{cor-components}, a $\sigma$-orbit $(\Dfr', \Ifr', \sigma')$ with $\Dfr'$ disconnected satisfies $\Ifr'=\emptyset$. Hence the only componentwise-maximal, disconnected $\sigma$-orbit satisfying Condition~\ref{cond-hasse-type-root-data} is Table~\ref{table-max}, entry 1.

By Theorem~\ref{th-sigma-triv}, it remains to check that the maximal $\sigma$-orbits $(\Dfr', \Ifr', \sigma')$ satisfying Condition~\ref{cond-hasse-type-root-data} with $\Dfr'$ connected of rank strictly greater than~$1$ are Table~\ref{table-max}, entries 2--16.   
By maximality, an isolated point of $\Ifr'$  is an extremity of $\Dfr'$. So $\Ifr'$ has at most three isolated points. Consider the  four cases:

\textit{Three isolated points in $\Ifr'$}. Then  $\type(\Dfr')=D_4$, $\type(\Ifr')=A_1^3$ and $\sigma'=1$. This is Table~\ref{table-max}, entry 9, $m=2$. 

\textit{Two isolated points in $\Ifr'$}. Then they are separated by a single vertex. Since they are both extremities, either $\rank(\Dfr')=3$ and $\type(\Ifr')=A_1^2$, or $\type(\Dfr')=D_n$, $n \geq 5$ and $\type (\Ifr')=A_1^2 \times A_{n-3}$. 
If $\rank (\Dfr')=3$, then Condition~\ref{cond-hasse-type-root-data} holds. This is Table~\ref{table-max}, entries 3 and 9--10 with $n=3$. 
The case $D_3$ is covered by the coincidental isomorphism $D_3 \cong A_3$.
For $n \geq 5$, Condition~\ref{cond-hasse-type-root-data} fails due to the $A_{n-3}$ factor.

\textit{A unique isolated point in $\Ifr'$}.
If $\rank \Dfr'=2$, then Condition~\ref{cond-hasse-type-root-data} holds unless $\type(\Dfr')=A_2$ and $\sigma \neq 1$. The cases $A_2,B_2, C_2, G_2$ are recorded in Table~\ref{table-max}, $n=2$, entries 2, 4, 8, 13.
As in Remark~\ref{rmk-Siegel-threefold}, $\type(\Dfr')=B_2$, $\Ifr'=\Dfr' \setminus \{\alpha_2\}$ (resp.\ $\type(\Dfr')=C_2$, $\Ifr'=\Dfr' \setminus \{\alpha_2\}$) occurs under entry 8: type $C_n$ (resp.\ entry 4: type $B_n$) via the coincidental isomorphism $C_2 \cong B_2$. 

When $\rank \Dfr' \geq 3$ and $\Ifr'$ admits a unique isolated point, $\Ifr'^{\geq 2} \neq \emptyset$. 
By the main classification, Theorem~\ref{th-sigma-triv}, $(\Dfr', \Ifr', \sigma')$ satisfies Condition~\ref{cond-hasse-type-root-data} if and only if it is one of Table~\ref{table-max}, entries 9--12. 

\textit{No isolated points in $\Ifr'$}. In this case, $\Ifr'=\Ifr'^{\geq 2}$. By Theorem~\ref{th-sigma-triv}, $(\Dfr', \Ifr', \sigma')$ satisfies Condition~\ref{cond-hasse-type-root-data} if and only if it is one of Table~\ref{table-max}, entries 4 ($n \geq 3$), 5 ($m \geq 2$), 6--8,  14--17.
\end{proof}
It remains to prove the Hodge-type classification, Theorem~\ref{th-Hodge}.
Recall, see \cite[Section~1.2.5]{Deligne-Shimura-varieties}, that a simple root $\alpha \in \Dfr$ is \emph{special} if $\alpha$ has multiplicity~$1$ in the decomposition of the highest root of the connected component $\Dfr_i$ of $\Dfr$ containing $\alpha$. Equivalently, $\alpha$ is special if and only if the corresponding fundamental coweight is minuscule.

\begin{lemma}
\label{lem-Hodge-simple}
Assume $(\Dfr', \Ifr', \sigma')$ is a $\sigma$-orbit of a Shimura-type triple $(\Dfr, \Ifr, \sigma)$. If\, $\Dfr'$ is connected and $\Ifr \subsetneqq \Dfr$, then $(\Dfr', \Ifr', \sigma')$ appears in Table~\ref{table-max}, entries 1--7 $($entry 1 occurring only with $m=1)$.
\end{lemma}

\begin{proof}
As explained in \cite[Section~1.2.5]{Deligne-Shimura-varieties}, Deligne's Griffiths transversality axiom for Shimura data (\opcitn, Equation~(2.1.1.1)) implies $\Dfr'\setminus \Ifr'= \{\alpha\}$ and  $\alpha$ is special. Table~\ref{table-max}, entries 8--17 are excluded since $\alpha$ is not special there.
\end{proof}

\begin{lemma}
\label{lem-Hodge-cpt-factors}
If a $\sigma$-orbit $(\Dfr', \Ifr', \sigma')$ of a Shimura-type triple $(\Dfr, \Ifr, \sigma)$ satisfies Condition~\ref{cond-hasse-type-root-data}, then item~\eqref{item-cpt-factors} of Theorem~\ref{th-Hodge} holds. 
\end{lemma}

\begin{proof}
Since $(\Dfr, \Ifr, \sigma)$ is of Hodge type, it arises from an $\fp$-group $G$ and $\mu \in X_*(G)$ as in Section~\ref{sec-translation}, while $(G,\mu)$ arises from a Hodge-type Shimura datum $\gx$ as in Section~\ref{sec-special-hodge}.  Assume $\Ifr'=\Dfr'$. By the main classification, Theorem~\ref{th-sigma-triv}, $\Dfr'$ is connected. By Condition~\ref{cond-hasse-type-root-data}, $\sigma'=-w_{0, \Dfr'}$. 

The root data of $\mathbf{G}_{\overline{\QQ}}$ and $G_k$ are naturally identified under specialization. 
Hence there exists a $\QQ$-simple factor $\mathbf{G}_i$ of $\mathbf{G}^{\ad}$ such that the Dynkin diagram of $\mathbf{G}_i$ admits a component  isomorphic to $\Dfr'$.  
By Deligne's  ``no compact factors over $\QQ$'' axiom \cite[Equation~(2.1.1.3)]{Deligne-Shimura-varieties},  there exists an $\RR$-simple factor $\mathbf{H}$ of $\mathbf{G}_{i, \RR}^{\ad}$ such that $\mathbf{H}(\RR)$ is not compact.  
By Deligne's polarization axiom [\opcitn, Equation~(2.1.1.2)], the $\RR$-simple factors of $\mathbf{G}_{\RR}^{\ad}$ are absolutely simple. 
Let $\mu_1 \in X_*(\mathbf{G})$ be a representative of the conjugacy class of cocharaters associated to $\XX$. 
The noncompactness of $\mathbf{H}(\RR)$ is equivalent to the nontriviality of the projection of $\mu_1$ onto $H_{\CC}$. 
In turn, the latter nontriviality corresponds to a $k$-simple factor $H$ of $G_k^{\ad}$ such that the projection of $\mu$ onto $H$ is nontrivial. 
Let $(\Dfr'', \Ifr'', \sigma'')$ be the $\sigma$-orbit of $\Dfr, \Ifr, \sigma)$ such that the Dynkin diagram of $H$ is a component~$\Dfr''$. By construction, we have $\Ifr'' \subsetneqq \Dfr''$ and $\type(\Dfr')=\type(\Dfr'')$ because they are both components of the Dynkin diagram of the $\QQ$-simple group $\mathbf{G}_i$.  
So $(\Dfr'', \Ifr'', \sigma'')$ satisfies Theorem~\ref{th-Hodge}\eqref{item-cpt-factors}.
\end{proof}

\begin{proof}[Proof of Theorem~\ref{th-Hodge}]
  By Lemmas~\ref{lem-Hodge-simple} and~\ref{lem-Hodge-cpt-factors}, every Hodge-type triple satisfies Theorem~\ref{th-Hodge}\eqref{item-special-shimura},~\eqref{item-cpt-factors}. 
  We explain why the converse follows from Deligne's classification \cite[Sections~2.3.4, 2.3.5, 2.3.7 and 2.3.9, Propositions~2.3.6 and~2.3.10, and Table~2.3.8]{Deligne-Shimura-varieties}. Assume $(\Dfr, \Ifr, \sigma)$ satisfies Theorem~\ref{th-Hodge}\eqref{item-special-shimura},~\eqref{item-cpt-factors}. As explained in the proof of Lemma~\ref{lem-Hodge-cpt-factors}, components $\Dfr_i$ with $\Ifr \cap \Dfr_i=\Dfr_i$ correspond to compact factors of $\mathbf{G}_{\RR}^{\ad}$.

If $(\mathbf{G}_1, \mathbf{X}_1)$, $(\mathbf{G}_2, \mathbf{X}_2)$ are Hodge-type Shimura data, then there exists a Hodge-type datum $\gx$ whose adjoint datum decomposes as $\mathbf{G}^{\ad}=\mathbf{G}_1^{\ad} \times \mathbf{G}_2^{\ad}$ and $\mathbf{X}^{\ad}=\mathbf{X}^{\ad}_1 \times \mathbf{X}_2^{\ad}$. 
In particular, the Dynkin triples of $\gx$ are disjoint unions of those of $(\mathbf{G}_1, \mathbf{X}_1)$, $(\mathbf{G}_2, \mathbf{X}_2)$.

Using this product construction, we may assume without loss of generality that all components $\Dfr_i$ of $\Dfr$ have the same type. Under this assumption, we exhibit a group $\mathbf{G}$ such that $\mathbf{G}^{\ad}$ is $\QQ$-simple and there exists a Hodge-type datum $\gx$ giving rise to $(\Dfr, \Ifr, \sigma)$. The $\mathbf{G}(\RR)$-conjugacy class $\mathbf{X}$ is determined by $\mu \in X_*(G)$. 

Let $d$ (resp.\ $d_{nc}, d_c$) be the number of components  of $\Dfr$ (resp.\ those $\Dfr_i$ with $\Dfr_i \cap \Ifr \subsetneqq \Dfr_i$, those with $\Dfr_i \cap \Ifr=\Dfr_i$). Let $F$ be a degree $d$ totally real extension of $\QQ$.
For each entry 1--7 of Table~\ref{table-max},  we specify 
\begin{alist}
\item a quasi-split $F$-group $\mathbf{G}_0^*$ associated to a totally real or totally imaginary quadratic $F$-algebra $K$; 

\item groups $\mathbf{G}_{0,v}$ over $F_v$ for all real places $v$ of $F$, such that $\mathbf{G}_{0,v}$ is compact for precisely $d_c$ real places;  
\item a prime $p$ unramified in $F$ with prescribed splitting behavior;
\item
\label{item-local-quasi-split}
for all primes $v$ of $F$ above $p$, the $F_v$-group $\mathbf{G}_{0,v}=\mathbf{G}_{0,v}^*$.
\end{alist}
One has $K \cong F \times F$ if and only if $\mathbf{G}_0^*$ is split.

In each of the cases below, a result of Kottwitz \cite[Proposition~2.6]{Kottwitz-elliptic-singular} implies that there exists an inner $F$-form $\mathbf{G}_0$ of $\mathbf{G}_0^*$ with the prescribed behavior at the archimedean places and those above $p$. 
In fact, \loccit shows a much stronger result; in particular, the group can be prescribed at all but finitely many places (often at all but one place). This is worked out in detail for certain orthogonal groups by Kret--Shin \cite[Section~8]{Kret-Shin-GSO}, and works similarly in the cases below. In all the cases below, the weight cocharacter $w\colon \gm \to \mathbf{G}_{\RR}$, see \cite[Section~1.1.11]{Deligne-Shimura-varieties}, is defined over $\QQ$, and
$\mathbf{G}=w(\gmq)\cdot \res_{F/\QQ}\mathbf{G_0}$ is a similitude group of the restriction of scalars $\res_{F/\QQ}\mathbf{G_0}$.

In some special cases, there is a more classical description of a Hodge-type $\gx$ giving rise to $(\Dfr, \Ifr, \sigma)$; see Remark~\ref{rmk-classical-cases}.

For entries $j=1,2,3$ of Table~\ref{table-max}, let $K/F$ be totally imaginary. Let $\mathbf{G}_0^*$ be the quasi-split, unitary $F$-group associated to $K/F$ of rank $j+1$ (an outer form of $\GL(j+1)$). For $j=2,3$, choose $p$ that splits completely in~$K$. For $j=1$, choose $p$ whose residual degrees $f_i$ relative $F$ (with $\sum f_i=d$) match the sizes of the cycles of $\sigma$ acting on $\Dfr$.
Let  $\mathbf{G}_0$ be an inner  $F$-form  such that the $d_c$ compact factors (resp.\ $d_{nc}$ noncompact factors)  satisfy $\mathbf{G}_{0,v} \cong U(j+1)=U(j+1,0)$  (resp.\  $\mathbf{G}_{0,v} \cong \U(1,1), \U(2,1), \U(2,2)$). 
By~\eqref{item-local-quasi-split}, $\mathbf{G}_{0,v} \cong \GL(n)_{F_v}$ is split (resp.\ a restriction of scalars $\res_{\FF_{p^{f_i}}/\FF_p})$ for all $v$ above $p$ when $j=2,3$ (resp.\ $j=1$).    
   
 In all three cases, $\mathbf{G}_{\RR}$ is a unitary similitude group $\textnormal{G}(\U(a_1,b_1) \times \cdots \times \U(a_d, b_d))$ with $(a_i,b_i)$ as above, where the single ``$G$'' outside the parentheses signifies that all factors have the same similitude.

For entry 4 (type $B_n$), $\mathbf{G}_0^*$ is a (necessarily) split spin group. Choose  $p$ which  splits completely in~$F$. Construct an inner form $\mathbf{G}_0$ such that $\mathbf{G}_{0,v} \cong \Spin(2n+1)_{\RR}$ at the $d_c$ compact real places, $\mathbf{G}_{0,v} \cong \Spin(2n-1,2)$ has signature $(2n-1,2)$ at the  $d_{nc}$ noncompact real places  and~\eqref{item-local-quasi-split} holds. 

Entry 5 is a hybrid of entries 2--3 and 4: Let $K$ and $p$ be as for entries 2--3. Let $\mathbf{G}_0$ be the (nonsplit) quasi-split $F$-form of $\Spin(4m+2)$ associated to $K/F$.  
Let $\mathbf{G}_0$ be an inner form such that $\mathbf{G}_{0,v} \cong \Spin(4m+2)$ is compact for $d_c$ real places (resp.\ $\mathbf{G}_{0,v} \cong \Spin(4m,2)$ for $d_{nc}$ real places) and~\eqref{item-local-quasi-split} holds. Since $p$ splits completely in $K$, $\mathbf{G}_{0,v}$ is $F_v$-split for $v$ above $p$.   

Entries 6--7 are of a different flavor because $-w_0=1$ and $\sigma' \neq 1$ there. Up to isomorphism, entry 6 is the same as entry 7 but with $m=2$. So consider entry $7$ extended to include $m=2$. Since $-w_0=1$ and $\sigma' \neq 1$, take $K/F$ totally real and nonsplit (see also Remark~\ref{rmk-inner-form-cpt}). Let $p$ be a prime which splits completely in $F$ and is totally inert along $F'/F$. Let $\mathbf{G}_0^*$ be the (nonsplit) quasi-split $F$-form of $\Spin(4m)$ associated to $K/F$. Let $\mathbf{G}_0$ be an $F$-inner form of $\mathbf{G}_0^*$ such that $\mathbf{G}_{0,v} \cong \Spin(4m)$ is compact (resp.\ $\mathbf{G}_{0,v} \cong \Spin(4m-2,2)$) for $d_c$ (resp.\ $d_{nc}$) real places and~\eqref{item-local-quasi-split} holds above $p$. Since the primes $v$ of $F$ above $p$ are inert in $K$, the Galois group $\gal(K/F)$ acts nontrivially on the Dynkin diagram of each $\mathbf{G}_{0,v}$ (of type $D_{2m}$). Hence $\sigma$ acts nontrivially on all components $\Dfr_i$ with $\Dfr_i \cap \Ifr \subsetneqq \Dfr_i$.       
\end{proof}

Let $K/F$ be as in the proof of Theorem~\ref{th-Hodge}. 

\begin{rmk}
 \label{rmk-classical-cases}  
 Assume $(\Dfr, \Ifr,\sigma)$ is a Hodge-type triple with $\Ifr \subsetneqq \Dfr$ and $\sigma$ acting transitively on the components of $\Dfr$. Then in Table~\ref{table-max}, entries 1--5 also arises from Hodge-type Shimura varieties that admit a more classical description. 
 The assumption implies that there are no compact factors and that $\Dfr$ is connected and $F=\QQ$ in entries 2--7.  As mentioned before, entries 6--7 are more complicated, even under the simplifying assumption, due to the role of the totally real quadratic extension $K/\QQ$.
 
 Entry 1 arises from Hilbert modular varieties associated to $F$. A special case of the construction in the proof of Theorem~\ref{th-Hodge} for entries $j=2,3$ is a unitary similitude group $\mathbf{G}=\GU(2,j-1)$ associated to a Hermitian form of signature $(2,j)$ for $K/\QQ$ an imaginary quadratic field. For $j=2$, the resulting Shimura varieties are often called Picard modular surfaces. Similarly, for entries 4--5, $\mathbf{G}$ may be taken to be the spin similitude group associated to a nondegenerate, symmetric bilinear form over $\QQ$, whose signature over $\RR$ is $(2n-1,2)$ (resp.\ $(4m,2)$).

 \end{rmk}
\begin{rmk}
    \label{rmk-inner-form-cpt}
For a Shimura datum $\gx$, the polarization axiom, see \cite[Equation~(2.1.1.2)]{Deligne-Shimura-varieties}, implies that $\mathbf{G}_{\RR}^{\ad}$ is an inner form of its compact form. This implies that $\gal(\CC/\RR)$ acts on $\Dfr$ by the opposition involution $-w_{0,\Dfr}$ [\opcitn, Section~2.3.4(b)]. This dictates that $K/F$ is totally imaginary in the construction for entries 2--3 and 5 (resp.\ totally real for entries 4, 6--7, with $K=F \times F$ split for entry 4). 

Entries 6--7 stand out in that $\gal(\CC/\RR)$ acts trivially on $\Dfr$, but $\sigma$ does not.
\end{rmk}

\phantomsection
\newcommand{\etalchar}[1]{$^{#1}$}
\providecommand{\bysame}{\leavevmode\hbox to3em{\hrulefill}\thinspace}


\begin{thebibliography}{EvdG09+++}

\bibitem[And23]{Andreatta-modp-period-maps}
F.~Andreatta, \emph{On two {${\rm mod}\,p$} period maps: {E}kedahl--{O}ort and
  fine {D}eligne--{L}usztig stratifications}, Math.\ Ann.\ \textbf{385} (2023),
no.~1-2, 511--550, \doi{10.1007/s00208-021-02356-7}.

\bibitem[AG05]{Andreatta-Goren-book}
F.~Andreatta and E.\,Z.~Goren, \emph{Hilbert Modular Forms: mod $p$ and $p$-Adic
  Aspects}, Mem.\ Amer.\ Math.\ Soc.\ \textbf{173} (2005), no.~819, 
\doi{10.1090/memo/0819}.

  
\bibitem[ABD{\etalchar{+}}11]{SGA3}
M.~Artin, J.-E.~Bertin, M.~Demazure, A.~Grothendieck, P.~Gabriel, M.~Raynaud, and J.-P.~Serre,  \emph{Sch\'emas en groupes (SGA 3). Tome III. Structure des sch\'emas en groupes r\'eductifs}, S\'eminaire de G\'eom\'etrie Alg\'ebrique du Bois Marie 1962--64, Revised and annotated edition of the 1970 French original, edited by P.~Gille and P.~Polo, Doc.\ Math. (Paris), vol.~7, Soc.\ Math.\ France,
Paris, 2011.

\bibitem[Bou68]{bourbaki-lie-4-6}
  N.~Bourbaki, \emph{\'El\'ements de math\'ematique. Fasc. XXXIV. Groupes et alg\`ebres de Lie. Chapitres IV--VI}, Actualit\'es Sci.\ Indust., No.~1337, Hermann, Paris, 1968. 


\bibitem[Del79]{Deligne-Shimura-varieties}
P.~Deligne, \emph{Vari\'et\'es de {S}himura: {I}nterpr\'etation modulaire, et
  techniques de construction de mod\`eles canoniques}, in: \emph{Automorphic forms,  representations and {$L$}-functions. Part~2} ({O}regon {S}tate {U}niv., {C}orvallis, Ore., 1977), pp.~247--289, {P}roc.\ {S}ympos.\ {P}ure {M}ath., vol.~33, Amer.\ Math.\ Soc., Providence, RI,
  1979, \doi{10.1090/pspum/033.2/546620}.

\bibitem[DK17]{Diamond-Kassaei-comp-minimal}
F.~Diamond and P.~Kassaei, \emph{Minimal weights of {H}ilbert modular forms in
  characteristic {$p$}}, Compos.\ Math.\ \textbf{153} (2017), no.~9, 1769--1778,
\doi{10.1112/S0010437X17007230}.

\bibitem[DK23]{Diamond-Kassaei-cone-minimal}
\bysame, \emph{The {C}one of {M}inimal {W}eights for {M}od $p$ {H}ilbert
  {M}odular {F}orms}, Int.\ Math.\ Res.\ Not.\ IMRN \textbf{2023}, no.~14, 12148--12171, 
\doi{10.1093/imrn/rnac121}.

\bibitem[EvdG09]{Ekedahl-Geer-EO}
T.~Ekedahl and G.~van~der Geer, \emph{Cycle Classes of the {E}-{O}
  Stratification on the Moduli of Abelian Varieties}, in: \emph{Algebra, Arithmetic and  Geometry: In Honor of Yu.\,I.~Manin. Vol.~1}, pp.~567--636, 
  Progr.\ Math., vol.~269, Birkh\"auser Boston, Boston, MA, 2009, 
\doi{10.1007/978-0-8176-4745-2_13}.

\bibitem[GK18]{Goldring-Koskivirta-global-sections-compositio}
W.~Goldring and J.\,S.~Koskivirta, \emph{Automorphic vector bundles with global
  sections on {$G$}-{\tt Zip}$^{{\mathcal{Z}}}$-schemes}, Compos.\ Math.\
  \textbf{154} (2018), no.~12, 2586--2605,
  \doi{10.1112/S0010437X18007467}.
  
\bibitem[GK19a]{Goldring-Koskivirta-Strata-Hasse}
\bysame, \emph{Strata {H}asse invariants, {H}ecke algebras and {G}alois
  representations}, Invent.\ Math.\ \textbf{217} (2019), no.~3, 887--984, 
\doi{10.1007/s00222-019-00882-5}.

\bibitem[GK19b]{Goldring-Koskivirta-zip-flags}
\bysame, \emph{Stratifications of Flag Spaces and Functoriality}, Int.\ Math.\ Res.\ Not.\ IMRN
\textbf{2019}, no.~12, 3646--3682, \doi{10.1093/imrn/rnx229}.

\bibitem[GK22a]{Goldring-Koskivirta-divisibility}
\bysame, \emph{Divisibility of mod $p$ automorphic forms and the cone
conjecture for certain {S}himura varieties of {H}odge-type}, preprint \arXiv{2211.16817} (2022). 

\bibitem[GK22b]{Goldring-Koskivirta-GS-cone}
\bysame, \emph{Griffiths-{S}chmid conditions for automorphic forms via
characteristic $p$}, preprint
\href{https://arxiv.org/abs/2211.16819}{{arXiv:\allowbreak 2211.16819}} (2022).


\bibitem[GS69]{Griffiths-Schmid-homogeneous-complex-manifolds}
P.~Griffiths and W.~Schmid, \emph{Locally homogeneous complex manifolds}, Acta.\
  Math.\ \textbf{123} (1969), 253--302, 
\doi{10.1007/BF02392390}.

\bibitem[IKY23]{imai-kato-youcis-prim-real}
N.~Imai, H.~Kato and A.~Youcis, \emph{The prismatic realization functor for
shimura varieties of abelian type}, preprint \arXiv{2310.08472} (2023).

\bibitem[IK21]{Imai-Koskivirta-vector-bundles}
N.~Imai and J.\,S.~Koskivirta, \emph{Automorphic vector bundles on the stack of
  $G$-zips}, Forum Math.\ Sigma \textbf{9} (2021), Paper No.~e37, 
\doi{10.1017/fms.2021.32}.

\bibitem[IK24]{Imai-Koskivirta-partial-Hasse}
\bysame, \emph{Partial {H}asse invariants for {S}himura varieties of
  {H}odge-type}, Adv.\ Math.\ \textbf{440} (2024), Paper No.~109518,
\doi{10.1016/j.aim.2024.109518}.

\bibitem[Jan03]{jantzen-representations}
J.\,C.~Jantzen, \emph{Representations of algebraic groups}, 2nd ed., Math.\ Surveys  Monogr., vol.~107, Amer.\ Math.\ Soc., Providence, RI,
  2003.

\bibitem[Kim18]{Kim-Rapoport-Zink-uniformization}
W.~Kim, \emph{Rapoport--{Z}ink uniformization of {S}himura varieties}, Forum
  Math.\ Sigma \textbf{6} (2018), Paper No.~e16,  
\doi{10.1017/fms.2018.18}.

\bibitem[Kis10]{Kisin-Hodge-Type-Shimura}
M.~Kisin, \emph{Integral models for {S}himura varieties of abelian type},
  J.~Amer.\ Math.\ Soc.\ \textbf{23} (2010), no.~4, 967--1012, 
\doi{10.1090/S0894-0347-10-00667-3}.

\bibitem[Kos19]{Koskivirta-automforms-GZip}
J.-S.~Koskivirta, \emph{Automorphic Forms on the Stack of {$G$}-Zips}, Results
  Math.\ \textbf{74} (2019), no.~3, Paper No.~91,  
\doi{10.1007/s00025-019-1021-z}.

\bibitem[KW18]{Koskivirta-Wedhorn-Hasse}
J.-S.~Koskivirta and T.~Wedhorn, \emph{Generalized {$\mu$}-ordinary {H}asse
  invariants}, J.~Algebra \textbf{502} (2018), 98--119, 
\doi{10.1016/j.jalgebra.2018.01.011}.

\bibitem[Kot86]{Kottwitz-elliptic-singular}
R.\,E.~Kottwitz, \emph{Stable trace formula: Elliptic singular terms}, Math.\ Ann.\
  \textbf{275} (1986), no.~3, 365--399, \doi{10.1007/BF01458611}.

\bibitem[Kot84]{Kottwitz-Shimura-twisted-orbital}
\bysame, \emph{Shimura varieties and twisted orbital integrals},
  Math.\ Ann.\ \textbf{269} (1984), no.~3, 287--300. 
  \doi{10.1007/BF01450697}.
  

\bibitem[KS24]{Kret-Shin-GSO}
A.~Kret and S.\,W.~Shin, \emph{Galois representations for even general special
  orthogonal groups}, J.~Inst.\ Math.\ Jussieu \textbf{23} (2024), no.~5,
  1959--2050, \doi{10.1017/S1474748023000427}.

\bibitem[LS18]{Lan-Stroh-stratifications-compactifications}
K.-W.~Lan and B.~Stroh, \emph{Compactifications of subschemes of integral models  of {S}himura varieties}, Forum Math.\ Sigma \textbf{6} (2018), Paper No.~e18, 
\doi{10.1017/fms.2018.20}.

\bibitem[MW04]{Moonen-Wedhorn-Discrete-Invariants}
B.~Moonen and T.~Wedhorn, \emph{Discrete invariants of varieties in positive
  characteristic}, Int.\ Math.\ Res.\ Not.\ IMRN \textbf{2004}, no.~72, 3855--3903, \doi{10.1155/S1073792804141263}.

\bibitem[PWZ11]{Pink-Wedhorn-Ziegler-zip-data}
R.~Pink, T.~Wedhorn and P.~Ziegler, \emph{Algebraic zip data}, Doc.\ Math.\
  \textbf{16} (2011), 253--300, 
\doi{10.4171/dm/332}.

\bibitem[PWZ15]{Pink-Wedhorn-Ziegler-F-Zips-additional-structure}
\bysame, \emph{${F}$-zips with additional structure}, Pacific J.~Math.\
  \textbf{274} (2015), no.~1, 183--236, 
\doi{10.2140/pjm.2015.274.183}.

\bibitem[RR85]{Ramanan-Ramanathan-projective-normality}
S.~Ramanan and A.~Ramanathan, \emph{Projective normality of flag varieties and
  {S}chubert varieties}, Invent.\ Math.\ \textbf{79} (1985), no.~2, 217--224,
\doi{10.1007/BF01388970}.

\bibitem[SYZ21]{Shen-Yu-Zhang-EKOR}
X.~Shen, C.-F.~Yu and C.~Zhang, \emph{E{KOR} strata for {S}himura varieties
  with parahoric level structure}, Duke Math.~J.\ \textbf{170} (2021), no.~14,
  3111--3236, 
\doi{10.1215/00127094-2021-0047}.

\bibitem[Spr98]{Springer-Linear-Algebraic-Groups-book}
T.\,A.~Springer, \emph{Linear Algebraic Groups}, 2nd ed., Progr.\ Math.,
  vol.~9, Birkh\"auser Boston, Boston, MA, 1998, 
\doi{10.1007/978-0-8176-4840-4}.

\bibitem[Vas99]{Vasiu-Preabelian-integral-canonical-models}
A.~Vasiu, \emph{Integral canonical models of {S}himura varieties of preabelian
  type}, Asian J.~Math.\ \textbf{3} (1999), no.~2, 401--518,
\doi{10.4310/AJM.1999.v3.n2.a8}.

\bibitem[Wed14]{Wedhorn-bruhat}
T.~Wedhorn, \emph{Bruhat strata and ${F}$-zips with additional structure},
  M\"unster J.~Math.\ \textbf{7} (2014), no.~2, 529--556, \doi{10.17879/58269760517}. 

\bibitem[WZ23]{Wedhorn-Ziegler-tautological}
T.~Wedhorn and P.~Ziegler, \emph{Tautological rings of {S}himura varieties and
  cycle classes of {E}kedahl--{O}ort strata}, Algebra Number Theory \textbf{17}
  (2023), no.~4, 923--980, \doi{10.2140/ant.2023.17.923}.

\bibitem[XZ19]{Xiao-Zhu-on-vector-valued}
L.~Xiao and X.~Zhu, \emph{On vector-valued twisted conjugation invariant
  functions on a group} (with an appendix by S.~Donkin), in: \emph{Representations of Reductive Groups}, pp.~361--425, Proc.\ Sympos.\
  Pure Math., vol.~101, Amer.\ Math.\ Soc., Providence, RI, 2019,  
\doi{10.1090/pspum/101/14}.

\bibitem[Zha18]{Zhang-EO-Hodge}
C.~Zhang, \emph{Ekedahl-{O}ort Strata for Good Reductions of {S}himura
  Varieties of {H}odge Type}, Canad.~J.\ Math.\ \textbf{70} (2018), no.~2,
  451--480,
\doi{10.4153/CJM-2017-020-5}.

\end{thebibliography}
\end{document}